\documentclass{amsart}

\usepackage{amsfonts,mathrsfs}
\usepackage{amsmath,amssymb,amsthm, amscd}
\usepackage{tikz}
\usetikzlibrary{arrows,calc,matrix,topaths,positioning,scopes,shapes,decorations,decorations.markings} 
\usepackage{dsfont}
\usepackage{epsfig}
\usepackage{tikz-cd}
\usepackage{hyperref}
\usepackage{fullpage}
\usepackage[foot]{amsaddr}
\usepackage[cp1251]{inputenc}

\makeatletter
\renewcommand{\email}[2][]{%
  \ifx\emails\@empty\relax\else{\g@addto@macro\emails{,\space}}\fi%
  \@ifnotempty{#1}{\g@addto@macro\emails{\textrm{(#1)}\space}}%
  \g@addto@macro\emails{#2}%
}
\makeatother

\newcommand{\sslash}{\mathbin{/\mkern-6mu/}}

\newcommand{\Specm}{\operatorname{Specm}}

\newcommand{\Image}{\operatorname{Im}}
\newcommand{\Map}{\operatorname{Map}}
\newcommand{\SL}{\operatorname{SL}}

\newcommand{\GL}{\operatorname{GL}}

\newcommand{\End}{\operatorname{End}}

\newcommand{\SO}{\operatorname{SO}}
\newcommand{\Hom}{\operatorname{Hom}}
\newcommand{\quotient}[2]{{\raisebox{.2em}{$#1$}\left/\raisebox{-.2em}{$#2$}\right.}}
\newcommand{\Rel}{\operatorname{Rel}}

\newcommand{\stated}{stated }
\newcommand{\Stated}{Stated }
\newcommand{\id}{id}
\DeclareMathOperator{\Ima}{Im}

\def\restriction#1#2{\mathchoice
              {\setbox1\hbox{${\displaystyle #1}_{\scriptstyle #2}$}
              \restrictionaux{#1}{#2}}
              {\setbox1\hbox{${\textstyle #1}_{\scriptstyle #2}$}
              \restrictionaux{#1}{#2}}
              {\setbox1\hbox{${\scriptstyle #1}_{\scriptscriptstyle #2}$}
              \restrictionaux{#1}{#2}}
              {\setbox1\hbox{${\scriptscriptstyle #1}_{\scriptscriptstyle #2}$}
              \restrictionaux{#1}{#2}}}
\def\restrictionaux#1#2{{#1\,\smash{\vrule height .8\ht1 depth .85\dp1}}_{\,#2}}

\begin{document}

\theoremstyle{plain}
\newtheorem{theorem}{Theorem}[section]
\newtheorem{main_theorem}[theorem]{Main Theorem}
\newtheorem{proposition}[theorem]{Proposition}
\newtheorem{corollary}[theorem]{Corollary}
\newtheorem{corollaire}[theorem]{Corollaire}
\newtheorem{lemma}[theorem]{Lemma}
\theoremstyle{definition}
\newtheorem{definition}[theorem]{Definition}
\newtheorem{Theorem-Definition}[theorem]{Theorem-Definition}
\theoremstyle{remark}
\newtheorem{remark}{Remark}
\newtheorem{example}{Example}
\newtheorem{notations}{Notations}

\sloppy

\title{Triangular decomposition of character varieties}

\author{Julien Korinman}
\address{ Institut Montpelli\'erain Alexander Grothendieck - UMR 5149 Universit\'e de Montpellier. Place Eug\'ene Bataillon, 34090 Montpellier France}
\email{julien.korinman@gmail.com}
\subjclass{$14$D$20$.}

\date{}


\begin{abstract} 
A marked surface is a compact oriented surface equipped with some pairwise disjoint arcs embedded in its boundary. In this paper, we extend the notion of character varieties to marked surfaces, in such a way that they have a nice behaviour for the operation of gluing two boundary arcs together.  These \stated character varieties are affine Poisson varieties which coincide with the Culler-Shalen character varieties when the surface is unmarked and are closely related to the Fock-Rosly and Alekseev-Kosmann-Malkin-Meinrenken constructions in the marked case. These Poisson varieties are the classical moduli spaces underlying  stated skein algebras and share similar properties. In particular, \stated character varieties admit triangular decompositions,  associated to  triangulations of the surface. We identify the Zariski tangent spaces of these varieties with some twisted groupoid cohomological groups and provide a generalization of Goldman's formula for the Poisson bracket of curve functions in terms of intersection form in homology. 
\end{abstract}

\maketitle

\tableofcontents

\section{Introduction}
\par \textbf{Moduli space of  $G$ flat structures and character varieties}
\vspace{3mm}
\par This paper deals with a generalization of the moduli of $G$-flat structures over a surface. Though our construction is completely algebraic, we start with a geometric (gauge theoretic) description that will help the reader to get a better intuition. 
Let us briefly review previous constructions, we refer to  \cite{LabourieCharVar, MarcheCours09, Sikora} for details. Consider a manifold $X$ and a Lie group $G$ and denote by $\mathcal{M}_G(X)$ the moduli space of  flat $G$ structures on $X$, \textit{i.e.} the set of isomorphism classes of pairs $(P,\nabla)$, where $P$ is a  $G$ principal bundle over $X$ and $\nabla$ a flat connection. Such moduli spaces arise as the phase spaces (solutions of the Euler-Lagrange equations) of $2$-dimensional Yang-Mills theory (\cite{AB, Wi1, Moore_2DYM}) and $3$-dimensional Chern-Simons theory \cite{Wi2, ASW}. Fixing a gauge, we identify the space of flat structures  with the spaces of $1$-forms $A\in \Omega^1(X, \mathfrak{g})$ such that $FA:=dA +\frac{1}{2}[A\wedge A] = 0$. Write $\Omega_F^1(X,\mathfrak{g}) \subset \Omega^1(X,\mathfrak{g})$ the subset of forms $A$ with $FA=0$. The group of automorphisms of $P$ identifies with the (gauge) group $\mathcal{G}$ of smooth maps $g: X\to G$ and the action writes $A^g=g^{-1}Ag +g^{-1}dg$. This gives a bijection 
$$ \mathcal{M}_G(X) \cong \quotient{\Omega_F^1(X, \mathfrak{g})}{\mathcal{G}}.$$
In order to endow $\mathcal{M}_G(X)$ with a geometric structure, Atiyah and Bott imposed in \cite{AB} some Sobolev regularity on both the $1$-forms $A$ and the gauge group elements $g$. This permits to endow the space $ \Omega^1(X, \mathfrak{g})$ with a structure of Banach space and $\Omega^1_F(X, \mathfrak{g})$ becomes a Banach sub-manifold. However, the action of the gauge group $\mathcal{G}$ is not principal so the quotient $\quotient{\Omega_F^1(X, \mathfrak{g})}{\mathcal{G}}$ does not inherits a manifold structure, rather a structure of stratified space. Instead, one can consider the subset $\Omega_F^1(X, \mathfrak{g})^0 \subset \Omega_F^1(X, \mathfrak{g})$ of principal orbits and define 
$$ \mathcal{M}^0_G(X) := \quotient{\Omega_F^1(X, \mathfrak{g})^0}{\mathcal{G}}\subset \mathcal{M}_G(X), $$
which becomes a Banach manifold.  For $A\in \Omega_F^1(X, \mathfrak{g})^0$, consider the (twisted) cochain complex  $(\mathrm{C}^{\bullet}(X, \mathfrak{g}), d_A)$, where $d_A \alpha = d \alpha + [A\wedge \alpha]$ ($FA=0$ implies $d_A^2=0$). A gauge group element $g\in \mathcal{G}$ induces an isomorphism between the complexes associated to $A$ and $A^g$, so its cohomology only depends on the gauge class $[A]$ up to canonical isomorphism. 
The tangent space over a point $A\in \Omega_F^1(X, \mathfrak{g})^0$ is naturally identified with the space $\mathrm{Z}^1_A(X, \mathfrak{g}) \cong T_A\Omega_F^1(X, \mathfrak{g})^0$ of cochains of this complex and the tangent space of $[A] \in  \mathcal{M}_G(X)$ is identified with the first cohomology group $\mathrm{H}^1_{[A]}(X, \mathfrak{g})\cong T_{[A]}\mathcal{M}_G(X)$.
In the particular case where $X= \Sigma$ is a closed oriented Riemann surface, Atiyah and Bott defined in \cite{AB} a symplectic form on $ \mathcal{M}^0_G(X)$ via the formula
$$ \omega^{AB}_{[A]} ([\alpha], [\beta]) = \int_{\Sigma} \left( \alpha \wedge \beta \right), \quad \alpha, \beta \in \mathrm{H}^1_A(\Sigma, \mathfrak{g}), $$
where $\left( \cdot, \cdot \right) : \mathfrak{g}^{\otimes 2} \to \mathbb{C}$ is a fixed Ad-invariant non-degenerate symmetric pairing on $\mathfrak{g}$ (for instance the Killing-form when $\mathfrak{g}$ is semi-simple).

\vspace{2mm}
\par Another way to give some geometric structure on the moduli space of $G$-flat structures is via character varieties. Suppose $X$ connected and fix an arbitrary basepoint $v_0\in X$. The holonomy of a flat connection along a closed curve $\gamma \subset X$ only depends, up to conjugacy, on the isotopy class of the curve. The so-called Riemann-Hilbert correspondence asserts that the holonomy map induces a bijection
$$ \mathrm{Hol} : \mathcal{M}_G(X) \cong \quotient{\Hom(\pi_1(X,v_0), G)}{G}, $$
with the space of conjugacy classes of representations $\rho : \pi_1(X,v_0)\to G$. When the group $G$  is algebraic reductive over $\mathbb{C}$ and $\pi_1(X, v_0)$ finitely generated, the space of representations $\mathcal{R}_G(X):=\Hom(\pi_1(X,v_0), G)$ has a natural structure of smooth affine scheme over $\mathbb{C}$ (possibly unreduced). Again the action of $G$ on  $\mathcal{R}_G(X)$ is not principal, so taking the naive quotient does only lead to a stratified space, however when $G$ is reductive, one can consider the algebraic quotient (sometimes called GIT quotient for Geometric Invariant Theory) named \textit{character variety}:
$$ \mathcal{X}_G(X):= \mathcal{R}_G(X) \sslash G.$$
The algebra of regular functions of $ \mathcal{X}_G(X)$ is defined as the subalgebra of $\mathbb{C}[\mathcal{R}_G(X)]$ of coinvariant vectors for the $\mathbb{C}[G]$ coaction. It is finitely generated but might have non trivial nilradical, so the character variety is not necessarily a variety; this latter feature does not occurs when $X$ is a surface and $G=\SL_n$ for instance. As a GIT quotient, the character variety is an affine (possibly singular) scheme which were proved to be normal in \cite[Theorem $11.1$]{Simpson_RepVarProjVar2} when $X$ is a closed surface of genus $g\geq 2$ and in \cite{Whang_RelCharVar} when $X$ is an open surface and $G=\SL_2$. The (set of closed points of the) character variety is very similar to the moduli space $\mathcal{M}_G(X)$: one has a surjective map $\mathcal{M}_G(X) \to  \mathcal{X}_G(X)$, which induces a bijection $\mathcal{M}_G(X)^0 \cong  \mathcal{X}^0_G(X)$, where $\mathcal{X}^0_G(X)$ denotes the subset of smooth points; so character variety can be thought as a good algebraic analogue for $\mathcal{M}_G(X)$. Character varieties where first introduced by Culler and Shalen in \cite{CullerShalenCharVar}, in a manner totally unrelated to gauge theory, where they played an essential role in the search of incompressible surfaces inside $3$-manifolds (see also \cite{FriedlKitayama_RepVarEssentialSurfaces} for a recent higher rank generalization). For $\gamma \subset X$ a simple closed curve and $f\in \mathbb{C}[G]^G$ a conjugacy invariant regular function on $G$, one can associate a so-called \textit{curve function} $f_{\gamma} \in \mathbb{C}[\mathcal{X}_G(X)]$ which sends a conjugacy class $[\rho]$ to $f(\rho(\gamma))$. Culler-Shalen original definition of character variety consists in defining $\mathbb{C}[\mathcal{X}_G(X)]$ as the subalgebra of $\mathbb{C}[\mathcal{R}_G(X)]$ generated by curve functions associated to the trace. When $G=\SL_2(\mathbb{C})$, by a theorem of Procesi, their definition coincides with the definition by GIT quotient, though they do not coincide for general $G$ (see \cite{LawtonSikora_VarChar}). In this paper, we only consider character varieties defined as GIT quotients.

  Let $\rho : \pi_1(X, v_0)\to G$ be a representation and denote by $\widetilde{X}$ a universal cover of $X$ with basepoint $\widetilde{v_0}$ over $v_0$. Consider the twisted complex of cochains $\mathrm{C}^{\bullet}(X, \mathrm{Ad}_{\rho}):= \Hom_{\mathbb{Z}[\pi_1(X)]}(\mathrm{C}_{\bullet}(\widetilde{X}, \mathbb{Z}), \mathfrak{g})$, where the fundamental group acts on the singular chains of $\widetilde{X}$ via deck transformations and on the Lie algebra $\mathfrak{g}$ via $\mathrm{Ad}\circ\rho$. An element $g\in G$ induces an isomorphism between the complexes  $\mathrm{C}^{\bullet}(X, \mathrm{Ad}_{\rho})$ and $\mathrm{C}^{\bullet}(X, \mathrm{Ad}_{g\rho g^{-1}})$ so its cohomology $\mathrm{H}^{\bullet}(X,\mathrm{Ad}_{\rho})$ only depends on the class $[\rho] \in \mathcal{X}_G(X)$ up to canonical isomorphism. A representation $\rho$ is said \textit{good} if it is irreducible and its stabilizer subgroup is the center of $G$.
  When $\rho$ is good, the Zariski tangent space of the character variety at $[\rho]$ identifies with
  $$ T_{[\rho]} \mathcal{X}_G(X)\cong \mathrm{H}^1(X,\mathrm{Ad}_{\rho}).$$
  When $X$ is a closed Riemannian manifold and $[\rho]$ a smooth point with $\mathrm{Hol}(A)=[\rho]$, using the integration pairing and Poincar\'e duality, one has a canonical isomorphism between $\mathrm{H}^1_{A}(X, \mathfrak{g})$ and $\mathrm{H}^1(X,\mathrm{Ad}_{\rho})$. Goldman noticed in \cite{Goldman_Symplectic} that under this isomorphism, and when $X=\Sigma$ is a closed Riemannian surface, the Atiyah-Bott symplectic pairing has a natural interpretation using the cup product in twisted cohomology, namely it writes as the composition
  $$\omega^{Goldman} : \mathrm{H}^1(\Sigma,\mathrm{Ad}_{\rho})\times \mathrm{H}^1(\Sigma,\mathrm{Ad}_{\rho})  \xrightarrow{\cup} \mathrm{H}^2(\Sigma,\mathrm{Ad}_{\rho}\otimes \mathrm{Ad}_{\rho}) \xrightarrow{\left( \cdot, \cdot \right)}  \mathrm{H}^2(\Sigma,\mathbb{C}) \xrightarrow{\int_{\Sigma}} \mathbb{C}.$$
  Based on prior formulas of Wolpert in Teichm\"uller spaces, Goldman deduced in \cite{Goldman86} an explicit formula for the Poisson bracket of two curve functions. 
  
  \begin{definition} A reductive complex algebraic Lie group $G$ will be called \textit{standard} if the algebra of regular functions of the character varieties associated to any finite type surface is generated by curve functions.
  \end{definition}
   When $G$ is standard, Goldman's formula in \cite{Goldman86} implies that the Atiyah-Bott-Goldman symplectic structure induces a Poisson structure on character varieties of closed surfaces (\textit{i.e.} implies that the Poisson bracket of two regular functions is a regular function).

  Using an argument of Procesi \cite{Procesi87}, it is proved in \cite{BrumfieldHilden, Procesi87, FlorentinoLawton, Sikora_generating_set} that the groups $(\mathbb{C}^*)^N$, $\GL_N(\mathbb{C})$, $\SL_N(\mathbb{C})$, $\mathrm{Sp}_N(\mathbb{C})$, $\mathrm{O}_N(\mathbb{C})$ and $\SO_{2N+1}(\mathbb{C})$ are standard. However, as A.Sikora pointed to us, the result in \cite{Sikora_CharVarSOn} suggests that $\SO_{2N}(\mathbb{C})$ should not be standard in our sense.

  More precisely, when $G$ is abelian this is trivial. When $G=\SL_2(\mathbb{C})$, the fact that $\mathbb{C}[\mathcal{X}_{\SL_2}(\Sigma)]$is generated by curve functions $\tau_{\gamma}$, where $\tau$ is the trace function, was proved by Brumfield and Hilden in \cite{BrumfieldHilden} using  Procesi's theorem in \cite{Procesi87}. Theorem $3$ in \cite{Sikora_generating_set} implies the same result for $G=\SL_N(\mathbb{C})$. Remark $4$ in \cite{Sikora_generating_set} implies that when $G=\GL_N(\mathbb{C})$, a generating set is given by curve functions of the form $\tau_{\gamma}$ and $\det^{-1}_{\gamma}$. When $G$ is orthogonal, symplectic or odd special orthogonal, one has an embedding $i: G \hookrightarrow \GL_N(\mathbb{C})$ defining a $G$-invariant function $f:= \tau \circ i$.  The fact that $\mathbb{C}[\mathcal{X}_G(\Sigma)]$ is generated by curve functions $f_{\gamma}$ was proved by Florentino and Lawton in Theorem $A.1$ of the three first arXiv versions of \cite{FlorentinoLawton} and appears in \cite[Theorem $5$]{Sikora_generating_set}.

 \vspace{2mm}
\par For a connected compact oriented surface $\Sigma$ with non-trivial boundary, Fock and Rosly defined in \cite{FockRosly} a Poisson structure on the smooth locus $\mathcal{M}^0_G(\Sigma)$ (see \cite{Audin_Survey_FR} for a survey and see Appendix \ref{sec_FR} for a detailed comparison with our present work). By choosing a conjugacy class $c_{\partial}\in \quotient{G}{G}$ for each boundary component $\partial \in \pi_0(\partial \Sigma)$ and considering the submanifold $\mathcal{R}_G(\Sigma,c) \subset \mathcal{R}_G(\Sigma)$ of representations $\rho: \pi_1(\Sigma)\to G$ sending a peripheral curve parallel to $\partial$ to an element of $c_{\partial}$, we get a manifold $\mathcal{M}_G(\Sigma, c) \subset \mathcal{M}_G(\Sigma)$ which is a symplectic leaf of the Fock-Rosly Poisson structure. At the same time and independently to the work in \cite{FockRosly}, Guruprasad, Huebschmann, Jeffrey and Weinstein defined in \cite{GHJW_ModSpacesParBd} a symplectic structure on $\mathcal{M}_G(\Sigma, c)$ by identifying a tangent space $T_{[\rho]} \mathcal{M}_G(\Sigma, c) $ with some subspace $\mathrm{H}^1_{par} \subset \mathrm{H}^1(\Sigma, \mathrm{Ad}_{\rho})$ of so-called parabolic cohomology classes and then defining a symplectic pairing on $\mathrm{H}^1_{par}$ by a non-trivial generalization of Goldman symplectic pairing.

 Eventually Roche and Szenes proved in \cite{RocheSzenes}  that Goldman's formula for curve functions defines an algebraic Poisson structure on $\mathcal{X}_{G}(\Sigma)$ (see also  \cite[Theorem $15$]{Lawton_PoissonGeomSL3} where the proof is done in the case $G=\SL_n$ but, as detailed in \cite[Comment $18$]{Lawton_PoissonGeomSL3}, can be generalized to arbitrary $G$).
  
 \vspace{2mm}
 \par Modern approaches to give to the moduli space of flat structures a geometric structure is via moduli stack, for which one can consider derived symplectic structures (see \cite{Toen_DerivedGeomQuantization} and reference therein for recent developments towards quantization), or D-modules \cite{GanGinzburg_AlmostComVar_Dmodules_CherednikAlg}. Eventually, related moduli spaces that we will not consider here are the so-called wild character varieties where we impose to the $1$-forms $A$ to have some degeneracy condition at some fixed points of the surface (\cite{Boalch_SymplecticWildCharVar}).
  
\vspace{3mm}
\par \textbf{Moduli spaces for marked surfaces}

\begin{definition} A \textit{marked surface} $\mathbf{\Sigma}=(\Sigma, \mathcal{A})$ is a compact oriented surface $\Sigma$ (possibly with boundary) with a finite set $\mathcal{A}=\{a_i\}_i$ of orientation-preserving immersions $a_i : [0,1] \hookrightarrow \partial \Sigma$, \textit{named boundary arcs}, whose restrictions to $(0,1)$ are embeddings and whose interiors are pairwise disjoint. 

 An \textit{embedding} $f:(\Sigma, \mathcal{A}) \to (\Sigma', \mathcal{A}')$ of marked surfaces is a orientation-preserving proper embedding $f:\Sigma \to \Sigma'$ so that for each boundary arc $a \in \mathcal{A}$ there exists $a' \in \mathcal{A}$ such that $f\circ a$ is the restriction of $a'$ to some subinterval of $[0,1]$. 
 Marked surfaces with embeddings form a category $\mathrm{MS}$ with monoidal structure given by disjoint union.  
\end{definition}

  By abuse of notations, we will often denote by the same letter the embedding $a_i$ and its image $a_i((0,1)) \subset \partial \Sigma$ and both call them boundary arcs. We will also abusively identify $\mathcal{A}$ with the disjoint union $\bigsqcup_i a_i((0,1)) \subset \partial \Sigma$ of open intervals. The main interest in considering marked surfaces is that they have a natural gluing operation. Let $\mathbf{\Sigma}=(\Sigma, \mathcal{A})$ be a marked surface and $a,b\in \mathcal{A}$ two boundary arcs. Set $\Sigma_{a\#b} := \quotient{\Sigma}{a(t) \sim b(1-t)}$ and $\mathcal{A}_{a\#b}:= \mathcal{A}\setminus a\cup b$. The marked surface $\mathbf{\Sigma}_{a\#b}=(\Sigma_{a\#b}, \mathcal{A}_{a\#b})$ is said obtained from $\mathbf{\Sigma}$ by gluing $a$ and $b$. 
     \vspace{2mm}
   \par
   
    Character varieties admit deformation quantizations named skein algebras. More precisely, in the case $G=\SL_2(\mathbb{C})$, the Kauffman-bracket skein algebra $\mathcal{S}_q(\Sigma)$ is non-commutative (unital associative algebra) depending on a deformation parameter $q$. Setting $q=\exp(\hbar)$, where $\hbar$ is a formal parameter, the reduction modulo $\hbar$ of  $\mathcal{S}_q(\Sigma)$ (\textit{i.e.} the commutative algebra $\mathcal{S}_{+1}(\Sigma)$) is isomorphic to the algebra regular functions of $\mathcal{X}_{\SL_2}(\Sigma)$ (\cite{Bullock, PS00, ChaMa}). Write $\star_{\hbar}$ the product in $\mathcal{S}_q(\Sigma)$. A Poisson structure on $\mathcal{S}_{+1}(\Sigma)$ is defined by the standard formula 
    $$f\star_{\hbar} g-g\star_{\hbar} f \equiv \hbar \{ f,g\} \pmod{\hbar^2}.$$
A theorem of Turaev \cite{Turaev91} shows that this Poisson structure coincides with the Atiyah-Bott-Goldman Poisson structure on $\mathcal{X}_{\SL_2}(\Sigma)$.

 A recent construction of Bonahon-Wong \cite{BonahonWongqTrace} and L\^e \cite{LeStatedSkein} extends Kauffman-bracket skein algebras to marked surfaces under the name stated skein algebras. The motivation behind this generalization is the good behaviour for the gluing operation which permits to define triangular decompositions of skein algebras. The original motivation for the present paper was to discover what is the (Poisson) moduli space behind stated skein algebras; this goal is fully achieved in the joint paper \cite{KojuQuesneyClassicalShadows} in collaboration with Quesney based on the present work (see also \cite{CostantinoLe19, KojuPresentationSSkein} for two alternative independent proofs which does not consider the Poisson structure). As a result, we will obtain triangular decompositions of character varieties: a geometric tool whose interest goes beyond the study of quantization. This relationship permitted to deduce from the geometric study made in the present paper some classification theorems about the representations of stated skein algebras at roots of unity in \cite{KojuMCGRepQT, KojuKaruo_RepRSSkein}.

  \vspace{2mm}
   \par We first sketch our construction in the gauge theoretical context, where the idea is very simple;  this will help the reader to get some geometric intuition. Fix a marked surface $\mathbf{\Sigma}=(\Sigma, \mathcal{A})$ and consider the subset
   $$ \Omega_F^1(\mathbf{\Sigma}, \mathfrak{g}) := \{ A \in \Omega^1(\Sigma, \mathfrak{g}) | FA=0, A_{| \mathcal{A}} = 0\} \subset  \Omega_F^1(\Sigma, \mathfrak{g}) $$
   of flat $1$-forms whose restrictions to the interior of the boundary arcs vanish and the subset
   $$ \mathcal{G}_{\mathbf{\Sigma}} = \{ g : \Sigma \to G | g_{|\mathcal{A}} = e \} \subset \mathcal{G}_{\Sigma} $$ 
   of smooth maps whose restrictions to the interior of the boundary arcs are constant with value the neutral element $e\in G$. Define the relative moduli space
   $$ \mathcal{M}_G(\mathbf{\Sigma}) := \quotient{ \Omega_F^1(\mathbf{\Sigma}, \mathfrak{g}) }{ \mathcal{G}_{\mathbf{\Sigma}}}.$$
  For unmarked surfaces (where $\mathcal{A}=\emptyset$) one recover our previous definition. To understand the role of the marking, consider the case where  $G$ is an abelian group. The classical moduli space identifies with the cohomology group $\mathcal{X}_G(\Sigma)\cong \mathrm{H}^1(\Sigma; G)$ whereas the relative moduli space identifies with the relative cohomology group $\mathcal{X}_G(\Sigma, \mathcal{A}) \cong \mathrm{H}^1(\Sigma, \mathcal{A}; G)$.
   In addition to the obvious isomorphism $\mathcal{M}_G(\mathbf{\Sigma}_1\bigsqcup \mathbf{\Sigma}_2) \cong  \mathcal{M}_G(\mathbf{\Sigma}_1) \times  \mathcal{M}_G(\mathbf{\Sigma}_2)$, one has a gluing map $$\pi_{a\#b} : \mathcal{M}_G(\mathbf{\Sigma}) \to \mathcal{M}_G(\mathbf{\Sigma}_{a\#b})$$ induced by the projection $\Sigma\to \Sigma_{a\#b}$. Let $c$ denote the common image of $a$ and $b$ in $\Sigma_{a\#b}$ by this projection. Since $c$ is contractile (it is an open arc), a $1$-form in  $\Omega_F^1(\mathbf{\Sigma}_{a\#b}, \mathfrak{g})$  is always gauge equivalent to a $1$-form whose restriction to $c$ vanishes. This implies that the gluing map $\pi_{a\#b}$ is surjective. Note that if $a,b,c,d$ are four distinct boundary arcs, then $\pi_{a\#b}\circ \pi_{c\#d}= \pi_{c\#d}\circ \pi_{a\#b}$.

   Here is an interesting consequence. Call \textit{triangle} and denote by $\mathbb{T}$ the marked surface made of a disc with three boundary arcs. A marked surface is \textit{triangulable} if it can be obtained from a disjoint union of triangles by gluing some pairs of boundary arcs. A \textit{triangulation} $\Delta$ is the data of these disjoint union of triangles, named the faces and whose set is $F(\Delta)$,  together with the pairs of glued arcs. The images in $\Sigma$ of these boundary arcs are called edges and their set is denoted by $\mathcal{E}(\Delta)$. Composing the gluing morphisms together, for any triangulate marked surface $(\mathbf{\Sigma}, \Delta)$, one gets a surjective morphism
   $$ \pi^{\Delta} : \mathcal{M}_G(\mathbf{\Sigma}) \twoheadrightarrow \prod_{\mathbb{T} \in F(\Delta)} \mathcal{M}_G(\mathbb{T}).$$
   We can actually characterize the kernel of this map. Call \textit{bigon}, and denote by $\mathbb{B}$, the marked surface made of a disc with two boundary arcs, say $a_L$ and $a_R$. For $g\in G$, consider a $1$-form $A(g)\in \Omega^1_F(\mathbb{B}, \mathfrak{g})$ on the disc whose holonomy along an arc joining $a_L$ to $a_R$ is $g$. This gives a bijection $G \cong \mathcal{M}_G(\mathbb{B})$ sending $g$ to $[A(g)]$. Now consider a marked surface $\mathbf{\Sigma}$ and two boundary arcs $a$ and $b$. Given $A \in \Omega_F^1(\mathbf{\Sigma}, \mathfrak{g})$ and $g\in G$, one can consider the $1$-form $A(g)\cup A$ on $\mathbb{B} \cup \Sigma$. When gluing the disc to $\Sigma$ by identifying $a_R$ with $a$, one get a surface which retracts to $\Sigma$. Using the retraction, one obtains from $A(g)\cup A$ a $1$-form $g\cdot A \in \Omega^1_F(\mathbf{\Sigma}, \mathfrak{g})$ and it is clear that we get a left group action $$\nabla_a^L : G\times \mathcal{M}_G(\mathbf{\Sigma}) \to \mathcal{M}_G(\mathbf{\Sigma}) $$ sending $(g, [A])$ to $[g\cdot A]=: g\cdot [A]$. In the same manner, by gluing $a_R$ to $b$, one gets a right group action $\nabla_b^R:  \mathcal{M}_G(\mathbf{\Sigma})  \times G \to \mathcal{M}_G(\mathbf{\Sigma}) $. These additional left/right group actions on moduli spaces are essential features in this paper and are probably the main original ingredient in our approach to character varieties. One can completely describe the moduli space associated to $\mathbf{\Sigma}_{a\#b}$ from the moduli space of $\mathbf{\Sigma}$ together with its $G$ left and right actions associated to $a$ and $b$ via the equivalence:
    
   \begin{equation}\label{eq_RightExactSeq}
   \mathcal{M}_G(\mathbf{\Sigma}_{a\#b}) \cong \quotient{\mathcal{M}_G(\mathbf{\Sigma})}{\left( g\cdot [\rho] = [\rho] \cdot g, g\in G \right)}.
   \end{equation}
 As a consequence, for a triangulated marked surface $(\mathbf{\Sigma}, \Delta)$, the space $\prod_{\mathbb{T} \in F(\Delta)}\mathcal{M}_G(\mathbb{T})$ acquires  a structure of $G^{\mathring{\mathcal{E}}(\Delta)}$- bimodule, where $\mathring{\mathcal{E}}(\Delta)$ denotes the set of inner edges of the triangulation, and one has
 \begin{equation}\label{eq_triangulation_modulispace}
  \mathcal{M}_G(\mathbf{\Sigma}) \cong \quotient{\prod_{\mathbb{T} \in F(\Delta)}\mathcal{M}_G(\mathbb{T})}{\left( g\cdot x - x \cdot g, g\in G^{\mathring{\mathcal{E}}(\Delta)}\right)}.
  \end{equation}
 So the  moduli space  $\mathcal{M}_G(\mathbf{\Sigma})$ is completely described by the moduli space of the triangle $\mathcal{M}_G(\mathbb{T})$ with its $G^{\mathcal{A}}$ left/right actions, together with the combinatorial data of the triangulation.

\vspace{5mm}
\par \textbf{Main results of the paper}
\vspace{3mm}
\par Fix $\mathbf{\Sigma}=(\Sigma, \mathcal{A})$ a marked surface and $G$ a standard Lie group. In order to work in the algebro-geometric context, we will replace the moduli space of $1$-forms by an algebraic space of representations. Instead of the fundamental group, we need to consider the fundamental groupoid $\Pi_1(\Sigma)$ whose objects are points in $\Sigma$ and morphisms $\alpha: v_1 \to v_2$ are homotopy classes of continuous paths $c_{\alpha}: [0,1] \to \Sigma$ such that $c(0)=v_1$ and $c(1)=v_2$. We will write $v_1=s(\alpha)$ (source point) and $v_2=t(\alpha)$ (target point). A path is called trivial if it is the homotopy class of a path which is either trivial or contained in a boundary arc. 
The holonomy map induces a bijection between the space $\Omega^1_F(\mathbf{\Sigma}, \mathfrak{g})$ and the representation space
$$ \mathcal{R}_G(\mathbf{\Sigma}) = \{ \rho : \Pi_1(\Sigma) \to G | \rho(\alpha)=e \mbox{ for all trivial path }\alpha\in \Pi_1(\mathcal{A}) \}.$$
In Section \ref{sec_definitions}, we will define an (infinitely generated) algebra $\mathbb{C}[\mathcal{R}_G(\mathbf{\Sigma})]$ whose maximal spectrum is $\mathcal{R}_G(\mathbf{\Sigma})$, so $\mathcal{R}_G(\mathbf{\Sigma})$ is the set of a closed points of an affine scheme (abusively denoted by the same symbol) over $\mathbb{C}$. In order to get an algebraic group action, the gauge group will be replaced by the algebraic gauge group $\mathcal{G}_{\mathbf{\Sigma}}$ of maps $g:\Sigma \to G$ whose restriction to $\mathcal{A}$ is constant equal to the neutral element $e\in G$. The gauge group action is given by 
$$ g\cdot \rho (\alpha) := g(s(\alpha)) \rho(\alpha) g(t(\alpha))^{-1}, \mbox{ for }\rho \in \mathcal{R}_G(\mathbf{\Sigma}), g\in \mathcal{G}_{\mathbf{\Sigma}}, \alpha\in \Pi_1(\Sigma). $$
The \textit{\stated character variety} will then be defined as
$$ \mathcal{X}_G(\mathbf{\Sigma}) := \mathcal{R}_G(\mathbf{\Sigma}) \sslash \mathcal{G}_{\mathbf{\Sigma}}.$$
Its main properties are summarized in the 

\begin{theorem}\label{theorem1} 
\begin{enumerate}

\item The \stated character variety $\mathcal{X}_G(\mathbf{\Sigma})$ is an affine Poisson variety. The Poisson structure depends on the choice of an $Ad$-invariant non-degenerate symmetric pairing $\left(\cdot, \cdot \right): \mathfrak{g}^{\otimes 2} \to \mathbb{C}$  and on a choice of orientations of the boundary arcs of $\mathbf{\Sigma}$. When $G$ is abelian, the Poisson structure does not depend on the orientations.

\item When $\mathbf{\Sigma}=(\Sigma, \emptyset)$ is unmarked, the \stated character variety is canonically isomorphic to the traditional (Culler-Shalen) one equipped with its Goldman's Poisson bracket. 

\item When $\Sigma$ is connected of genus $g$ and the marking non-empty (so $\partial \Sigma \neq \emptyset$), the \stated character variety $\mathcal{X}_G(\mathbf{\Sigma})$ is isomorphic to $G^d$ where 
$$d:=\dim \left( \mathrm{H}_1(\Sigma, \mathcal{A} ; \mathbb{C}) \right)=2g-2+|\mathcal{A}|+|\pi_0(\partial \Sigma)|.$$
As a Poisson variety, $\mathcal{X}_G(\mathbf{\Sigma})$ is isomorphic to the Fock-Rosly moduli space associated to a graph $\Gamma \subset \Sigma$ on which $\Sigma$ retracts by deformation.

\item If $a$ is a boundary arc of $\mathbf{\Sigma}$, the algebra of regular functions of the \stated character variety has both a left and right co-module structure on $\mathbb{C}[G]$, denoted $\Delta_a^L : \mathbb{C}[\mathcal{X}_G(\mathbf{\Sigma})] \rightarrow \mathbb{C}[G]\otimes \mathbb{C}[\mathcal{X}_G(\mathbf{\Sigma})]$ and $\Delta_a^R : \mathbb{C}[\mathcal{X}_G(\mathbf{\Sigma})] \rightarrow  \mathbb{C}[\mathcal{X}_G(\mathbf{\Sigma})]\otimes \mathbb{C}[G]$. If $a$ and $b$ are two boundary arcs with the same orientation, there exists a Poisson embedding $i_{a\#b}$ lying in the following exact sequence: 

$$ 0 \rightarrow \mathbb{C}[\mathcal{X}_G(\mathbf{\Sigma}_{a\#b})] \xrightarrow{i_{a\#b}} \mathbb{C}[\mathcal{X}_G(\mathbf{\Sigma})]  \xrightarrow{ \Delta_a^L - \sigma \circ \Delta_b^R} 
 \mathbb{C}[G] \otimes \mathbb{C}[\mathcal{X}_G(\mathbf{\Sigma})],  
$$
where $\sigma(x\otimes y) = y\otimes x$. In other words, one has $ \mathbb{C}[\mathcal{X}_G(\mathbf{\Sigma}_{a\#b})]=\mathrm{coHH}^0\left( \mathbb{C}[G], {}_a \mathbb{C}[\mathcal{X}_G(\mathbf{\Sigma})]_b \right)$ (see  Subsection $2.5$ for this notation). Moreover the gluing operation is co-associative in the sense that if $a,b,c,d$ are four boundary arcs, one has $i_{a\#b}\circ i_{c\#d}=i_{c\#d}\circ i_{a\#b}$.
\item Let $a,b$ be two boundary arcs of $\mathbf{\Sigma}$ and $\mathbf{\Sigma}_{a\circledast b}$ the marked surface obtained from $\mathbf{\Sigma}\sqcup \mathbb{T}$ by gluing $a$ to one edge of the triangle $\mathbb{T}$ and gluing $b$ to another edge. Then $\mathcal{X}_G(\mathbf{\Sigma}_{a \circledast b})$ is isomorphic to the Alekseev-Malkin's fusion of $\mathcal{X}_G(\mathbf{\Sigma})$ (see Section \ref{sec_fusion} for definitions). Moreover when $\mathbf{\Sigma}$ is a connected marked surface with exactly one boundary arc, then $\mathcal{X}_G(\mathbf{\Sigma})$ is isomorphic to the Alekseev-Kosmann-Meinrenken moduli spaces appearing in  \cite{AlekseevMalkin_PoissonCharVar, AlekseevMalkin_PoissonLie, AlekseevKosmannMeinrenken}.
\item When $\mathbf{\Sigma}=(\Sigma, \mathcal{A})$ is connected and $\mathcal{A}\neq \emptyset$, the $G^{\mathcal{A}}$- Poisson variety $\mathcal{X}_G(\mathbf{\Sigma})$ is a twist of the quasi Poisson variety defined independently by Li Bland-Severa   and Nie in \cite{LiBlandSevera_ModuliSpacesQuiltedSurfaces,Nie_QPoissonGoldmanFormula}.
\end{enumerate}
\end{theorem}
The left exact sequence in the fourth item of Theorem \ref{theorem1} is the algebraic analogue of Equation \eqref{eq_RightExactSeq}.

\begin{corollary}[Triangular decomposition of character varieties]\label{corollary1}
For a triangulated marked surface $(\mathbf{\Sigma}, \Delta)$, one has an exact sequence

\begin{equation*}
0 \rightarrow \mathbb{C}[\mathcal{X}_G(\mathbf{\Sigma})]  \xrightarrow{i^{\Delta}} \otimes_{\mathbb{T}\in F(\Delta)} \mathbb{C}[\mathcal{X}_G(\mathbb{T})]
 \xrightarrow{\Delta^L -  \sigma \circ \Delta^R} 
\left(\otimes_{e\in \mathring{\mathcal{E}}(\Delta)}\mathbb{C}[G]\right) \otimes  \left( \otimes_{\mathbb{T}\in F(\Delta)} \mathbb{C}[\mathcal{X}_G(\mathbb{T})] \right),
\end{equation*}

where the embedding $i^{\Delta}$ is a Poisson morphism.
\end{corollary}
Corollary \ref{corollary1} is the algebraic analogue of Equation \eqref{eq_triangulation_modulispace}. 
\vspace{2mm}\par 
When $\mathbf{\Sigma}=(\Sigma, \mathcal{A})$ is connected, a functor $\rho$ is a \textit{good} representation if either $\mathcal{A}\neq \emptyset$ or $\mathcal{A}=\emptyset$ and the restriction of $\rho$ to $\pi_1(\Sigma, v)$ for one (and thus all) basepoint $v\in \Sigma$ is irreducible and has stabilizer equal to the center of $G$. In general, $\rho$ is good if its restriction to every connected component of $\mathbf{\Sigma}$ is good.
Given a functor $\rho \in \mathcal{R}_G(\mathbf{\Sigma})$, we will define a chain complex $(C_{\bullet}(\Sigma, \mathcal{A} ; \rho ), \partial_{\bullet})$ and a cochain complex $(C^{\bullet}(\Sigma, \mathcal{A} ; \rho ), d^{\bullet})$, satisfying the following

\begin{theorem}\label{theorem2}
Given $\rho \in \mathcal{R}_G(\mathbf{\Sigma})$ a good representation with class $[\rho]\in \mathcal{X}_G(\mathbf{\Sigma})$, there exists canonical isomorphisms $\Lambda: T_{[\rho]} \mathcal{X}_{G}(\mathbf{\Sigma}) \xrightarrow{\cong} \mathrm{H}^1(\Sigma, \mathcal{A}; \rho)$ between the Zariski tangent space and the first twisted cohomological group, and $\Lambda^*: \Omega^1_{[\rho]} \mathcal{X}_G(\mathbf{\Sigma}) \xrightarrow{\cong} \mathrm{H}_1(\Sigma, \mathcal{A}; \rho)$ between the cotangent space and the first twisted homological group respectively.
\end{theorem}

\par This cohomological description of the tangent space will play an important role in the definition of the Poisson structure of \stated character varieties, which will appear as an intersection form on twisted groupoid homology. A related description appeared in \cite{GHJW_ModSpacesParBd} for marked surfaces having exactly one marking per boundary component (see Remark \ref{remark_GHJW} for details).

As for Culler-Shalen character varieties, to a conjugacy invariant regular function $f \in \mathbb{C}[G]^G$ and a simple closed curve $\gamma \in \Sigma$, one can associate a curve function $f_{\gamma} \in \mathbb{C}[\mathcal{X}_G(\mathbf{\Sigma})]$. The new feature in \stated character varieties is that for any regular function $f\in \mathbb{C}[G]$ (not necessarily conjugacy invariant) and any arc $\alpha$ whose endpoints lies in $\mathcal{A}$, one can also define a curve function $f_{\alpha}  \in \mathbb{C}[\mathcal{X}_G(\mathbf{\Sigma})]$ (still defined by $f_{\alpha}([\rho]):= f(\rho(\alpha))$). These functions are the analogue of the stated arcs appearing in stated skein algebras. 
When $G$ is standard, we will prove that $\mathbb{C}[\mathcal{X}_G(\mathbf{\Sigma})]$ is generated by curve functions, like in the unmarked case.
 The Poisson bracket is then characterized the following formula.

\begin{theorem}[Generalized Goldman formula]\label{theorem3}
Given two curve functions $f_{\mathcal{C}_1}, h_{\mathcal{C}_2} \in \mathbb{C}[\mathcal{X}_G(\mathbf{\Sigma})]$ and $\rho\in \mathcal{R}_G(\mathbf{\Sigma})$,  the Poisson bracket is characterized by the following formula:
\begin{eqnarray*} \{ f_{\mathcal{C}_1} , h_{\mathcal{C}_2} \}^{\mathfrak{o}} ([\rho])  &=&   \sum_a \sum_{(v_1, v_2)\in S(a)} \varepsilon(v_1, v_2) \left( X_{f, \mathcal{C}_1}(v_1)\otimes X_{h, \mathcal{C}_2}(v_2), r^{\mathfrak{o}(v_1,v_2)} \right) \\ &&+2\sum_{v\in c_1\cap c_2} \varepsilon(v) \left( X_{f, \mathcal{C}_1}(v), X_{h, \mathcal{C}_2}(v) \right) 
 \end{eqnarray*}
\end{theorem}
\par In the above formula, $c_1$ and $c_2$ are two geometric representatives of $\mathcal{C}_1$ and $\mathcal{C}_2$ in transverse position, the first summation is over the boundary arcs $a$, the second summation is over the pairs $(v_1,v_2)$ with $v_i\in a\cap c_i$, the elements $\varepsilon(v), \varepsilon(v_1,v_2), \mathfrak{o}(v_1,v_2) \in \{-1, +1\}$ are signs and $r^{\pm}$ are some classical $r$ matrices. We refer to Subsection $3.4$ for definitions.  When $\mathbf{\Sigma}$ is unmarked, the right-hand-side of above formula coincides with Goldman formula in \cite{Goldman86}. In particular, this formula is still valid for unmarked surfaces with non empty boundary, as was shown by Lawton (\cite[Theorem $15$]{Lawton_PoissonGeomSL3}) and Roche-Szenes (\cite{RocheSzenes}).

\vspace{5mm}
\par \textbf{Organization of the paper}
\vspace{3mm}
\par In Section \ref{sec_definitions}, we introduce the general definition of \stated character varieties. We will define an algebra $\mathbb{C}[\mathcal{R}_G(\mathbf{\Sigma})]$ whose maximal spectrum is the representation space $\mathcal{R}_G(\mathbf{\Sigma})$. We then define the algebraic gauge group action and define the \stated character variety as the maximal spectrum of the subalgebra of $\mathbb{C}[\mathcal{R}_G(\mathbf{\Sigma})]$ of coinvariant functions. The algebra $\mathbb{C}[\mathcal{R}_G(\mathbf{\Sigma})]$ is not finitely generated, hence the representation space is not an affine variety. Thus it will not be  obvious at this stage that the \stated character variety is an affine variety. We then introduce a \textit{discrete model} for the \stated character variety. The idea is that, to define an equivalence class of representation $[\rho]$ in the \stated character variety, we do not need to specify the value of $\rho$ on every paths of the fundamental groupoid, but only on a finite number of generating paths. We introduce the notion of \textit{finite presentation of the fundamental groupoid} which consists of the data of a finite number of paths and a finite number of relations which are sufficient to characterize a class in the relative character variety. Let us give a simple example. Consider the triangle $\mathbb{T}$, so a disc with three boundary arcs $a_1,a_2,a_3$. Fix arbitrary points $v_i \in a_i$ and choose a path $\alpha_i : v_i \to v_{i+1}$ ($i$ is considered modulo $3$) like in Figure \ref{fig_triangle}. We will say that the fundamental groupoid $\Pi_1(\mathbb{T})$ admits a presentation $\mathbb{P}=(\mathbb{V}, \mathbb{G}, \mathbb{RL})$, where $\mathbb{V}=\{v_1,v_2,v_3\}$ are the basepoints ($\mathbb{V}$ non trivially intersect each boundary arc once and each connected component of $\Sigma$), the generators are the paths $\mathbb{G}=\{ \beta_1^{\pm 1}, \beta_2^{\pm 1}, \beta_3^{\pm 1}\}$ (any path between the elements of $\mathbb{V}$ are composition of paths in $\mathbb{G}$) and the relations are the trivial relations $\beta_i \beta_i^{-1} =1$ and the non trivial relation $\beta_1\beta_2\beta_3=1$. 

\begin{figure}[!h] 
\centerline{\includegraphics[width=2.5cm]{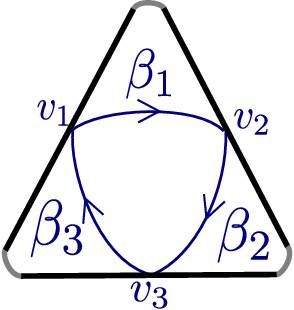} }
\caption{A finite presentation for the fundamental groupoid of the triangle.} 
\label{fig_triangle} 
\end{figure} 

As a consequence the map 
$$ \psi : \mathcal{X}_G(\mathbb{T}) \to \mathcal{X}_G(\mathbb{T}, \mathbb{P}) := \{ (g_1,g_2,g_3) \in G^3 | g_1g_2g_3=e \}$$
sending $[\rho]$ to $(\rho(\beta_1), \rho(\beta_2), \rho(\beta_3))$ will be proved to be an isomorphism. The variety $\mathcal{X}_G(\mathbb{T}, \mathbb{P})\cong G^2$ will be called a discrete model for $\mathcal{X}_G(\mathbb{T})$.

\vspace{2mm}
\par In Section \ref{sec_cohomology}, we define the chain and co-chain complexes defining the twisted groupoid (co)-homological groups, define a natural non-degenerate pairing between them and compare them to usual twisted (co)-homological groups. We then prove Theorem \ref{theorem2}. Eventually we define a skew-symmetric form, the \textit{intersection form}: 
$$\bigcap{}^{\mathfrak{o}} : \mathrm{H}_1(\Sigma, \mathcal{A}; \rho)^{\otimes 2} \rightarrow \mathbb{C}$$
which will characterize the Poisson structure. We then show that this form has a nice behavior for the gluing operation.
\vspace{2mm}
\par In Section \ref{sec_Poisson}, we define the Poisson structures on the algebras of regular functions of \stated character varieties. We first define the Poisson structure on the \stated character variety of the triangle and use triangular decomposition to extend it to general surfaces. We then prove the generalized Goldman formula of Theorem \ref{theorem3} which will imply, on the one hand, that for unmarked surfaces our Poisson structure coincides with Goldman's one and, on the other hand, that it does not depend on the choice of a triangulation but only on the orientations of the boundary arcs.
\vspace{2mm}
\par In Section \ref{sec_abelian}, we consider the case where $G=\mathbb{C}^*$. We prove that the \stated character variety $\mathcal{X}_{\mathbb{C}^*}(\mathbf{\Sigma})$ is canonically isomorphic to the relative singular cohomology group $\mathrm{H}^1(\Sigma, \mathcal{A} ; \mathbb{C}^*)$ and provide a simple description of its Poisson structure. The motivation to study this particular case lies in the connection, established in \cite{KojuQuesneyQNonAb} between this affine variety and the balanced Chekhov-Fock algebra (quantum Teichm\"uller space). 
\vspace{2mm}
\par In Section \ref{sec_fusion} we prove the fifth item of Theorem \ref{theorem1} about the fusion operation using the generalized Goldman formula.
\vspace{2mm}
\par In Appendix \ref{appendix_graph}, we prove that the algebra of regular functions of the \stated character variety is generated by the curve functions provided that $G$ is standard. This result is the key ingredient to identify \stated character varieties with their discrete models. In Appendix \ref{sec_FR}, we compare our construction of \stated character variety for marked surfaces with the Fock-Rosly construction in \cite{FockRosly} based on ciliated graphs. In Appendix \ref{sec_Alekseev} we compare our construction with the moduli spaces appearing in \cite{AlekseevMalkin_PoissonCharVar, AlekseevMalkin_PoissonLie, AlekseevKosmannMeinrenken}. In Appendix \ref{sec_QPoisson}, we use the generalized Goldman formula of Theorem \ref{theorem3} to prove that when $\mathbf{\Sigma}=(\Sigma, \mathcal{A})$ is connected and $\mathcal{A}\neq \emptyset$, the $G^{\mathcal{A}}$- Poisson variety $\mathcal{X}_G(\mathbf{\Sigma})$ is a twist of the quasi Poisson variety defined independently by Li Bland-Severa   and Nie in \cite{LiBlandSevera_ModuliSpacesQuiltedSurfaces,Nie_QPoissonGoldmanFormula}.

\vspace{3mm}
 $\textbf{Acknowledgements:}$ The author is thankful to S.Baseilhac, F.Bonahon, F.Costantino, L.Funar, A.Quesney, P.Roche, F.Ruffino, A.Sikora and J.Toulisse for useful discussions and to the University of South California, the Federal University of S\~ao Carlos, the University of S\~ao Paulo ICMC and the Waseda University for their kind hospitality during the completion of this work. He warmly thanks the referee for pointing to several missing references in literature and for important suggestions which improved the quality of the paper.
 He also thanks S.Lawton for pointing to him the reference \cite{Lawton_PoissonGeomSL3}. He acknowledges  support from the Coordena\c{c}\~ao de Aperfei\c{c}oamento de Pessoal de n\'ivel Superior (CAPES), from the GEometric structures And Representation varieties (GEAR) network, from  the Japanese Society for Promotion of Science (JSPS), from the Centre National de la Recherche Scientifique (CNRS) and from  the European Research Council (ERC DerSympApp) under the European 
 Union
  Horizon 2020 research and innovation program (Grant Agreement No. 768679).


\section{Stated character varieties}\label{sec_definitions}


\subsection{Turning $\Map(X,G)$ into an affine scheme}

\par Let $X$ be a non empty set, $G$ an affine reducible group scheme and $\Map(X,G)$ the set of maps $f:X\to G$. In order to turn $\Map(X,G)$ into an affine scheme, let us define a commutative algebra $\mathbb{C}[\Map(X,G)]$ such that the set of closed points of $\Specm(\mathbb{C}[\Map(X,G)])$ is in one to one bijection with the set of maps $f:X\to G$. We then abusively denote by the same symbol $\Map(X,G)$ the underlying affine scheme and its set of closed points. First, suppose that $G=\SL_N(\mathbb{C})$, so that $\mathbb{C}[\SL_N]=\quotient{\mathbb{C}[X_{i,j}, 1\leq i,j \leq N]}{(\det -1)}$ (where $\det$ is seen as polynomial in the coordinate functions $X_{i,j}$). In this case, one can define
$$ \mathbb{C}[\Map(X,G)]:= \quotient{\mathbb{C}[X_{i,j}^{\alpha}, 1\leq i, j \leq N, \alpha\in X]}{ \left(\det(M_{\alpha})-1 \right)}.$$
Here $M_{\alpha}$ is the $N\times N$ matrix with coefficients in the polynomial algebra $\mathbb{C}[X_{i,j}^{\alpha}, 1\leq i, j \leq N, \alpha \in X]$ defined by $M_\alpha:=( X_{i,j}^{\alpha})_{1\leq i,j\leq N}$. For a general affine reducible group scheme $G$, one can choose an embedding $G\subset \GL_N(\mathbb{C})$ so that $\mathbb{C}[G]=\quotient{\mathbb{C}[X_{i,j}, 1\leq i,j\leq N]}{\left( P_1, \ldots, P_k\right)}$ for some polynomials $P_i$. We can similarly define 
$$ \mathbb{C}[\Map(X,G)]:=  \quotient{\mathbb{C}[X_{i,j}^\alpha, 1\leq i, j \leq N, \alpha\in X]}{ \left( P_1(M_\alpha), \ldots, P_k(M_\alpha) \right)}.$$
Note that for each $\alpha\in X$, there is an obvious embedding $\iota_{\alpha}: \mathbb{C}[G] \hookrightarrow \mathbb{C}[\Map(X,G)]$ sending $X_{i,j}$ to $X_{i,j}^{\alpha}$. For $x\in \mathbb{C}[G]$ and $\alpha \in X$, we write $x_{\alpha}:=\iota_{\alpha}(x) \in \mathbb{C}[\Map(X,G)]$. 
\par 
A closed point of $\Specm(\mathbb{C}[\Map(X,G)])$, described by a character $\chi: \mathbb{C}[\Map(X,G)]\to \mathbb{C}$, induces characters $\chi_{\alpha}:= \chi\circ \iota_{\alpha}: \mathbb{C}[G]\to \mathbb{C}$ which corresponds to elements $f_{\chi}({\alpha})\in G$. We thus get a function $f_{\chi}:X\to G$ and the assignation $\chi \mapsto f_{\chi}$ is a bijection between the closed points of $\Specm(\mathbb{C}[\Map(X,G)])$ and $\Map(X,G)$ as desired.
\par Note that $\Map(X,G)$ has a group structure given by $f_1 \cdot f_2 ({\alpha})=f_1({\alpha})f_2({\alpha})$ for ${\alpha}\in X$ which is reflected by the fact that $\mathbb{C}[\Map(X,G)]$ has a natural Hopf algebra structure uniquely characterized by the requirement that each $\iota_{\alpha}$ is a Hopf algebra morphism. More precisely, let $(\mathbb{C}[G], \Delta, \epsilon, S)$ be the Hopf algebra of regular functions of $G$, then the Hopf algebra structure of $\mathbb{C}[\Map(X,G)]$ is given by 
$$ \Delta(x_{\alpha}) := \sum x^{(1)}_{\alpha} \otimes x^{(2)}_{\alpha}, \quad \epsilon(x_{\alpha}):=\epsilon(x), \quad S(x_{\alpha}):= S(x)_{\alpha}.$$
Here and henceforth, we use Sweedler's notation $\Delta(x)=\sum x^{(1)}\otimes x^{(2)}$ for the coproduct. The coproduct can be alternatively written in matrix notation as 
$$\Delta( M_{\alpha})= M_{\alpha} \boxtimes M_{\alpha}.$$

\begin{remark}
For each finite subset $S\subset X$, we get an embedding $\iota_S:= \otimes_{\alpha \in S}\iota_{\alpha}: \mathbb{C}[G]^{\otimes S} \hookrightarrow \mathbb{C}[\Map(X,G)]$ and we can identify $\mathbb{C}[\Map(X,G)]$ with the projective limit
$$ \mathbb{C}[\Map(X,G)]\cong \varprojlim_{S\subset X, S \mathrm{ finite}} \mathbb{C}[G]^{\otimes S}.$$
The advantage of this alternative definition is that it is independent on the linear embedding $G\subset \GL_N(\mathbb{C})$.

\end{remark}

\subsection{Definition of \stated character varieties}

Let $\mathbf{\Sigma}=(\Sigma, \mathcal{A})$ be a marked surface. 
Recall that $\mathcal{R}_G(\mathbf{\Sigma})$ is the set of functors $\rho : \Pi_1(\Sigma)\to G$ whose restriction to $\Pi_1(\mathcal{A})$ is trivial. ln order to turn it into an affine scheme, note that as a set it is a subset of $\Map(\Pi_1(\Sigma), G)$. 
 Define the ideal $\mathcal{I}_{\Delta}\subset \mathbb{C}[\Map(\Pi_1(\Sigma), G)]$ generated by the elements $x_{\alpha\beta} - \sum x^{(1)}_{\alpha}\cdot x^{(2)}_{\beta}$ for $x\in \mathbb{C}[G]$ and $\alpha, \beta \in \Pi_1(\Sigma)$ two paths such that $t(\alpha)=s(\beta)$.  Define the ideal $\mathcal{I}_{\epsilon} \subset  \mathbb{C}[\Map(\Pi_1(\Sigma), G)]$ generated by the elements $x_{\alpha^t}-\epsilon(x)$ for $x\in \mathbb{C}[G]$ and $\alpha^{t}\in \Pi_1(\mathcal{A})$ (a trivial path). 

\begin{definition}(Representation scheme)
The algebra $\mathbb{C}[\mathcal{R}_G(\mathbf{\Sigma})]$ is the quotient of the algebra $  \mathbb{C}[\Map(\Pi_1(\Sigma), G)]$ by the ideal $\mathcal{I}_{\Delta}+\mathcal{I}_{\epsilon}$. 
\end{definition}

 \begin{remark} For instance, when $G=\SL_N(\mathbb{C})$, we get the synthetic formula 
  $$ \mathbb{C}[\mathcal{R}_{\SL_N}(\mathbf{\Sigma})] \cong
  \quotient{\mathbb{C}\left[X^{\alpha}_{i,j}, \begin{array}{l} \alpha \in \Pi_1(\Sigma), \\ 1\leq i,j\leq N \end{array} \right]}
  {\begin{pmatrix} \det(M_{\alpha})=1, M_{\alpha}M_{\beta}= M_{\alpha\beta}, \\ M_{\alpha^{t}}= \mathds{1}_N \end{pmatrix}}. $$
  \end{remark}

\begin{lemma}\label{lemma_def}
 The maximal spectrum of $\mathbb{C}[\mathcal{R}_G(\mathbf{\Sigma})]$ is in canonical bijection with the representation space $\mathcal{R}_G(\mathbf{\Sigma})$.
 \end{lemma}
 
 \begin{proof}
 Let $p: \mathbb{C}[\Map(\Pi_1(\Sigma), G)] \to \mathbb{C}[\mathcal{R}_G(\mathbf{\Sigma})]$ be the quotient map. 
 Let $\chi : \mathbb{C}[\mathcal{R}_G(\mathbf{\Sigma})] \rightarrow \mathbb{C}$ be a character. The character  $\widetilde{\chi}:= \chi\circ p:  \mathbb{C}[\Map(\Pi_1(\Sigma), G)]\to \mathbb{C}$ defines a map $\rho: \Pi_1(\Sigma) \to G$. Since $\widetilde{\chi}$ vanishes on the ideal $\mathcal{I}_{\Delta}$, it satisfies $\rho(\alpha \beta)=\rho(\alpha) \rho(\beta)$, i.e. $\rho$ is a functor. Since $\widetilde{\chi}$ vanishes on the ideal $\mathcal{I}_{\epsilon}$, it satisfies $\rho(\alpha^t)=e$ for all $\alpha^t\in \Pi_1(\mathcal{A})$, i.e. $\rho \in \mathcal{R}_G(\mathbf{\Sigma})$. Conversely, a functor  $\rho \in \mathcal{R}_G(\mathbf{\Sigma})$ is in particular a map $\Pi_1(\Sigma)\to G$ so defines a character $\widetilde{\chi}: \mathbb{C}[\Map(\Pi_1(\Sigma), G)] \to \mathbb{C}$. The fact that $\rho$ is functor implies $\widetilde{\chi}(\mathcal{I}_{\Delta})=0$ and the fact that it is trivial on $\Pi_1(\mathcal{A})$ implies $\widetilde{\chi}(\mathcal{I}_{\epsilon})=0$ so $\widetilde{\chi}$ lifts to  a character $\chi \mathbb{C}[\mathcal{R}_G(\mathbf{\Sigma})] \rightarrow \mathbb{C}$. 
 These two assignments $\rho \rightarrow \chi$ and $\chi \rightarrow \rho$ are inverse to each other, thus define the desired bijection.
 \end{proof}

\par Define the Hopf algebra $\mathbb{C}[\mathcal{G}]:= \mathbb{C}[\Map( \Sigma\setminus \mathcal{A} , G)]$ whose maximal spectrum is in natural bijection with the gauge group $\mathcal{G}$ of maps $g:\Sigma\to G$ with  trivial restriction on $\mathcal{A}$.   The algebraic group $\mathcal{G}$ acts algebraically on the algebra $\mathbb{C}[\mathcal{R}_G(\mathbf{\Sigma})]$ as follows. We define a left Hopf co-action $\Delta^{\mathcal{G}}: \mathbb{C}[\mathcal{R}_G(\mathbf{\Sigma})]  \rightarrow \mathbb{C}[\mathcal{G}] \otimes\mathbb{C}[\mathcal{R}_G(\mathbf{\Sigma})] $ by the formula:
$$ {\Delta}^{\mathcal{G}}(x_{\alpha}) = \left\{ 
\begin{array}{ll}
\sum (x'_{s(\alpha)}\cdot S(x''')_{t(\alpha)} )\otimes x''_{\alpha} & \mbox{, if }s(\alpha), t(\alpha) \in \Sigma\setminus \mathcal{A}; \\
\sum x'_{s(\alpha)} \otimes x''_{\alpha} & \mbox{, if }s(\alpha) \in \Sigma\setminus \mathcal{A}, t(\alpha) \in \mathcal{A}; \\
\sum S(x'')_{t(\alpha)} \otimes x'_{\alpha} & \mbox{, if } s(\alpha) \in \mathcal{A}, t(\alpha) \in \Sigma \setminus \mathcal{A}; \\
1 \otimes x_{\alpha} &\mbox{, if } s(\alpha), t(\alpha) \in \mathcal{A}.
\end{array}\right. $$
\par This co-action defines  an algebraic action of $\mathcal{G}$ on $\mathbb{C}[\mathcal{R}_G(\mathbf{\Sigma})] $ which corresponds to the  group action $ \mathcal{G} \times \mathcal{R}_G(\mathbf{\Sigma}) \rightarrow \mathcal{R}_G(\mathbf{\Sigma})$ of the introduction, defined by 
$$ g\cdot \rho (\alpha) := g(s(\alpha)) \rho(\alpha) g(t(\alpha))^{-1}, \mbox{ for any }\rho \in \mathcal{R}_G(\mathbf{\Sigma}), g\in \mathcal{G}, \alpha\in \Pi_1(\Sigma). $$

\begin{definition}(\Stated character varieties)
Consider the sub-algebra $\mathbb{C}[\mathcal{X}_G(\mathbf{\Sigma})]:= \mathbb{C}[\mathcal{R}_G(\mathbf{\Sigma})]^{\mathcal{G}}\subset \mathbb{C}[\mathcal{R}_G(\mathbf{\Sigma})]$ of $\mathcal{G}$-invariant functions. The \stated character variety $\mathcal{X}_G(\mathbf{\Sigma})$ is defined as the maximal spectrum of the algebra $\mathbb{C}[\mathcal{X}_G(\mathbf{\Sigma})]$.
\end{definition}

 We will prove that $\mathbb{C}[\mathcal{X}_G(\mathbf{\Sigma})]$ is finitely generated and reduced, hence that the \stated character variety is an affine variety (except possibly for closed surfaces and $G\neq \GL_N, \SL_N$ in which case the question whether $\mathcal{X}_G(\Sigma)$ is reduced or not is open in general, see \cite{Sikora}). 

\subsection{Stabilizer}\label{sec_stabilizer}

In this subsection, we suppose that $\mathbf{\Sigma}$ is connected.
For $\rho \in \mathcal{R}_G(\mathbf{\Sigma})$, we denote by $S_{\rho}\subset \mathcal{G}_{\mathbf{\Sigma}}$ its stabilizer for the $\mathcal{G}_{\mathbf{\Sigma}}$ action. For $v\in \Sigma\setminus \mathcal{A}$, let $\rho_v : \pi_1(\Sigma, v)\to G$ be the restriction of $\rho$ to $\pi_1(\Sigma, v)$ and $p_v: \mathcal{R}_G(\mathbf{\Sigma})\to \Hom(\pi_1(\Sigma, v), G)$ be the regular projection sending $\rho$ to $\rho_v$. We denote by $S_{\rho, v}\subset G$ the stabilizer of $\rho_v$ for the $G$ action by conjugacy. 

\begin{lemma}\label{lemma_stabilizer}
\begin{enumerate}
\item If $\mathcal{A}\neq \emptyset$, then the action of $\mathcal{G}_{\mathbf{\Sigma}}$ on $\mathcal{R}_G(\mathbf{\Sigma})$ is free. 
\item If $\mathcal{A}=\emptyset$, for every $\rho\in \mathcal{R}_G(\mathbf{\Sigma})$ and $v\in \Sigma$ the map $S_{\rho}\to S_{\rho, v}$ sending $g$ to $g(v)$ is an isomorphism.
\end{enumerate}
\end{lemma}

\begin{proof}
Let $\rho\in \mathcal{R}_G(\mathbf{\Sigma})$ and $g\in \mathcal{G}_{\mathbf{\Sigma}}$ such that $g\cdot \rho = \rho$. Then for $\alpha: v \to w$ a path, one has 
$$ g(v) \rho(\alpha) = \rho(\alpha) g(w).$$
If $\mathcal{A}\neq \emptyset$, we can choose $w\in \mathcal{A}$ in which case $g(w)=e$ and, by connectedness,  for every $v\in \Sigma$ one can find a path $\alpha: v\to w$. The above equality then implies $g(v)=e$ as well. Therefore $g$ is the constant map with value $e$ and the stabilizer of $\rho$ is trivial. If $\mathcal{A}=\emptyset$, fix a basepoint $v\in \Sigma$. For every $w\in \Sigma$, one can find a path $\alpha: v\to w$ and then $g(w)= \rho(\alpha)^{-1} g(v) \rho(\alpha)$. Therefore $g$ is determined by $g(v)$ and the map $S_{\rho}\to G$ sending $g$ to $g(v)$ is injective. If $g(v)\in G$ is fixed, one can extend it to a map $g: \Sigma \to G$ by the formula $g(w):=  \rho(\alpha)^{-1} g(v) \rho(\alpha)$ if and only if  $\rho(\alpha)^{-1} g(v) \rho(\alpha)$ does not depend on the choice of the path $\alpha$ connecting $v$ to $w$. If $\beta: v\to w$ is another path and $\gamma_v= \alpha \beta^{-1}\in \pi_1(\Sigma, v)$, then $\rho(\alpha)^{-1} g(v) \rho(\alpha)=\rho(\beta)^{-1} g(v) \rho(\beta)$ if and only if $\rho(\gamma_v)$ commutes with $g(v)$. Therefore $g(v)\in G$ is in the image of the embedding  $S_{\rho}\to G$ if and only if it commutes with all elements $\rho(\gamma_v)$ for $\gamma_v \in \pi_1(\Sigma, v)$, i.e. if and only if $g(v)\in S_{\rho, v}$. This concludes the proof.
\end{proof}

\subsection{Curve functions}

\par We now define a set of regular functions on the character varieties which will be proved to generate the algebra of regular functions. A \textit{curve} $\mathcal{C}$ in $\mathbf{\Sigma}$ is a homotopy class of continuous map $c:[0,1]\rightarrow \Sigma$ such that either $c(0)=c(1)$ (closed curve) or $c(0), c(1) \in \partial \mathcal{A}$ (open curve or arc). 
For open curves, we allow the homotopy to move the endpoints $c(0)$ and $c(1)$ inside their boundary arcs. The map $c$ is called a \textit{geometric representative} of $\mathcal{C}$. The path $\alpha_c : c(0) \rightarrow c(1)$ in $\Pi_1(\Sigma)$ defined by $c$ is called a \textit{path representative} of $\mathcal{C}$.
\vspace{2mm}
\par Let $\mathcal{C}$ be a curve and $f\in \mathbb{C}[G]$ a regular function which is further assumed to be invariant by conjugacy if $\mathcal{C}$ is closed. Let  
$\alpha_{\mathcal{C}}$ a path representative of $\mathcal{C}$.

\begin{definition}
 We define the \textit{curve function} $f_{\mathcal{C}}\in \mathbb{C}[\mathcal{X}_G(\mathbf{\Sigma})]$ to be the class in $\mathbb{C}[\mathcal{R}_G(\mathbf{\Sigma})]$ of the element $f_{\alpha_C} \in \mathbb{C}[\Map(\Pi_1(\Sigma),G)]$. 
 \end{definition}
 
 This class does not depend on the choice of the path representative and is invariant under the gauge group action, hence the function $f_{\mathcal{C}}$ is well defined. 
 
\begin{proposition}\label{prop_holonomy_functions}
When $G$ is standard, the algebra $\mathbb{C}[\mathcal{X}_G(\mathbf{\Sigma})]$ is generated by the curve functions.
\end{proposition}

\par The proof of Proposition \ref{prop_holonomy_functions} is postponed to Appendix A. For now on, we will only consider standard Lie groups $G$ in order to use Proposition \ref{prop_holonomy_functions}.

\subsection{Discrete models}

\par We now define the notion of finite presentation of the fundamental groupoid. Let $\mathbf{\Sigma}=(\Sigma, \mathcal{A})$ be a marked surface. 

\begin{definition}
A \textit{finite generating set} for the fundamental groupoid $\Pi_1(\Sigma)$ relatively to $\mathcal{A}$ is a pair $(\mathbb{V}, \mathbb{G})$, where $\mathbb{V}$ is a finite subset of $\Sigma$ and $\mathbb{G}$ is a finite subset of $\Pi_1(\Sigma)$ such that: 
\begin{enumerate}
\item The set $\mathbb{V}$ is the set of endpoints of the elements of $\mathbb{G}$, \emph{i.e.} $\mathbb{V}=\{ s(\beta), t(\beta) | \beta\in \mathbb{G} \}$.
\item Any pair of generators $\beta_1, \beta_2\in \mathbb{G}$ admits some geometric representatives whose interior are disjoint embedded curves.
\item If $\beta \in \mathbb{G}$ then $\beta^{-1}\in \mathbb{G}$, where $\beta^{-1}$ is the path with opposite orientation of $\beta$.
\item Every curve $\mathcal{C}$ admits a path representative $\alpha_{\mathcal{C}}$ such that $\alpha_{\mathcal{C}}=\beta_1\ldots \beta_n$ with $\beta_i \in \mathbb{G}$.
\item Every boundary arc $a\in \mathcal{A}$ contains a unique element $v_a \in \mathbb{V}$.
\end{enumerate}
\end{definition}

 \par A finite generating set $(\mathbb{V}, \mathbb{G})$ can be characterized by an unoriented embedded graph $\Gamma \subset \Sigma_{\mathcal{P}}$ whose vertices are the elements of $\mathbb{V}$ and whose edges are some geometric representatives of the elements of  $\mathbb{G}$ whose interior are pairwise disjoint. More precisely, we represent any pair  $\beta, \beta^{-1}$ of generators in $\mathbb{G}$ by a single edge in $\Gamma$. Such a graph will be called a \textit{presenting graph} of the generating set. We also denote by $\Pi_1(\mathbb{G})$ the sub-category of $\Pi_1(\Sigma)$ whose objects are the elements of $\mathbb{V}$ and morphisms are composition of elements of $\mathbb{G}$. Note that $\mathbb{V}$ intersects non-trivially each boundary arc and each connected component of $\Sigma$ and that each curve admits a path representative in $\Pi_1(\mathbb{G})$, so the inclusion $\Pi_1(\mathbb{G})\subset \Pi_1(\Sigma)$ is fully faithful. 
 
 \vspace{2mm}
 \par Let $\mathcal{F}(\mathbb{G})$ denote the  free semi-group generated by the elements of $\mathbb{G}$ and let $\Rel_{\mathbb{G}}$ denote the sub-set of $\mathcal{F}(\mathbb{G})$ of elements of the form $R=\beta_{1}\star \ldots \star \beta_{n}$ such that $t(\beta_{i})= s(\beta_{i+1})$ and such that the path $\beta_1\ldots \beta_n$ is trivial. We write $R^{-1}:= \beta_n^{-1} \star \ldots \star \beta_1^{-1}$.

 \begin{definition}
  A finite subset $\mathbb{RL}$ of $\Rel_{\mathbb{G}}$ is called a \textit{finite set of relations} if:
 
 \begin{enumerate}
 \item Every word  $R\in \Rel_{\mathbb{G}}$ can be decomposed as $ R = \beta \star R_1^{\varepsilon_1} \star \ldots \star R_m^{\varepsilon_m}\star \beta^{-1}$, where $R_i \in \mathbb{RL}$, $\varepsilon_i \in \{ \pm 1 \}$ and $\beta=\beta_{1}\star \ldots \star \beta_{n}\in \mathcal{F}(\mathbb{G})$ is such that $t(\beta_{i})= s(\beta_{i+1})$. 
 \item If $\beta\in \mathbb{G}$, the relation $\beta\star \beta^{-1}$ belongs to $\mathbb{RL}$. We call such a relation a \textit{trivial relation} of $\mathbb{RL}$.
 \end{enumerate}
 
 \par A \textit{finite presentation} of $\Pi_1(\Sigma)$ (relatively to $\mathcal{A}$) is a triple $\mathbb{P}=(\mathbb{V}, \mathbb{G}, \mathbb{RL})$ where $(\mathbb{V}, \mathbb{G})$ is a finite generating set and $\mathbb{RL}$ a finite set of relations.
 \end{definition}
 
  Given such a finite presentation with $\mathbb{G}=\{\beta_1, \ldots , \beta_n\}$ and $\mathbb{RL}=\{R_1, \ldots R_m\}$, we define the map $\mathcal{R} : G^{\mathbb{G}} \rightarrow G^{\mathbb{RL}}$ by the formula
 $$ \mathcal{R}(g_1, \ldots, g_n) = \left(R_1(g_1,\ldots, g_n), \ldots, R_m(g_1, \ldots, g_m) \right) \mbox{, for any }(g_1, \ldots, g_n)\in G^{\mathbb{G}}. $$

 \begin{definition}
 The \textit{discrete representation variety} is the subset $\mathcal{R}_G(\mathbf{\Sigma}, \mathbb{P}):= \mathcal{R}^{-1}(e, \ldots, e) \subset G^{\mathbb{G}}$.
 \end{definition}
 
  Since $G$ is affine, $G^{\mathbb{G}}$ is an affine variety and since the subset $\mathcal{R}_G(\mathbf{\Sigma})$ is defined by polynomial equations, the discrete representation variety is a finitely generated affine scheme. As we shall see, when $\mathcal{A}\neq \emptyset$ or when $G=\GL_N, \SL_N$, it is reduced as well, so it is a variety indeed.
 \vspace{2mm}
 \par Decompose the set of vertices as $\mathbb{V}= \mathring{\mathbb{V}}\cup \mathbb{V}^{\partial}$ where $\mathring{\mathbb{V}}=\mathbb{V} \cap (\Sigma\setminus \mathcal{A})$ and $\mathbb{V}^{\partial}= \mathbb{V}\cap \mathcal{A}$. We define the \textit{discrete gauge group} to be the algebraic reducible group $\mathcal{G}_{\mathbb{P}}:= G^{\mathring{\mathbb{V}}}$. The discrete gauge group acts algebraically on $G^{\mathbb{G}}$ as follows. Let $g=(g_{\mathring{v}_1}, \ldots, g_{\mathring{v}_s}) \in \mathcal{G}_{\mathbb{P}}$ and $\rho=(\rho(\beta_1), \ldots , \rho(\beta_n))\in G^{\mathbb{G}}$. Define $g\cdot \rho = (g\cdot \rho(\beta_1), \ldots , g\cdot \rho(\beta_n))$ by the formula: 
 $$ g\cdot \rho (\beta_i) = \left\{ 
 \begin{array}{ll}
 g(s(\beta_i)) \rho(\beta_i) g(t(\beta_i))^{-1} &\mbox{, if }s(\beta_i), t(\beta_i) \in \mathring{\mathbb{V}}; 
 \\ g(s(\beta_i)) \rho(\beta_i) & \mbox{, if }s(\beta_i) \in \mathring{\mathbb{V}}, t(\beta_i) \in \mathbb{V}^{\partial}; 
 \\ \rho(\beta_i)g(t(\beta_i))^{-1} & \mbox{, if }s(\beta_i)\in \mathbb{V}^{\partial}, t(\beta_i)\in \mathring{\mathbb{V}}; 
 \\ \rho(\beta_i) &\mbox{, if } s(\beta_i), t(\beta_i) \in \mathbb{V}^{\partial}.
 \end{array} \right.
 $$
 \par The action of $\mathcal{G}_{\mathbb{P}}$ preserves the sub-variety $\mathcal{R}_G(\mathbf{\Sigma}, \mathbb{P})\subset G^{\mathbb{G}}$, hence induces an algebraic action of the discrete gauge group on the representation variety. 
 
 \begin{definition}
 The \textit{discrete \stated character variety} is the GIT quotient
 $$ \mathcal{X}_G(\mathbf{\Sigma}, \mathbb{P}):= {\mathcal{R}_G(\mathbf{\Sigma}, \mathbb{P})}\sslash {\mathcal{G}_{\mathbb{P}}}. $$
 \end{definition}

 In other words, $\mathbb{C}[\mathcal{X}_G(\mathbf{\Sigma}, \mathbb{P})] = \mathbb{C}[\mathcal{R}_G(\mathbf{\Sigma}, \mathbb{P})]^{\mathcal{G}_{\mathbb{P}}}$ is the sub-algebra of functions invariants under the action of the gauge group. Since the discrete gauge group is reductive and the representations variety is an affine variety, the discrete character variety is an affine variety whenever $\mathcal{R}_G(\mathbf{\Sigma}, \mathbb{P})$ is. 
 \vspace{2mm}
 \par The idea of defining a gauge equivalence class of connections by their holonomies over a finite set of paths is called lattice gauge field theory in the physics literature. It has been used by several mathematicians including the authors of \cite{BuffenoirRoche, BuffenoirRoche2, FockRosly, AlekseevKosmannMeinrenken, AlekseevGrosseSchomerus_LatticeCS1, AlekseevGrosseSchomerus_LatticeCS2, GHJW_ModSpacesParBd, BaseilhacRoche_LGFT1}. We now list some finite presentations that will be used in the paper.
 
 \begin{example}\label{example_finite_presentations}
 \begin{enumerate}
 \item Let $\mathbf{\Sigma}=(\Sigma, \emptyset)$ be an unmarked surface such that $\Sigma$ is  connected. Fix $b\in \Sigma$ a basepoint and consider a finite presentation $P=\left<G, R\right>$ of the fundamental group $\pi_1(\Sigma, b)$. We associate to this finite presentation a finite presentation $\mathbb{P}$ of the fundamental groupoid where $\mathbb{V}=\{b\}$, the set of generators $\mathbb{G}$ is the set of elements of $G$ together with their inverse and the set of relations $\mathbb{RL}$ is the set $R$ to which we add the eventual missing trivial relations $\gamma \star \gamma^{-1}$. The discrete representation space associated to this presentation is the set of group morphisms $\rho : \pi_1(\Sigma, b) \rightarrow G$ and the discrete gauge $\mathcal{G}_{\mathbb{P}}=G$ acts by conjugacy. Hence the discrete character variety associated to such a presentation is the traditional (Culler-Shalen) one, that is $\mathcal{X}_G(\mathbf{\Sigma}, \mathbb{P})= {\Hom \left( \pi_1(\Sigma, b), G \right) }\sslash G$.
 
 \item Suppose that $\mathbf{\Sigma}$ is unmarked and $\Sigma$ closed. A cellular decomposition of $\Sigma$ induces a finite presentation of the fundamental groupoid where $\mathbb{V}$ is the set of $0$-cells, $\mathbb{G}$ is the set of $1$-cells and the non-trivial relations of $\mathbb{RL}$ correspond to the $2$-cells.

 \item To a fat graph $\Gamma$, one can associate a surface $\Sigma(\Gamma)$ by thickening the graph. If moreover the fat graph has a cilitated structure $c$ (total ordering of the adjacent half-edge of each puncture) one can associate a marked surface $\mathbf{\Sigma}^0(\Gamma,c)=(\Sigma(\Gamma), \mathcal{A}(c))$ by placing a boundary arc at each cilium. One then get a finite presentation $\mathbb{P}=(\mathbb{V}, \mathbb{G}, \mathbb{RL})$ of $\Pi_1(\Sigma)$ relatively to $\mathcal{A}(c)$ where $\mathbb{V}$ is the set of vertices of $\Gamma$, $\mathbb{G}$ its edges and $\mathbb{RL}$ has only trivial relations. The associated discrete model was considered by Fock and Rosly in \cite{FockRosly}. We refer to Appendix \ref{sec_FR} for details.
 
  \item The \textit{bigon} (disc with two boundary arcs) and the triangle (disc with three boundary arcs) have natural presentations of their fundamental groupoid depicted in Figure \ref{fig_presentation}. The presentation for the bigon has generators $\mathbb{G}=\{\alpha^{\pm 1}\}$, where $\alpha$ has endpoints in both boundary arcs and only trivial relations, so the corresponding discrete model is $\mathcal{X}_G(\mathbb{B}, \mathbb{P})\cong G$. The presentation for the triangle, described in the introduction, has generators $\mathbb{G}=\{\alpha_i^{\pm 1}, i =1,2,3\}$ and the only one non-trivial relation $\alpha_1\star \alpha_2\star\alpha_3$, so the corresponding discrete model is $\mathcal{X}_G(\mathbb{T}, \mathbb{P})\cong \{(g_1,g_2,g_3)\in G^3| g_1g_2g_3=e\}$. 
 
 \item For a triangulated marked surface $(\mathbf{\Sigma}, \Delta)$ (see Section \ref{sec_triangulation_modularoperad} for details on triangulations), we define a finite presentation $\mathbb{P}^{\Delta}$ where $\mathbb{V}$ has one point $v_e$  in each edge $e\in \mathcal{E}(\Delta)$ of the triangulation,  $\mathbb{G}$ has $6$ generators ${\alpha_{i,\mathbb{T}}}^{\pm 1}, i=1,2,3$ in each face $\mathbb{T}\in F(\Delta)$, defined as in the case of the triangle, and $\mathbb{RL}$ has one non-trivial relation ${\alpha_{1,\mathbb{T}}}\star{\alpha_{2,\mathbb{T}}}\star{\alpha_{3,\mathbb{T}}}$ for each face $\mathbb{T}\in F(\Delta)$. Figure \ref{fig_presentation} illustrates the example of a one punctured torus. The discrete model associated to such a triangulation was considered by Buffenoir and Roche in \cite{BuffenoirRoche, BuffenoirRoche2}. 
 
\begin{figure}[!h] 
\centerline{\includegraphics[width=10cm]{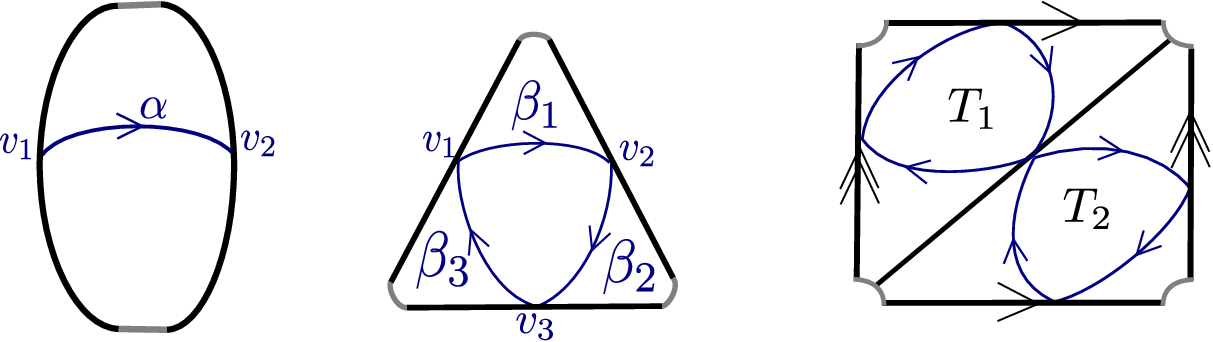} }
\caption{On the left, the bigon  and the triangle with their canonical presentations.  On the right, a triangulated  torus having one boundary component and no boundary arc with the finite presentation associated to a triangulation.} 
\label{fig_presentation} 
\end{figure} 
 
 \item Suppose $\Sigma$ is a compact connected oriented surface of genus $g$ with $n>0$ boundary components. For each boundary component $\partial$ choose a single boundary arc $a_{\partial}\subset \partial$ and set $\mathcal{A}=\{ a_{\partial} \}_{\partial \in \pi_0(\partial \Sigma)}$ and $\mathbf{\Sigma}=(\Sigma, \mathcal{A})$. Choose a point $v_0$ in the interior of $\Sigma$ and for each boundary arc, choose a point $v_{\partial} \in a_{\partial}$. Set $\mathbb{V}=\{v_0\} \cup \{v_{\partial}, \partial \in \pi_0(\partial \Sigma)\}$. Consider some longitudes and meridians $\lambda_1,\mu_1, \ldots, \lambda_g, \mu_g \in \pi_1(\Sigma, v_0)(=\End_{\Pi_1(\Sigma)}(v_0))$. For each boundary component $\partial$, choose a path $\delta_{\partial} : v_0\to v_{\partial}$ and a peripheral loop $\gamma_{\partial}: v_{\partial} \to v_{\partial}$ parallel to $\partial$. Set $\mathbb{G}=\{\lambda_i, \mu_i, \delta_{\partial}, \gamma_{\partial} | 1\leq i \leq g, \partial \in \pi_0(\partial \Sigma)\}$. The relation $R:= \prod_{i=1}^g [\lambda_i, \mu_i] \prod_{\partial \in \pi_0(\Sigma)} \delta_{\partial}\gamma_{\partial} \delta_{\partial}^{-1}$ together with the trivial relations form a set of relations $\mathbb{RL}$ so that $\mathbb{P}=(\mathbb{V}, \mathbb{G}, \mathbb{RL})$ is a finite presentation of $\Pi_1(\Sigma)$ relatively to $\mathcal{A}$. The associated discrete model $\mathcal{X}_G(\mathbf{\Sigma}, \mathbb{P})$ where considered by Guruprasad, Huebschmann, Jeffrey and Weinstein in \cite{GHJW_ModSpacesParBd}. 
 
 \item Consider a marked surface $\mathbf{\Sigma}=(\Sigma, \mathcal{A})$, where $\Sigma$ is connected of genus $g$ and $\mathcal{A}\neq \emptyset$. 
 The groupoid $\Pi_1(\Sigma)$ admits a finite presentation relative to $\mathcal{A}$ with no non-trivial relation  depicted in Figure \ref{fig_generators_final} and defined as follows. 
  For each boundary arc $a$ fix a point $v_a \in a$ and set $\mathbb{V}=\{v_a\}_{a\in \mathcal{A}}$. We fix one particular boundary arc $a_0$ in some boundary component $\partial_0\in \pi_0(\partial \Sigma)$ with point $v_0:=v_{a_0}$. 
   Consider some longitudes and meridians $\lambda_1, \mu_1, \ldots, \lambda_g, \mu_g \in \pi_1(\Sigma, v_0)(=\End_{\Pi_1(\Sigma)}(v_0))$. For each boundary component $\partial$ of $\partial \Sigma$ with no boundary arc, consider a closed path $\delta_{\partial} \in \pi_1(\Sigma, v_0)$ encircling $\partial$ once. For each boundary component $\partial \neq \partial_0$ of $\partial \Sigma$ having some boundary arcs $a_{\partial, 1}, \ldots, a_{\partial, k}$ ordained cyclically in the counterclockwise direction,  consider a path $\delta_{\partial}: v_0\to v_{\partial, 1}$ in the case where $\partial \neq \partial_0$, and some paths $\beta_{\partial, i}: v_{a_{\partial, i}} \to v_{a_{\partial, i+1}}, i\in \mathbb{Z}/k\mathbb{Z}$ homotopic to subarcs of $\partial$. 
   The set $\mathbb{G}'$ formed by the paths $\lambda_i, \mu_i, i=1,\ldots, g$ and  by the paths $\delta_{\partial}$ and $\beta_{\partial, i}$, together with their inverses, forms a finite set of generators for $\Pi_1(\Sigma)$ defining a finite presentation having exactly one non trivial relation. Using this relation, one can express any element of the form $\beta_{\partial, i}$ in term of the other generators. The set $\mathbb{G}$ obtained from $\mathbb{G}'$ by removing an arbitrary pair of generators $\beta_{\partial, i}^{\pm 1}$ form the  generating set of a finite presentation $\mathbb{P}=(\mathbb{V}, \mathbb{G}, \mathbb{RL})$ of $\Pi_1(\Sigma)$ (relative to $\mathcal{A}$) having no non-trivial relation. Note that the set $\mathbb{G}$ has cardinal
   $$ d:= \frac{1}{2}|\mathbb{G}| = 2g-2 +|\mathcal{A}| + |\pi_0(\partial \Sigma)|.$$
   Therefore the discrete model is $\mathcal{X}_G(\mathbf{\Sigma}, \mathbb{P}) \cong G^d$.

 In the particular case where $\mathbf{\Sigma}$ has exactly one boundary arc, the associated discrete model was considered by Alekseev and Malkin in \cite{AlekseevMalkin_PoissonCharVar} in the context of classical lattice gauge field theory and in \cite{AlekseevSchomerus_RepCS, Faitg_LGFT_MCG, Faitg_LGFT_SSkein, BaseilhacRoche_LGFT1} in the quantum case.
 
\begin{figure}[!h] 
\centerline{\includegraphics[width=9cm]{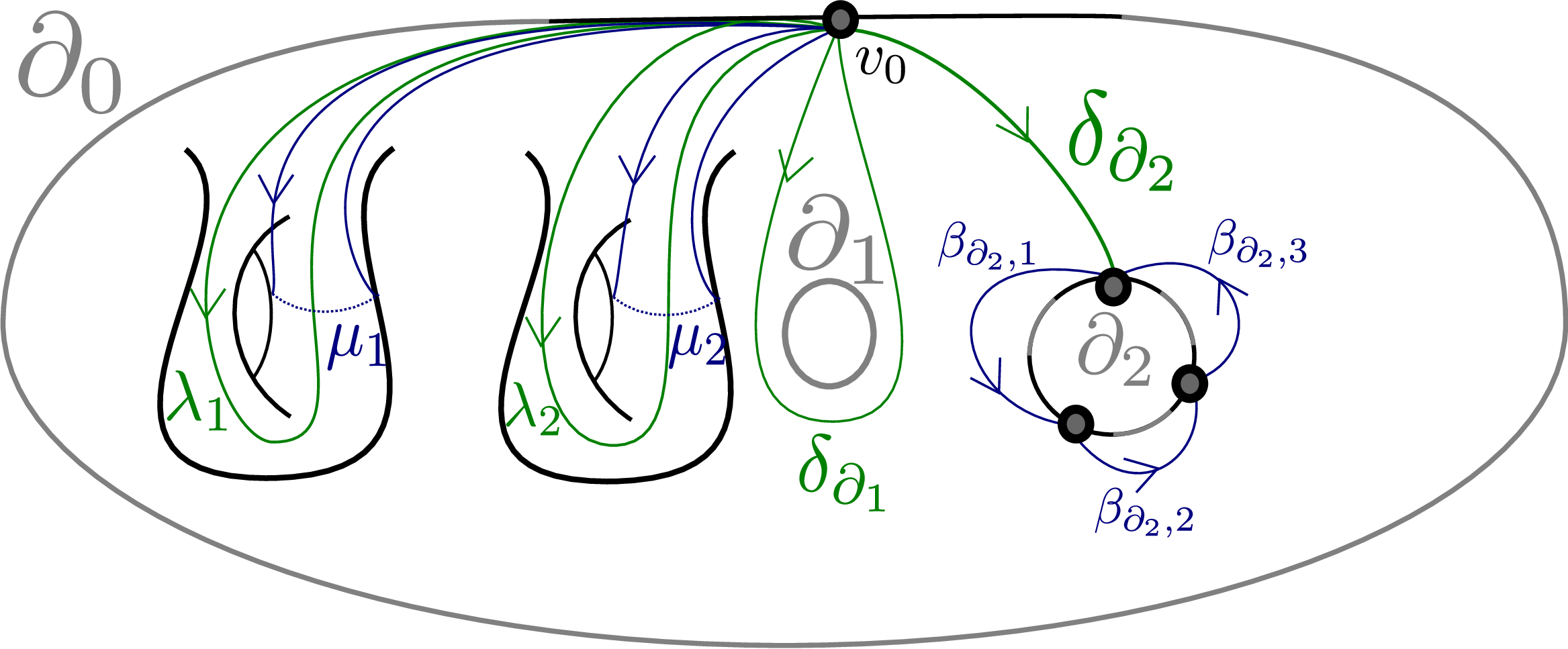} }
\caption{A set of generators $\mathbb{G}$ for $\Pi_1(\Sigma)$ when the marked surface has genus $2$, and three boundary components $\partial_0,\partial_1, \partial_2$ having respectively $1, 0$ and $3$ boundary arcs. Here, $\mathbb{G}$ is obtained from $\mathbb{G}'$ by removing the arc $\beta_{\partial_0, 1}$. } 
\label{fig_generators_final} 
 \end{figure}

 \end{enumerate}
\end{example}

\begin{remark} By Example \ref{example_finite_presentations} the character variety of any marked surface admits a discrete model.
\end{remark}

\par Let $(\mathbf{\Sigma}, \mathbb{P})$ be a punctured surface with a finite presentation and consider a curve $\mathcal{C}$ together with a regular function $f\in \mathbb{C}[G]$ which is further assumed to be invariant by conjugacy if $\mathcal{C}$ is closed. Consider a path representative of $\mathcal{C}$ together with a decomposition $\alpha_{\mathcal{C}}= \beta_{i_1}\ldots \beta_{i_k}$ where the $\beta_{i}\in \mathbb{G}$ are in the generating set. 

\begin{definition}
The \textit{curve function} $f_{\mathcal{C}} \in \mathbb{C}[\mathcal{X}_G(\mathbf{\Sigma}, \mathbb{P})]$ is defined by the formula $f_{\mathcal{C}} (\rho) = f(\rho(\beta_{i_1})\ldots \rho(\beta_{i_k}))$.
\end{definition}

It results from the definition of being a finite set of relations that $f_{\mathcal{C}}$ does not depend on the path representative $\alpha_{\mathcal{C}}$ nor on its decomposition in $\mathbb{G}$. Moreover $f_{\mathcal{C}}$ is invariant under the action of the discrete gauge group, hence it is a regular function $f_{\mathcal{C}}\in \mathbb{C}[\mathcal{X}_G(\mathbf{\Sigma}, \mathbb{P})]$ of the character variety. 

\vspace{2mm}
\par We now define a canonical isomorphism $\Psi^{\mathbb{P}} : \mathcal{X}_G(\mathbf{\Sigma}) \xrightarrow{\cong} \mathcal{X}_G(\mathbf{\Sigma}, \mathbb{P})$ between the  character variety defined in the first sub-section and the discrete model we introduced in this subsection. Recall that we defined a morphism $\mathcal{R}: G^{\mathbb{G}}\rightarrow G^{\mathbb{RL}}$ such that $\mathcal{R}_G(\mathbf{\Sigma}, \mathbb{P})=\mathcal{R}^{-1}(e, \ldots, e)$. Denote by $\mathcal{R}^* : \mathbb{C}[G]^{\otimes \mathbb{RL}} \rightarrow \mathbb{C}[G]^{\otimes \mathbb{G}}$ the morphism of algebras associated to $\mathcal{R}$. We have the following exact sequence
$$ \mathbb{C}[G]^{\otimes \mathbb{RL}} \xrightarrow{\mathcal{R}^* -\eta^{\otimes \mathbb{G}} \circ \epsilon^{\otimes \mathbb{RL}}} \mathbb{C}[G]^{\otimes \mathbb{G}} \rightarrow \mathbb{C}[\mathcal{R}_G(\mathbf{\Sigma}, \mathbb{P})] \rightarrow 0.$$
\par On the other hand, the representation space $\mathcal{R}_G(\mathbf{\Sigma})$ is defined by the following exact sequence
$$ \mathcal{I}_{\Delta}+\mathcal{I}_{\epsilon} \xrightarrow{\iota} \mathbb{C}[\Map(\Pi_1(\Sigma), G)] \rightarrow \mathbb{C}[\mathcal{R}_G(\mathbf{\Sigma})] \rightarrow 0,$$
where $\iota$ represents the inclusion map. Consider the natural injective morphism $\widetilde{\phi}^{\mathbb{P}}:= \otimes_{\alpha \in \mathbb{G}}\iota_{\alpha} : \mathbb{C}[G]^{\otimes \mathbb{G}} \hookrightarrow \mathbb{C}[\Map(\Pi_1(\Sigma), G)]$. Denote by $\mathcal{I}_{\mathbb{P}}\subset \mathbb{C}[\Map(\Pi_1(\Sigma), G)] $ the ideal generated by the algebra $\widetilde{\phi}^{\mathbb{P}} \circ (\mathcal{R}^* - \eta^{\otimes \mathbb{G}} \circ \epsilon^{\otimes \mathbb{RL}}) \left(\mathbb{C}[G]^{\otimes \mathbb{RL}} \right)$.

\vspace{2mm}
\par If $R:= \beta_1 \star \ldots \star \beta_n \in \Rel_{\mathbb{G}}$ is a relation and $x\in \mathbb{C}[G]$, define the element $x_R:= \sum (x^{(1)})_{\beta_1} \ldots (x^{(n)})_{\beta_n}$ and denote by $\mathcal{I}_R \subset  \mathbb{C}[\Map(\Pi_1(\Sigma), G)]$ the ideal generated by the elements $x_R - \epsilon(x)$ with $x\in \mathbb{C}[G]$. By definition, the ideal $\mathcal{I}_{\mathbb{P}}$ is the sum of the ideals $\mathcal{I}_R$ with $R\in \mathbb{RL}$. We have the equalities:

\begin{eqnarray*}
x_R - \epsilon(x) &=& \sum (x^{(1)})_{\beta_{i_1}}   \ldots  (x^{(k)})_{\beta_{i_k}} - \epsilon(x) \\
&=&  \left( \sum (x^{(1)})_{\beta_{i_1}}  \ldots  (x^{(k)})_{\beta_{i_k}} -(x)_{\beta_{i_1}\ldots \beta_{i_k}} \right) \\
 && +  \left( (x)_{\beta_{i_1}\ldots \beta_{i_k}} - \epsilon(x) \right) \in \mathcal{I}_{\Delta} + \mathcal{I}_{\epsilon}
\end{eqnarray*}

\par This proves the inclusion $\mathcal{I}_{\mathbb{P}} \subset \mathcal{I}_{\Delta}+\mathcal{I}_{\epsilon}$, hence  the morphism $\widetilde{\phi}^{\mathbb{P}} : \mathbb{C}[G]^{\otimes \mathbb{G}} \hookrightarrow  \mathbb{C}[\Map(\Pi_1(\Sigma), G)] $ induces a morphism $\phi^{\mathbb{P}} : \mathbb{C}[\mathcal{R}_G(\mathbf{\Sigma}, \mathbb{P})] \hookrightarrow \mathbb{C}[\mathcal{R}_G(\mathbf{\Sigma})]$.  Consider the injective algebra morphism $\mathbb{C}[\mathcal{G}_{\mathbb{P}}]=\mathbb{C}[G]^{\otimes \mathring{\mathbb{V}}} \hookrightarrow  \mathbb{C}[\Map(\Sigma \setminus \mathcal{A}, G)]= \mathbb{C}[\mathcal{G}]$ induced by the inclusion $\mathring{\mathbb{V}}\subset\Sigma\setminus \mathcal{A}$. This inclusion induces a surjective morphism of algebraic groups $\iota_{\mathcal{G}} : \mathcal{G} \rightarrow \mathcal{G}_{\mathbb{P}}$. The injective morphism $\phi^{\mathbb{P}} : \mathbb{C}[\mathcal{R}_G(\mathbf{\Sigma}, \mathbb{P})] \hookrightarrow \mathbb{C}[\mathcal{R}_G(\mathbf{\Sigma})]$ is $\iota_{\mathcal{G}}$-equivariant by definition, hence it induces an injective algebra morphism $\phi^{\mathbb{P}} : \mathbb{C}[\mathcal{X}_G(\mathbf{\Sigma}, \mathbb{P})] \hookrightarrow \mathbb{C}[\mathcal{X}_G(\mathbf{\Sigma})]$. We denote by $\Psi^{\mathbb{P}} : \mathcal{X}_G(\mathbf{\Sigma}) \rightarrow \mathcal{X}_G(\mathbf{\Sigma}, \mathbb{P})$ the surjective regular map induced by $\phi^{\mathbb{P}}$.

\begin{proposition}\label{proposition_discrete_model}
 The regular morphism $\Psi^{\mathbb{P}} : \mathcal{X}_G(\mathbf{\Sigma}) \rightarrow \mathcal{X}_G(\mathbf{\Sigma}, \mathbb{P})$ is an isomorphism. Therefore, the \stated character variety $\mathcal{X}_G(\mathbf{\Sigma})$ is a scheme of finite type.
\end{proposition}

\begin{corollary}\label{coro_discrete_model} Suppose that $\mathbf{\Sigma}$ is connected.
\begin{enumerate}
\item If $\mathcal{A}\neq \emptyset$, then $\mathcal{X}_G(\mathbf{\Sigma})\cong G^n$ for some $n\geq 0$. In particular it is a smooth affine variety.
\item If $\mathcal{A}=\emptyset$, then $\mathcal{X}_G(\mathbf{\Sigma})$ is isomorphic to the usual character variety. In particular it is reduced (thus a variety) when $G=\GL_N, \SL_N$ or when $\Sigma$ is open. 
\end{enumerate}
\end{corollary}

\par The surjectivity of $\phi^{\mathbb{P}}$ will follow from Proposition \ref{prop_holonomy_functions}. To prove the injectivity, we first state two technical lemmas. 

\begin{lemma}\label{lemma_injectivity1}
Let $\psi_1, \psi_2 : \mathbb{C}[G]\rightarrow \mathbb{C}[G]^{\otimes 2}$ be the two morphisms defined by $\psi_1:=\Delta-\eta^{\otimes 2}\circ \epsilon$ and $\psi_2:= \id\otimes \epsilon - \epsilon\otimes S$. Denote by $\mathcal{I}_1, \mathcal{I}_2 \subset \mathbb{C}[G]^{\otimes 2}$ the ideals generated by the images of $\psi_1$ and $\psi_2$ respectively. Then one has an inclusion $\mathcal{I}_2\subset \mathcal{I}_1$.
\end{lemma}

\par Note that Lemma \ref{lemma_injectivity1} implies that in both $\mathbb{C}[\mathcal{R}_G(\mathbf{\Sigma})]$ and $\mathbb{C}[\mathcal{R}_G(\mathbf{\Sigma}, \mathbb{P})]$, we have the equality $[x]_{\beta} = [S(x)]_{\beta^{-1}}$. The inclusion $\mathcal{I}_1\subset \mathcal{I}_2$ obviously holds for any Hopf algebra. Moreover, the fact that for any $g_1, g_2 \in G$ we have $g_1g_2=e$ if and only if $g_1=g_2^{-1}$ implies, by the Nullstellensatz, that the radicals $\sqrt{\mathcal{I}}_1$ and $\sqrt{\mathcal{I}}_2$ are equal. However it is not obvious, a priori, that $\mathcal{I}_1$ is radical. 

\begin{proof}
We first suppose that $G=\GL_N(\mathbb{C})$ and write 
$$\mathbb{C}[\GL_N]:= \quotient{ \mathbb{C}[\det^{-1}, x_{i,j}, 1\leq i,j \leq N ]}{(\det \cdot \det^{-1} -1)}.$$
 A straightforward computation shows the equality
$$ x_{i,j}\otimes 1 - 1 \otimes S(x_{i,j}) = \sum_k \left( \Delta(x_{i,k}) - \epsilon(x_{i,k}) \right) \left( 1\otimes S(x_{k,j}) \right) \subset \mathcal{I}_1.$$
This proves the inclusion  $\mathcal{I}_2\subset \mathcal{I}_1$ when $G=\GL_N(\mathbb{C})$. Now for a general affine Lie group $G$, consider a closed embedding $G \hookrightarrow \GL_N(\mathbb{C})$ defined by a surjective Hopf morphism $p: \mathbb{C}[\GL_N] \rightarrow \mathbb{C}[G]$. The morphism $p$ sends the ideals $\mathcal{I}_1$ and $\mathcal{I}_2$ associated to $\GL_N(\mathbb{C})$ to the corresponding ideals associated to $G$. Hence the result holds for any $G$.
\end{proof}

\begin{lemma}\label{lemma_injectivity2}
Let $R_1, R_2$ and $R=\beta_1 \ldots \beta_n$ be some relations in $\Rel_{\mathbb{G}}$ and $\beta \in \Pi_1(\Sigma)$ a path such that $t(\beta)=s(\beta_1)$. Then the followings statements hold: 
\begin{enumerate}
\item One has the inclusion $\mathcal{I}_{R_1 \star R_2} \subset \mathcal{I}_{R_1} + \mathcal{I}_{R_2}$.
\item One has the inclusion $\mathcal{I}_{\beta \star R \star \beta^{-1}} \subset \mathcal{I}_R + \mathcal{I}_{\beta \star \beta^{-1}}$.
\item One has the inclusion $\mathcal{I}_{R^{-1}} \subset \mathcal{I}_R + \sum_i \mathcal{I}_{\beta_i \star \beta_i^{-1}}$.
\item If $\alpha= \alpha_1 \ldots \alpha_n $ is a path in $\mathbb{G}$ such that each $\alpha_i$ is in $\mathbb{G}$, then for any $x\in \mathbb{C}[G]$ one has
$$ x_{\alpha} - \sum (x^{(1)})_{\alpha_1} \ldots (x^{(n)})_{\alpha_n} \in \mathcal{I}_{\mathbb{P}}.$$
\end{enumerate}
 \par In particular, for any relation $R\in \Rel_{\mathbb{G}}$, one has $\mathcal{I}_R \subset \mathcal{I}_{\mathbb{P}}$.
\end{lemma}

\begin{proof}
Fix $x \in \mathbb{C}[G]$ and compute:
\begin{eqnarray*}
x_{R_1\star R_2} - \epsilon(x) &=& \sum (x^{(1)})_{R_1}(x^{(2)})_{R_2} - \epsilon(x) \\
&=& \sum \left( (x^{(1)})_{R_1} - \epsilon(x^{(1)}) \right) \left( (x^{(2)})_{R_2} - \epsilon(x^{(2)}) \right)\\
&&  + \left(x_{R_1}-\epsilon(x)\right) + \left(x_{R_2} - \epsilon(x)\right) \in  \mathcal{I}_{R_1} + \mathcal{I}_{R_2}
\end{eqnarray*}
This proves the first assertion. For the second, we compute:
\begin{eqnarray*}
x_{\beta\star R \star \beta^{-1}} - \epsilon(x) &=& \sum (x^{(1)})_{\beta} (x^{(2)})_R (x^{(3)})_{\beta^{-1}} - \epsilon(x) \\
&=& \sum (x^{(1)})_{\beta} (x^{(2)})_{R} \left( x^{(3)}_{\beta^{-1}} - S(x^{(3)})_{\beta} \right) \\
&& + \sum (x^{(1)})_{\beta} \left( (x^{(2)})_R -\epsilon(x^{(2)}) \right) S(x^{(3)})_{\beta} \in \mathcal{I}_R + \mathcal{I}_{\beta \star \beta^{-1}}
\end{eqnarray*}
Here we used Lemma \ref{lemma_injectivity1} for the last inclusion. To prove the third assertion, we first introduce some notations. Let $\varepsilon = (\varepsilon_1, \ldots, \varepsilon_n) \in \{ -, + \}^n$ and denote by $| \varepsilon |$ the number of indices $i$ such that $\varepsilon_i = -$. Using Sweedler's notation $\Delta^{(n-1)}(x)= \sum x^{(1)}\otimes \ldots\otimes x^{(n)}$, write $b_{\varepsilon}:= \sum a^{(1)}_{\varepsilon_1} \ldots a^{(n)}_{\varepsilon_n}$, where $a^{(i)}_+ := S(x^{(i)})_{\beta_{n-i}}$ and $a^{(i)}_- := S(x^{(i)})_{\beta_{n-i}} - (x^{(i)})_{\beta_{n-i}^{-1}}$. One has the equality
$$ x_{R^{-1}} - \epsilon(x) = \sum_{\varepsilon \in \{-, + \}^n} (-1)^{|\varepsilon|} b_{\varepsilon} - \epsilon(x).$$
If there exists an index $i$ such that $\varepsilon_i = -$, then $b_{\varepsilon} \in \mathcal{I}_{\beta_i \star \beta_i^{-1}}$, by Lemma \ref{lemma_injectivity1}. If $\varepsilon_i=+$ for every index $i$, then $b_{(+, \ldots, +)} -\epsilon(x) = S(x)_R - \epsilon(x) \in \mathcal{I}_R$. This proves the third assertion. By definition of being a set of relations, the three first assertions imply the inclusion  $\mathcal{I}_R \subset \mathcal{I}_{\mathbb{P}}$ for any $R\in \Rel_{\mathbb{G}}$. To prove the last assertion, consider the relation $R':= (\alpha_1\ldots \alpha_n)\star \alpha_n^{-1} \star \ldots \star \alpha_1^{-1}$. One has the following congruences
\begin{eqnarray*}
x_{\alpha_1 \ldots \alpha_n}- \sum (x^{(1)})_{\alpha_1} \ldots (x^{(n)})_{\alpha_n} 
&\equiv&  \sum (x^{(1)})_{\alpha_1\ldots \alpha_n} S(x^{(n)})_{\alpha_n}\ldots S(x^{(1)})_{\alpha_1} -\epsilon(x) \pmod{\mathcal{I}_{\mathbb{P}}} \\
&\equiv & \sum (x^{(1)})_{\alpha_1\ldots \alpha_n} (x^{(n)})_{\alpha_n^{-1}} \ldots (x^{(1)})_{\beta_1^{-1}} - \epsilon(x) \pmod{\mathcal{I}_{\mathbb{P}}} \\
&\equiv&  x_{R'} - \epsilon(x) \equiv 0 \pmod{\mathcal{I}_{\mathbb{P}}} 
\end{eqnarray*}
This proves the last assertion and completes the proof.

\end{proof}

\begin{proof}[Proof of Proposition \ref{proposition_discrete_model}] Since the algebra $\mathbb{C}[\mathcal{X}_G(\mathbf{\Sigma})]$ is generated by its curve functions by Proposition \ref{prop_holonomy_functions} and since the morphism $\phi^{\mathbb{P}}$ sends curve functions to curve functions, the morphism $\phi^{\mathbb{P}}$ is surjective by the definition of being a generating set.
To prove the injectivity, we need to show the inclusion $  \left( \mathcal{I}_{\Delta}+\mathcal{I}_{\epsilon} \right) \cap \widetilde{\phi}^{\mathbb{P}} \left( \mathbb{C}[G]^{\otimes \mathbb{G}} \right) \subset \mathcal{I}_{\mathbb{P}}$. The algebra $\mathcal{I}_{\epsilon} \cap \widetilde{\phi}^{\mathbb{P}} \left( \mathbb{C}[G]^{\otimes \mathbb{G}} \right)$ is generated by elements of the form $x_R -\epsilon(x)$ for $R \in \Rel_{\mathbb{G}}$ and $x\in \mathbb{C}[G]$, hence it is included in $\mathcal{I}_{\mathbb{P}}$ by Lemma \ref{lemma_injectivity2}. The algebra $\mathcal{I}_{\Delta} \cap \widetilde{\phi}^{\mathbb{P}} \left( \mathbb{C}[G]^{\otimes \mathbb{G}} \right)$ is generated by elements of the form
$x_{\alpha_1 \ldots \alpha_n} - \sum (x^{(1)})_{\alpha_1} \ldots (x^{(n)})_{\alpha_n}$ where $\alpha= \alpha_1 \ldots \alpha_n$ and the paths $\alpha_i$ belong to $\mathbb{G}$ and 
$x\in \mathbb{C}[G]$. By the last assertion of Lemma  \ref{lemma_injectivity2}, this algebra is included in  $\mathcal{I}_{\mathbb{P}}$. This proves the injectivity of $\phi^{\mathbb{P}}$ and concludes the proof.

\end{proof}

\begin{proof}[Proof of Corollary \ref{coro_discrete_model}] The first and second assertions follow using for $\mathbb{P}$ the presentations of the third and first item of Example \ref{example_finite_presentations} respectively and using the fact that the representation scheme $\Hom(\pi_1(\Sigma, v), G)$ is reduced whenever either $G=\GL_N, \SL_N$ or $\Sigma$ is open (see \cite{Sikora}).
\end{proof}

\subsection{The gluing formula}

\begin{definition} Let $\mathbf{\Sigma}$ be a punctured surface and $a$ a boundary arc.
 We define a left Hopf comodule $\widetilde{\Delta}_a^{L} :  \mathbb{C}[\Map(\Pi_1(\Sigma), G)]\rightarrow \mathbb{C}[G] \otimes  \mathbb{C}[\Map(\Pi_1(\Sigma), G)]$ and a right Hopf comodule $\widetilde{\Delta}_a^R :  \mathbb{C}[\Map(\Pi_1(\Sigma), G)] \rightarrow  \mathbb{C}[\Map(\Pi_1(\Sigma), G)] \otimes \mathbb{C}[G] $ by the formulas: 

\begin{align*}
&\widetilde{\Delta}_a^L (x_{\alpha}) := \left\{ \begin{array}{ll}
\sum x' S(x''') \otimes x''_{\alpha} & \mbox{, if }s(\alpha), t(\alpha) \in a; \\
\sum x' \otimes x''_{\alpha} & \mbox{, if } s(\alpha) \in a, t(\alpha)\notin a; \\
\sum S(x'') \otimes x'_{\alpha} & \mbox{, if }s(\alpha) \notin a, t(\alpha) \in a; \\
1\otimes x_{\alpha} & \mbox{, if }s(\alpha), t(\alpha) \notin a.
\end{array}\right.
\\
&\widetilde{\Delta}_a^R:= \sigma \circ (S\otimes \id) \circ \widetilde{\Delta}_a^L.
\end{align*}
\end{definition}

\par It follows from the axioms of cocommutativity and compatibility of the coproduct with the counit in the Hopf algebra $\mathbb{C}[G]$, that the Hopf comodules $\widetilde{\Delta}_a^L$ and $\widetilde{\Delta}_a^R$ vanish on the ideal $\mathcal{I}_{\Delta}+\mathcal{I}_{\epsilon}$, hence induce some Hopf comodules $\Delta_a^{L} : \mathbb{C}[\mathcal{R}_G(\mathbf{\Sigma})] \rightarrow \mathbb{C}[G] \otimes \mathbb{C}[\mathcal{R}_G(\mathbf{\Sigma})]$ and  $\Delta_a^R : \mathbb{C}[\mathcal{R}_G(\mathbf{\Sigma})] \rightarrow \mathbb{C}[\mathcal{R}_G(\mathbf{\Sigma})] \otimes \mathbb{C}[G] $ by passing to the quotient.
The Hopf comodules $\Delta_a^L, \Delta_a^R$ are equivariant for the gauge group action, hence induce, by restriction, Hopf comodules (still denoted by the same letter) $\Delta_a^{L} : \mathbb{C}[\mathcal{X}_G(\mathbf{\Sigma})] \rightarrow \mathbb{C}[G] \otimes \mathbb{C}[\mathcal{X}_G(\mathbf{\Sigma})]$ and  $\Delta_a^R : \mathbb{C}[\mathcal{X}_G(\mathbf{\Sigma})] \rightarrow \mathbb{C}[\mathcal{X}_G(\mathbf{\Sigma})] \otimes \mathbb{C}[G] $ .

\vspace{2mm}
\par
Now consider two (distinct) boundary arcs $a, b$ of $\mathbf{\Sigma}$ and $\mathbf{\Sigma}_{a\#b}$ the glued marked surface. Denote by $\pi : \Sigma_{\mathcal{P}} \rightarrow \Sigma_{a\#b}$ the natural projection and by $c$ the common image of $a$ and $b$ by $\pi$. Define an algebra morphism $i_{a\# b} : \mathbb{C}[\mathcal{R}_G(\mathbf{\Sigma}_{ a\#b})] \rightarrow \mathbb{C}[\mathcal{R}_G(\mathbf{\Sigma})]$ as follows. Let $\alpha \in \Pi_1(\Sigma_{a\#b})$ be a path and $c_{\alpha}: [0,1]\rightarrow \Sigma_{a\#b}$ be a geometric representative of $\alpha$ transversed to $c$. Choose a sequence $0=t_0 < t_1 < \ldots < t_n=1$ such that $c_{\alpha}( (t_i, t_{i+1}))$ does not intersect $c$. Each geometric arc ${c_{\alpha}}_{|[t_i, t_{i+1}]}$ induces a path $\alpha_i \in \Pi_1(\Sigma)$.

\begin{definition} The morphism $i_{a\# b} : \mathbb{C}[\mathcal{R}_G(\mathbf{\Sigma}_{| a\#b})] \rightarrow \mathbb{C}[\mathcal{R}_G(\mathbf{\Sigma})]$  is defined by:  
$$i_{a\#b}([x_{\alpha}]):= \sum [(x^{(1)})_{\alpha_1} \ldots (x^{(n)})_{\alpha_n}] \in \mathbb{C}[\mathcal{R}_G(\mathbf{\Sigma})],$$ 
where $x\in \mathbb{C}[G]$ and $\alpha$ a path transversed to $c$.
\end{definition}

If follows from the definitions of $\mathcal{I}_{\Delta}$ and $\mathcal{I}_{\epsilon}$ that the element $i_{a\#b}([x_{\alpha}])$ does not depend on the choice of a geometric representative of $\alpha$ nor on its decomposition and that the map  $i_{a\# b} : \mathbb{C}[\mathcal{R}_G(\mathbf{\Sigma}_{| a\#b})] \rightarrow \mathbb{C}[\mathcal{R}_G(\mathbf{\Sigma})]$ is an algebra morphism. The restriction $\widetilde{\pi} : \Sigma\setminus \mathcal{A} \rightarrow \Sigma_{a\#b} \setminus (\mathcal{A}_{a\#b} \cup c)$ of the projection $\pi$ is a homeomorphism. We define a Hopf algebra morphism $\phi_{a\#b}^{\mathcal{G}} : \mathbb{C}[\mathcal{G}_{\mathbf{\Sigma}_{a\#b}}] \rightarrow \mathbb{C}[\mathcal{G}_{\mathbf{\Sigma}}]$ by sending a generator $x_{v} \in \mathbb{C}[\mathcal{G}_{\mathbf{\Sigma}_{a\#b}}]$ to the generator $x_{\widetilde{\pi}^{-1}(v)} \in \mathbb{C}[\mathcal{G}_{\mathbf{\Sigma}}]$ if $v\notin c$ and to the element $\epsilon(x)$ if $v\in c$. The morphism $i_{a\# b} : \mathbb{C}[\mathcal{R}_G(\mathbf{\Sigma}_{a\#b})] \rightarrow \mathbb{C}[\mathcal{R}_G(\mathbf{\Sigma})]$ is $\phi_{a\#b}^{\mathcal{G}}$-equivariant, hence induces by restriction a morphism (still denoted by the same symbol): 
$$i_{a\# b} : \mathbb{C}[\mathcal{X}_G(\mathbf{\Sigma}_{a\#b})] \rightarrow \mathbb{C}[\mathcal{X}_G(\mathbf{\Sigma})].$$ 

\vspace{2mm}
\par As explained in the introduction, the main motivation to introduce our generalization of character varieties lies in the following gluing formula.
\begin{proposition}[Fundamental gluing property]\label{gluing_formula}
The following sequence is exact
$$ 0 \rightarrow \mathbb{C}[\mathcal{X}_G(\mathbf{\Sigma}_{a\#b})] \xrightarrow{i_{a\#b}} \mathbb{C}[\mathcal{X}_G(\mathbf{\Sigma})] \xrightarrow{ \Delta_a^L - \sigma \circ \Delta_b^R}   \mathbb{C}[G] \otimes  \mathbb{C}[\mathcal{X}_G(\mathbf{\Sigma})], $$
where $\sigma(x\otimes y) = y\otimes x$. Moreover, if $a,b,c,d$ are four distinct boundary arcs, one has $i_{a\#b}\circ i_{c\#d}=i_{c\#d}\circ i_{a\#b}$.
\end{proposition}
\par We first state a technical lemma.
\begin{lemma}\label{lemma_gluing} Let $\mathcal{H}$ be a Hopf algebra and $\Delta_L : \mathcal{H}^{\otimes 2} \rightarrow \mathcal{H}^{\otimes 3}$ be the morphism defined by $\Delta_L(x\otimes y \otimes z) := \sum x'\otimes S(x'')y' \otimes y''$. Denote by $\Psi_1, \Psi_2: \mathcal{H}^{\otimes 2} \rightarrow \mathcal{H}^{\otimes 3}$ the morphisms defined by $\Psi_1 := \id\otimes \Delta - \Delta\otimes \id$ and $\Psi_2 := \Delta_L - \id\otimes \eta \otimes \id$ respectively. Then $\ker(\Psi_1)=\ker(\Psi_2)= \Ima(\Delta)$.
\end{lemma}

\begin{proof} Consider the automorphism $\varphi : \mathcal{H}^{\otimes 3} \rightarrow \mathcal{H}^{\otimes 3}$ defined by $\varphi(x\otimes y \otimes z)= -\sum x'\otimes S(x'')y\otimes z$. A straightforward computation shows that $\varphi$ has the inverse defined by $\varphi^{-1}(x\otimes y \otimes z) = -\sum x'\otimes x'' y \otimes z$ and that $\varphi\circ \Psi_1 = \Psi_2$. Hence we have the equality $\ker(\Psi_1)=\ker(\Psi_2)$. The inclusion $\ker(\Delta) \subset \ker(\Psi_1)$ follows from the co-associativity of the co-product. It remains to show the inclusion $\ker(\Psi_2)\subset \Ima(\Delta)$ to conclude. Let $X= \sum x_i \otimes y_i \in \ker(\Psi_2)$ and define $Y:= \sum \eta \circ\epsilon(x_i)y_i \in \mathcal{H}$. One has
\begin{eqnarray*}
\Psi_2(X) = 0 &\iff & \sum x_i' \otimes S(x_i'')y_i'\otimes y_i'' = \sum x_i\otimes 1 \otimes y_i
 \\  &\implies & \sum \eta\circ \epsilon (x_i)y_i' \otimes y_i'' = \sum x_i\otimes y_i \iff \Delta(Y)=X 
 \end{eqnarray*}
 where we passed from the first line to the second by composing the equalities with $\mu \otimes \id$. We thus have proved  that $\ker(\Psi_1)=\ker(\Psi_2)=\Ima(\Delta)$. This concludes the proof.
\end{proof}
\par Consider a punctured surface $\mathbf{\Sigma}$ with two boundary arcs $a$ and $b$ and a finite presentation $\mathbb{P}=(\mathbb{V}, \mathbb{G}, \mathbb{RL})$ of the fundamental groupoid $\Pi_1(\Sigma)$ such that $\mathbb{V}\cap a$ and $\mathbb{V}\cap b$ have  cardinal one. Denote by $\pi: \Sigma_{\mathcal{P}} \rightarrow {\Sigma_{|a\# b}}_{\mathcal{P}_{a\#b}}$ the projection map. Define the presentation $\mathbb{P}_{a\#b}$ of $\Pi_1\left(\Sigma_{a\# b}\right)$ by setting $\mathbb{V}_{a\#b} = \pi (\mathbb{V})$, $\mathbb{G}_{a\#b} = \pi_* (\mathbb{G})$ and $\mathbb{RL}_{a\#b} = \pi_*(\mathbb{RL})$. Denote by $\Delta_a^L : \mathbb{C}[\mathcal{X}_G(\mathbf{\Sigma}, \mathbb{P})] \rightarrow \mathbb{C}[G] \otimes \mathbb{C}[\mathcal{X}_G(\mathbf{\Sigma}, \mathbb{P})]$ and $\Delta_b^R : \mathbb{C}[\mathcal{X}_G(\mathbf{\Sigma}, \mathbb{P})] \rightarrow  \mathbb{C}[\mathcal{X}_G(\mathbf{\Sigma}, \mathbb{P})]\otimes  \mathbb{C}[G]$ the Hopf comodule maps induced by the isomorphism $\phi^{\mathbb{P}} : \mathbb{C}[\mathcal{X}_G(\mathbf{\Sigma}, \mathbb{P})] \cong \mathbb{C}[\mathcal{X}_G(\mathbf{\Sigma})] $. Also denote by $i_{a\#b}: \mathbb{C}[\mathcal{X}_G(\mathbf{\Sigma}_{a\#b}, \mathbb{P}_{a\#b})]\rightarrow \mathbb{C}[\mathcal{X}_G(\mathbf{\Sigma}, \mathbb{P})]$ and $\phi_{a\#b}^{\mathcal{G}_{\mathbb{P}}}: \mathbb{C}[\mathcal{G}_{\mathbb{P}_{a\#b}}] \rightarrow \mathbb{C}[\mathcal{G}_{\mathbb{P}}] $ the morphisms induced by $\phi^{\mathbb{P}}$, $\phi^{\mathbb{P}_{a\#b}}$ and the inclusion $\mathring{\mathbb{V}}\hookrightarrow \mathring{\mathbb{V}}_{a\#b}$.
\vspace{2mm}
\par By Proposition \ref{proposition_discrete_model}, to prove Proposition \ref{gluing_formula} it is sufficient to find a finite presentation $\mathbb{P}$ such that the following sequence is exact
$$ 0 \rightarrow \mathbb{C}[\mathcal{X}_G(\mathbf{\Sigma}_{a\#b}, \mathbb{P}_{a\#b})] \xrightarrow{i_{a\#b}} \mathbb{C}[\mathcal{X}_G(\mathbf{\Sigma}, \mathbb{P})] \xrightarrow{\Delta_a^L - \sigma \circ  \Delta_b^R}   \mathbb{C}[G] \otimes  \mathbb{C}[\mathcal{X}_G(\mathbf{\Sigma}, \mathbb{P})].$$

\begin{proof}[Proof of Proposition \ref{gluing_formula}]
Fix a finite presentation $\mathbb{P}$ of $\Pi_1(\Sigma_{\mathcal{P}})$ such that: 
\begin{enumerate}
\item Both $\mathbb{V}\cap a =\{v_a \}$ and $\mathbb{V}\cap b =\{v_b \}$ are singletons. 
\item There exists some path $\beta_a : v_a \rightarrow \mathring{v}$ and $\beta_b : \mathring{v}' \rightarrow v_b$ in $\mathbb{G}$ such that $\mathring{v}, \mathring{v}' \in \mathring{\mathbb{V}}$.
\item If $\alpha \in \mathbb{B} \setminus \{ \beta_a^{\pm 1}, \beta_b^{\pm 1} \}$, then $\{ v_a, v_b \} \cap \{s(\alpha), t(\alpha) \} = \emptyset$. Moreover $\mathbb{RL}$ does not contain non trivial relations involving the paths $\beta_a^{\pm 1}, \beta_b^{\pm 1}$.
\end{enumerate}
\par Fix also a gluing map $\varphi : \overline{a} \cong \overline{b}$ sending $v_a$ to $v_b$ and denote by $\pi: \Sigma_{\mathcal{P}} \rightarrow \Sigma_{a\# b}$ the projection map, by $c$ the image of $a$ and $b$ by $\pi$ and by $v_c$ the image of $v_a$ and $v_b$. For simplicity, using the projection $\pi$, we will identify the sets $\mathbb{G} = \mathbb{G}_{a\#b}$, the sets $\mathbb{R}=\mathbb{R}_{a\#b}$ and write $\mathring{\mathbb{V}}_{a\#b} = \mathring{\mathbb{V}}\cup \{v_c\}$. Under these identifications, we have the equalities $\mathcal{R}_G(\mathbf{\Sigma}_{|a\#b}, \mathbb{P}_{a\#b}) = \mathcal{R}_G(\mathbf{\Sigma}, \mathbb{P})$ and $\mathcal{G}_{\mathbb{P}_{a\#b}} = \mathcal{G}_{\mathbb{P}}\times G_{v_c}$. 

\vspace{2mm} \par
 Under the identification $\mathbb{C}[\mathcal{G}_{\mathbb{P}_{a\#b}}] = \mathbb{C}[\mathcal{G}_{\mathbb{P}}] \otimes \mathbb{C}[G]_{v_c}$ the comodule map $\Delta_{\mathcal{G}_{\mathbb{P}_{a\#b}}}^L : \mathbb{C}[\mathcal{R}_G(\mathbf{\Sigma}, \mathbb{P})] \rightarrow \left( \mathbb{C}[\mathcal{G}_{\mathbb{P}}] \otimes \mathbb{C}[G]_{v_c} \right) \otimes  \mathbb{C}[\mathcal{R}_G(\mathbf{\Sigma}, \mathbb{P})]$ decomposes as $\Delta_{\mathcal{G}_{\mathbb{P}_{a\#b}}}^L(f) = (\Delta^L_{\mathcal{G}_{\mathcal{P}}} (f))_{13}\cdot (\Delta^L_{v_c} (f))_{23}$, where $\Delta^L_{v_c} : \mathbb{C}[\mathcal{X}_G(\mathbf{\Sigma}, \mathbb{P})] \rightarrow \mathbb{C}[G]_{v_c} \otimes \mathbb{C}[\mathcal{X}_G(\mathbf{\Sigma}, \mathbb{P})]$ is the Hopf co-action induced at the point $v_c$.   Hence one has the equality $\ker \left( \Delta_{\mathcal{G}_{\mathbb{P}_{a\#b}}}^L-\eta\otimes \eta \otimes \id\right) = \ker \left( \Delta_{\mathcal{G}_{\mathcal{P}}} -\eta \otimes \id \right) \cap \ker \left(\Delta^L_{v_c}-\eta\otimes \id \right)$, and one has the following exact sequence
 
$$ 0 \rightarrow \mathbb{C}[\mathcal{X}_G(\mathbf{\Sigma}_{a\#b}, \mathbb{P}_{a\#b})] \xrightarrow{i_{a\#b}} \mathbb{C}[\mathcal{X}_G(\mathbf{\Sigma}, \mathbb{P})] \xrightarrow{\Delta^L_{v_c} - \eta\otimes \id} \mathbb{C}[G]_{v_c} \otimes \mathbb{C}[\mathcal{X}_G(\mathbf{\Sigma}, \mathbb{P})]. $$

\par We need to show the equality $\ker\left( \Delta^L_{v_c} - \eta\otimes \id \right) = \ker \left( \Delta_a^L - \sigma \circ \Delta_b^R \right)$ to conclude the proof. Recall that the algebra of regular functions of the representation variety is defined as the coimage:
$$  \mathbb{C}[G]^{\otimes \mathbb{RL}} \xrightarrow{\mathcal{R}^* - \epsilon^{\otimes \mathbb{RL}}} \mathbb{C}[G]^{\otimes \mathbb{G}} \rightarrow \mathbb{C}[\mathcal{R}_G(\mathbf{\Sigma}, \mathbb{P})] \rightarrow 0$$
\par Since $\mathbb{RL}$ does not contain any non trivial relation involving $\beta_a^{\pm 1}, \beta_b^{\pm 1}$, we have the tensor decomposition $\mathbb{C}[\mathcal{R}_G(\mathbf{\Sigma}, \mathbb{P})] = \mathbb{C}[G]_{\beta_a} \otimes \mathbb{C}[G]_{\beta_b} \otimes \mathcal{A}$, where  $\mathbb{C}[G]_{\beta_a}$ is the image in the quotient of the factor $\mathbb{C}[G]_{\beta_a} \otimes \mathbb{C}[G]_{\beta_a^{-1}}$, $\mathbb{C}[G]_{\beta_b}$ is the image of the factor $\mathbb{C}[G]_{\beta_b}\otimes \mathbb{C}[G]_{\beta_b^{-1}}$ and $\mathcal{A}$ is the image of the factor $\otimes_{\alpha \in \mathbb{G}\setminus \{\beta_a^{\pm 1}, \beta_b^{\pm 1} \}} \mathbb{C}[G]_{\alpha}$. Denote by $\Psi_1, \Psi_2 : \mathbb{C}[G]_{\beta_a}\otimes \mathbb{C}[G]_{\beta_b} \rightarrow \mathbb{C}[G]_{\beta_a} \otimes \mathbb{C}[G]_{v_c} \otimes \mathbb{C}[G]_{\beta_b}$ defined, as in Lemma \ref{lemma_gluing}, by the formulas $\Psi_1 := \id\otimes \Delta - \Delta\otimes \id$ and $\Psi_2 := \Delta_L - \id\otimes \eta \otimes \id$ respectively. Also define the permutation isomorphism $P :  \mathbb{C}[G]_{v_c} \otimes \mathbb{C}[G]_{\beta_a} \otimes \mathbb{C}[G]_{\beta_b} \cong  \mathbb{C}[G]_{\beta_a} \otimes \mathbb{C}[G]_{v_c} \otimes \mathbb{C}[G]_{\beta_b}$ defined by $P(x\otimes y \otimes z) := y\otimes x \otimes z$. 
\vspace{2mm}
\par By definition of the comodule maps $\Delta_a^L$ and $\Delta_b^R$, the following diagram commutes:
$$\begin{tikzcd}
\mathbb{C}[G]_{\beta_a} \otimes \mathbb{C}[G]_{\beta_b}\otimes \mathcal{A}
\arrow[rr, "\Delta_a^L - \sigma \circ \Delta_b^R"] 
\arrow[rrd,  "\Psi_1 \otimes \id_{\mathcal{A}}"']
&&
 \mathbb{C}[G]_{v_c} \otimes \left(\mathbb{C}[G]_{\beta_a} \otimes \mathbb{C}[G]_{\beta_b}\otimes \mathcal{A}\right) 
 \arrow[d, "P\otimes \id_{\mathcal{A}}", "\cong"'] 
 \\ {} &&
 \mathbb{C}[G]_{\beta_a} \otimes \mathbb{C}[G]_{v_c} \otimes \mathbb{C}[G]_{\beta_b} \otimes \mathcal{A}
\end{tikzcd}$$
\par Moreover, by definition of the gauge group action, the following diagram commutes:
$$\begin{tikzcd}
\mathbb{C}[G]_{\beta_a} \otimes \mathbb{C}[G]_{\beta_b}\otimes \mathcal{A}
\arrow[rr, "\Delta_{v_c}^L - \eta\otimes \id"] 
\arrow[rrd,  "\Psi_2 \otimes \id_{\mathcal{A}}"']
&&
 \mathbb{C}[G]_{v_c} \otimes \left(\mathbb{C}[G]_{\beta_a} \otimes \mathbb{C}[G]_{\beta_b}\otimes \mathcal{A}\right) 
 \arrow[d, "P\otimes \id_{\mathcal{A}}", "\cong"'] 
 \\ {} &&
 \mathbb{C}[G]_{\beta_a} \otimes \mathbb{C}[G]_{v_c} \otimes \mathbb{C}[G]_{\beta_b} \otimes \mathcal{A}
\end{tikzcd}$$

\par Hence, by Lemma \ref{lemma_gluing},  we have the equalities:
\begin{eqnarray*}
\mathbb{C}[\mathcal{X}_G(\mathbf{\Sigma}_{a\#b}, \mathbb{P}_{a\#b})] &=& \mathbb{C}[\mathcal{X}_G(\mathbf{\Sigma}, \mathbb{P})] \cap \ker \left( \Delta_{v_c}^L - \eta \otimes \id\right) \\
&=& \mathbb{C}[\mathcal{X}_G(\mathbf{\Sigma}, \mathbb{P})] \cap \ker \left(\Psi_1\otimes \id_{\mathcal{A}} \right) \\
&=& \mathbb{C}[\mathcal{X}_G(\mathbf{\Sigma}, \mathbb{P})] \cap \ker \left(\Psi_2\otimes \id_{\mathcal{A}} \right) \\
&=& \mathbb{C}[\mathcal{X}_G(\mathbf{\Sigma}, \mathbb{P})] \cap \ker \left( \Delta_a^L - \sigma \circ \Delta_b^R \right) 
\end{eqnarray*}
\par We thus have proved that following sequence is exact
$$ 0 \rightarrow \mathbb{C}[\mathcal{X}_G(\mathbf{\Sigma}_{a\#b}, \mathbb{P}_{a\#b})] \xrightarrow{i_{a\#b}} \mathbb{C}[\mathcal{X}_G(\mathbf{\Sigma}, \mathbb{P})] \xrightarrow{\Delta_a^L - \sigma \circ  \Delta_b^R}   \mathbb{C}[G] \otimes  \mathbb{C}[\mathcal{X}_G(\mathbf{\Sigma}, \mathbb{P})]. $$
\par We conclude using Proposition \ref{proposition_discrete_model}.

\end{proof}

\subsection{Triangular decompositions}\label{sec_triangulation_modularoperad}

\begin{definition}\label{def_triangulation}
A marked surface $\mathbf{\Sigma}$ is \textit{triangulable} if it can be obtained from a disjoint union $\mathbf{\Sigma}^{\Delta}:=\bigsqcup_{\mathbb{T}\in F(\Delta)} \mathbb{T}$ of triangles by gluing some pairs of boundary arcs. A \textit{triangulation} $\Delta$ is the  data of the disjoint union $\mathbf{\Sigma}^{\Delta}:=\bigsqcup_{\mathbb{T}\in F(\Delta)} \mathbb{T}$ together with the pairs of boundary arcs glued together.
\end{definition}

The only non triangulable connected marked surfaces are:  the (unmarked) closed connected surfaces, the unmarked sphere with one or two boundary component, the disc with one or two boundary arcs.  

\begin{remark}\label{remark_punctured_surfaces}
 The groupoid of marked surfaces with isomorphisms is equivalent to a category of punctured surfaces so both language can be used interchangeably, though a triangulation is easier to visualise using punctured surfaces. A punctured surface is a pair $(S, \mathcal{P})$ where $S$ is a compact oriented surface and $\mathcal{P}\subset S$ a finite subset of punctures which intersects non-trivially each connected component of $\partial S$. Isomorphisms of punctured surfaces are preserving-orientation homeomorphisms which preserves the sets of punctures. One associates a punctured surface $(S, \mathcal{P})$ to a marked surface $(\Sigma, \mathcal{A})$ by shrinking down to a puncture each connected component of $\partial \Sigma\setminus \mathcal{A}$. The reverse operation consists in blowing up each inner puncture and setting $\mathcal{A}=\partial S \setminus \mathcal{P}$. A $\mathcal{P}$-arc in $(S,\mathcal{P})$ is an immersion $e : [0,1]\to S$ sending $0, 1$ to elements of $\mathcal{P}$ and whose restriction to $(0,1)$ is an embedding into $S\setminus \mathcal{P}$. An ideal triangulation of $(S,\mathcal{P})$ is a maximal set of pairwise non homotopic (relatively to their endpoints) $\mathcal{P}$-arcs (the edges) with disjoint interior. 
A triangulation of $(\Sigma, \mathcal{A})$ is the same as an ideal triangulation of its associated punctured surface $(S, \mathcal{P})$. 
\end{remark}

\par  Consider a triangulated marked surface $(\mathbf{\Sigma}, \Delta)$, so $\mathbf{\Sigma}$ is obtained from  $\mathbf{\Sigma}^{\Delta}:=\bigsqcup_{\mathbb{T}\in F(\Delta)} \mathbb{T}$ by gluing the triangles along pairs of edges. Each inner edge $e\in \mathcal{E}(\Delta)$ lifts to two boundary arcs $e'$ and $e''$ of $\mathbf{\Sigma}^{\Delta}$. By composing the morphisms $i_{e'\#e''}$, one obtains an injective morphism $i^{\Delta} : \mathbb{C}[\mathcal{X}_G(\mathbf{\Sigma})]\hookrightarrow \otimes_{\mathbb{T}\in F(\Delta)}\mathbb{C}[\mathcal{X}_G(\mathbb{T})]$. The comodule maps $\Delta_{e'}^L$ and $\Delta_{e''}^R$ induce comodule maps $\Delta^L$ and $\Delta^R$ such that we have the following exact sequence

\begin{equation*}
0 \rightarrow \mathbb{C}[\mathcal{X}_G(\mathbf{\Sigma})]  \xrightarrow{i^{\Delta}} \otimes_{\mathbb{T}\in F(\Delta)} \mathbb{C}[\mathcal{X}_G(\mathbb{T})]
 \xrightarrow{\Delta^L -  \sigma \circ \Delta^R} 
\left(\otimes_{e\in \mathring{\mathcal{E}}(\Delta)}\mathbb{C}[G]\right) \otimes  \left( \otimes_{\mathbb{T}\in F(\Delta)} \mathbb{C}[\mathcal{X}_G(\mathbb{T})] \right).
 \end{equation*}

\vspace{2mm}
\par The short exact sequence of Proposition \ref{gluing_formula} can be reformulated as follows. If $A$ is an algebra and $M$ a $A$-bimodule, the $0$-th Hochschild homology group is defined by $\mathrm{HH}_0(A, M):= \quotient{M}{(a\cdot m - m\cdot a, a \in A, m\in M)}$. Denoting by $\nabla^L : A\otimes M \rightarrow M$ and $\nabla^R: M\otimes A \rightarrow M$ the left and right module maps, the algebra $\mathrm{HH}_0(A,M)$ is thus defined by the coimage in the following exact sequence

$$  A \otimes M \xrightarrow{ \nabla^L - \sigma \circ \nabla^R  } M \rightarrow \mathrm{HH}_0 (A,M) \rightarrow 0.$$

\par Now consider a co-algebra $C$ with a bicomodule $M$ defined by the comodules maps $\Delta^L : M \rightarrow C\otimes M$ and $\Delta^R : M \rightarrow M\otimes C$. By dualizing the preceding exact sequence, it is natural to define the $0$-th coHochschild cohomology group $\mathrm{coHH}^0(C,M)$  as the kernel in the following exact sequence
$$ 0 \rightarrow \mathrm{coHH}^0(C,M) \rightarrow M  \xrightarrow{\Delta^L  - \sigma \circ\Delta^R} C \otimes M. $$
\par Denote by  ${}_a \mathbb{C}[\mathcal{X}_G(\mathbf{\Sigma})]_b$ the $\mathbb{C}[G]$  bi-comodule defined by $\Delta_a^L$ and $\Delta_b^R$. Proposition \ref{gluing_formula} can be re-written more elegantly by the formula
$$ \mathbb{C}[\mathcal{X}_G(\mathbf{\Sigma}_{|a\#b})] = \mathrm{coHH}^0(\mathbb{C}[G], {}_a \mathbb{C}[\mathcal{X}_G(\mathbf{\Sigma})]_b ). $$


\section{Twisted groupoid (co)homologies and the tangent spaces}\label{sec_cohomology}


\begin{notations} Denote by $\mathfrak{g}$ the Lie algebra of the complex affine reducible Lie group $G$. We fix once and for all a non-degenerate symmetric $G$-invariant pairing $\left(\cdot, \cdot \right) : \mathfrak{g}^{\otimes 2} \rightarrow \mathbb{C}$. If $g\in G$, denote by $L_g :G\rightarrow G$ and $R_g:G\rightarrow G$ the regular maps defined by $L_g(h)= gh$ and $R_g(h)=hg$. For a tangent vector  $X\in T_h G$ and a group element $g$, we will use the notations $gX:= D_h L_g (X)\in T_{gh}G$ and $Xg:= D_h R_g(X)\in T_{hg}G$. If $g\in G$, we define the symmetric non-degenerate pairing $\left(\cdot, \cdot \right)_g : T_g G \otimes T_g G \rightarrow \mathbb{C}$ by the formula $(X, Y)_g:= (g^{-1}X, g^{-1}Y)$.
\end{notations}

\subsection{Twisted groupoid (co)homologies}

\par Let $\mathbf{\Sigma}$ be a punctured surface and $\rho \in \mathcal{R}_G(\mathbf{\Sigma})$ a representation. Denote by $P^{(n)}$ the set of $n+1$-tuples $(\alpha_n, \ldots, \alpha_0)$ of elements of $\Pi_1(\Sigma)$ such that $t(\alpha_i)=s(\alpha_{i-1})$. Define the vector space $\mathrm{C}_n (\Sigma ; \rho)$ as the  quotient
$$ \mathrm{C}_n (\Sigma ; \rho):= \quotient{\left(\oplus_{(\alpha_n, \ldots, \alpha_0) \in P^{(n)}} T_{\rho(\alpha_n\ldots \alpha_0)}G \right)}{\sim},$$
where the equivalence relation $\sim$ is defined for any $(\alpha_n, \ldots, \alpha_0, \beta)\in P^{(n+1)}$ and $X\in T_{\rho(\alpha_n \ldots \alpha_0)}G$ by the formula $X \sim X\rho(\beta)$, where $X\rho(\beta) \in  T_{\rho(\alpha_n \ldots \alpha_0\beta)}G$. Given $(\alpha_n, \ldots, \alpha_0)\in P^{(n)}$ and $X\in T_{\rho(\alpha_n \ldots \alpha_0)}G$, we denote by $\left<(\alpha_n, \ldots, \alpha_0), X\right> \in  \mathrm{C}_n (\Sigma ; \rho)$ the class of the corresponding element. Define a map $\partial_n :  \mathrm{C}_n (\Sigma ; \rho) \rightarrow  \mathrm{C}_{n-1} (\Sigma ; \rho)$ by the formula
\begin{equation*}
 \partial_n\left( \left< (\alpha_n, \ldots, \alpha_0), X\right> \right) := \left< (\alpha_{n-1}, \ldots, \alpha_0), \rho(\alpha_n)^{-1}X \right>  + \sum_{k=1}^n (-1)^{k+n+1}\left< (\alpha_n, \ldots, \alpha_k \alpha_{k-1}, \ldots, \alpha_0), X  \right>.
  \end{equation*}
  \par A straightforward computation shows that $\partial_n \circ \partial_{n-1} = 0$, hence we have defined a chain complex $(\mathrm{C}_{\bullet}(\Sigma; \rho), \partial_{\bullet})$. Define the sub-complex $\mathrm{C}_{\bullet}(\mathcal{A}; \rho) \subset \mathrm{C}_{\bullet}(\Sigma; \rho)$ as the sub-space spanned by elements $\left<(\alpha_n, \ldots, \alpha_0), X\right>$ where the $\alpha_i\in \Pi_1(\mathcal{A})$.
  
  \begin{definition}
   The chain complex $(\mathrm{C}_{\bullet}(\Sigma, \mathcal{A}; \rho), \partial_{\bullet})$ is defined by setting  $\mathrm{C}_{n}(\Sigma, \mathcal{A}; \rho):= \quotient{ \mathrm{C}_{n}(\Sigma; \rho)}{\mathrm{C}_{n}(\mathcal{A}; \rho)}$ and by passing the boundary map to the quotient. We will denote by $\mathrm{H}_n(\Sigma, \mathcal{A}; \rho)$ its homology groups.
   \end{definition}
   
  \par Denote by $\mathrm{C}^n(\Sigma ; \rho)$ the vector space of maps $\sigma^n : P^{(n)} \rightarrow \oplus_{(\alpha_n, \ldots, \alpha_0) \in P^{(n)}}T_{\rho(\alpha_n\ldots \alpha_0)}G$ such that $\sigma^n (\alpha_n, \ldots, \alpha_0) \in T_{\rho(\alpha_n, \ldots, \alpha_0)}G$ and such that $\sigma^n (\alpha_n, \ldots, \alpha_0 \beta)= \sigma^n(\alpha_n, \ldots, \alpha_0) \rho(\beta)$ for all $(\alpha_n, \ldots, \alpha_0, \beta) \in P^{(n+1)}$. Define a map $d^n : \mathrm{C}^n(\Sigma ; \rho) \rightarrow \mathrm{C}^{n+1}(\Sigma ; \rho)$ by the formula
  $$ d^n c^n (\alpha_{n+1}, \ldots, \alpha_0):= \rho(\alpha_{n+1})c(\alpha_i, \ldots, \alpha_0) + \sum_{k=1}^n (-1)^{k+n}c(\alpha_{n+1}, \ldots, \alpha_k \alpha_{k-1}, \ldots, \alpha_0).$$
  A straightforward computation shows that $d^n \circ d^{n+1} = 0$, hence $\left( \mathrm{C}^{\bullet}(\Sigma ; \rho), d^{\bullet} \right)$ is a cochain complex. 
  
  \begin{definition}
  The complex $\left( \mathrm{C}^{\bullet}(\Sigma, \mathcal{A} ; \rho), d^{\bullet} \right)$ is defined as the sub-complex whose graded part  $\mathrm{C}^n(\Sigma, \mathcal{A} ; \rho) \subset \mathrm{C}^n(\Sigma ; \rho)$ consists in the maps $\sigma^n$ vanishing on the elements $(\alpha_n, \ldots, \alpha_0)$ where $\alpha_i\in \Pi_1(\mathcal{A})$. We denote by $\mathrm{H}^n(\Sigma, \mathcal{A}; \rho)$ its cohomology groups.
  \end{definition}

\begin{definition}
   We define a $0$-graded pairing $\left< \cdot, \cdot \right> : \mathrm{C}_{\bullet}(\Sigma, \mathcal{A}; \rho) \otimes \mathrm{C}^{\bullet}(\Sigma, \mathcal{A}; \rho) \rightarrow \mathbb{C}$ by the formula 
  $$\left< <(\alpha_n, \ldots, \alpha_{0}), X > , c^n \right> := \left(  X, c^n(\alpha_n, \ldots, \alpha_0) \right)_{\rho(\alpha_n \ldots \alpha_0)}.$$
 \end{definition} 
  
  Since the pairings $\left(\cdot, \cdot \right)_g : T_g G^{\otimes 2} \rightarrow \mathbb{C}$ are non-degenerate,  the pairing $\left< \cdot, \cdot \right> $ is also non-degenerate. If follows from the definitions that, for all $\sigma_{n+1} \in \mathrm{C}_{n+1}(\Sigma, \mathcal{A}; \rho)$ and for all $c^n \in \mathrm{C}^n(\Sigma, \mathcal{A}; \rho)$, one has the equality $\left< \partial_{n+1} \sigma_{n+1}, c^n \right> = \left< \sigma_{n+1}, d^n c^n \right>$. Hence the pairing $\left< \cdot, \cdot \right>$ induces a $0$-graded non-degenerate pairing in homology (still denoted by the same letter): 
  $$\left< \cdot, \cdot \right> : \mathrm{H}_{\bullet}(\Sigma, \mathcal{A}; \rho) \otimes \mathrm{H}^{\bullet}(\Sigma, \mathcal{A} ; \rho) \rightarrow \mathbb{C}.$$ 

\begin{notations} If $\alpha \in \Pi_1(\Sigma)$ and $X \in T_{\rho(\alpha)}G$, we will denote by $[\alpha, X] \in \mathrm{H}^1(\Sigma, \mathcal{A}; \rho)$ the class of the element $\left< (\alpha, 1_{t(\alpha)}), X\right>$, where $1_{t(\alpha)}$ represents the constant path based at $t(\alpha)$.
\end{notations}

\subsection{Comparison with standard twisted group (co)homology}

Suppose that $\mathbf{\Sigma}=(\Sigma, \emptyset)$ is a connected unmarked surface and fix a basepoint $v\in \Sigma$. Consider the universal covering 
$$ \widehat{\Sigma}= \{ \alpha : [0,1] \to \Sigma, t(\alpha)=v\}, \quad \pi: \widehat{\Sigma}\to \Sigma, \quad \pi: \alpha \mapsto s(\alpha).$$
$\pi_1(\Sigma, v)$ (right) acts on $\widehat{\Sigma}$ by $\alpha \cdot \gamma:= \alpha \gamma$. 
Let $\rho: \Pi_1(\Sigma)\to G$ and consider its restriction $\rho_{v}: \pi_1(\Sigma, v) \to G$. Then $\pi_1(\Sigma, v)$ acts on $\mathfrak{g}$ by $\gamma\cdot X:= \rho_v(\gamma)^{-1}X\rho_v(\gamma)$.
Consider the complexes
$$\mathrm{C}_{\bullet}(\Sigma, Ad_{\rho_v}):= \mathrm{C}_{\bullet}(\widehat{\Sigma}, \mathbb{Z})\otimes_{\mathbb{Z}[\pi_1(\Sigma, v)]} \mathfrak{g}, \quad \mathrm{C}^{\bullet}(\Sigma, Ad_{\rho_v}):= \Hom_{\mathbb{Z}[\pi_1(\Sigma, v)]}(\mathrm{C}_{\bullet}(\widehat{\Sigma}, \mathbb{Z}), \mathfrak{g}).$$

\begin{proposition}\label{prop_group_homology} The chain complexes $\mathrm{C}_{\bullet}(\Sigma ; Ad_{\rho_v})$ and $\mathrm{C}_{\bullet}(\Sigma; \rho)$ are homotopy equivalent. Similarly, the cochain complexes $\mathrm{C}^{\bullet}(\Sigma ; Ad_{\rho_v})$ and $\mathrm{C}^{\bullet}(\Sigma ; \rho)$ are homotopy equivalent.
\end{proposition}
 
 Therefore, the  complex $\mathrm{C}_{\bullet}(\Sigma; \rho)$ is a "basepoint free" analogue of $\mathrm{C}_{\bullet}(\Sigma; Ad_{\rho_v})$. As for $\mathcal{X}_G(\mathbf{\Sigma})$, the fact that we switch to a basepoint free object is what permits the gluing operation. 
 
\begin{proof}
The main idea is to associate to each $(\alpha_n, \ldots, \alpha_0)\in P^{(n)}$ with $t(\alpha_0)=v$ a singular chain $\sigma_{(\alpha_n, \ldots, \alpha_0)}: \Delta^n \to \widehat{\Sigma}$. Let $v_i \in \mathbb{R}^{n+1}$ be the point whose $j$ coordinate is $\delta_{ij}$ and consider the simplex $\Delta^n:= [v_0, \ldots, v_n] \subset \mathbb{R}^{n+1}$, where $[\cdot ]$ denotes the convex hull. By definition, a singular map is a continuous map $\sigma^n: \Delta^n \to \widehat{\Sigma}$ and $\mathrm{C}_n(\Sigma; Ad_{\rho_v})$ is spanned by elements $[\sigma^n \otimes X]$, with $X\in \mathfrak{g}$, modulo the relation $[\sigma^n \cdot \gamma \otimes X]=[\sigma^n \otimes \rho_v(\gamma)^{-1}X\rho_v(\gamma)]$ for $\gamma\in \pi_1(\Sigma , v)$. 
\par  Let $\beta_i:= [v_{n-i}, v_{n-i+1}] \subset \Delta^n$ and consider $\Lambda^n := \cup_{i=1}^n \beta_i \subset \Delta^n$. Let $\iota: \Lambda^n \hookrightarrow \Delta^n$ be the inclusion map and consider a retraction by deformation $r: \Delta^n \to \Lambda^n$ and a homotopy $h: [0,1]\times \Delta^n \to \Delta^n$ such that $h_0= \id$, $h_1= \iota \circ r$ and $\restriction{h_t}{\Lambda^n}= \id_{\Lambda^n}$ for all $t\in [0,1]$. To $(\alpha_n, \ldots, \alpha_0) \in P^{(n)}$ for which $t(\alpha_0)=v$, we associate a continuous map $\widetilde{\sigma}_{(\alpha_n, \ldots, \alpha_0)}: \Lambda^n \to \widehat{\Sigma}$ as follows. The map  $\widetilde{\sigma}_{(\alpha_n, \ldots, \alpha_0)}$ sends $v_{n-i}$ to the path $\alpha_i \alpha_{i-1}\ldots \alpha_0 \in \widehat{\Sigma}$. For $\tau \in [0,1]$ consider the path $\alpha_i^{\tau}: [0,1]\to \Sigma$, $\alpha_i^{\tau}(t):= \alpha_i (\tau t)$ (so $\alpha_i^0=1$ and $\alpha_i^1=\alpha_i$). The restriction of $\widetilde{\sigma}_{(\alpha_n, \ldots, \alpha_0)}$ to the edge $\beta_i=\{ \tau v_{n-i} +(1-\tau)v_{n-i+1}, \tau \in [0,1]\}$ is 
$$ \widetilde{\sigma}_{(\alpha_n, \ldots, \alpha_0)}(\tau v_{n-i} +(1-\tau)v_{n-i+1}):= \alpha_i^{\tau} \alpha_{i-1}\ldots \alpha_0.$$
The $n$-simplex ${\sigma}_{(\alpha_n, \ldots, \alpha_0)}: \Delta^n \to \widehat{\Sigma}$ is defined by ${\sigma}_{(\alpha_n, \ldots, \alpha_0)}:= \widetilde{\sigma}_{(\alpha_n, \ldots, \alpha_0)} \circ r$. Define $\psi_{\bullet}: \mathrm{C}_{\bullet}(\Sigma, \rho) \to \mathrm{C}_{\bullet}(\Sigma, Ad_{\rho_v})$ by the formula
$$ \Psi_n: \mathrm{C}_n(\Sigma; \rho) \to \mathrm{C}_n(\Sigma; Ad_{\rho_v}), \quad \Psi_n: \left<(\alpha_n, \ldots, \alpha_0), X\right> \mapsto [\sigma_{(\alpha_n, \ldots, \alpha_0)}\otimes \rho(\alpha_n\ldots \alpha_0)^{-1}X].$$
Note that if $\gamma \in \pi_1(\Sigma, v)$ then 
\begin{multline*} \Psi_n\left( \left< (\alpha_n, \ldots, \alpha_0 \gamma), X \rho(\gamma)\right> \right) = [\sigma_{(\alpha_n, \ldots, \alpha_0)}\cdot \gamma \otimes \rho(\gamma)^{-1} \rho(\alpha_n\ldots \alpha_0)^{-1} X \rho(\gamma) ]
\\ = [\sigma_{(\alpha_n, \ldots, \alpha_0)}\otimes \rho(\alpha_n \ldots \alpha_0)X] = \Psi_n\left( \left<(\alpha_n, \ldots, \alpha_0), X\right>\right).
\end{multline*}
Therefore $\Psi_n$ is well defined. To prove that $\Psi_n$ is a morphism of chain complexes, recall the boundary map in singular homology is given by 
$$\partial [\sigma^n \otimes X]= \sum_{i=0}^n (-1)^i [\restriction{\sigma^n}{[v_0, \ldots, \check{v}_i, \ldots, v_n]} \otimes X].$$
By analyzing the definition of $\sigma_{(\alpha_n, \ldots, \alpha_0)}$ we see that 
$$ \restriction{\sigma_{(\alpha_n, \ldots, \alpha_0)}}{[v_0, \ldots, \check{v}_i, \ldots, v_n]} = \left\{ \begin{array}{ll}
 \sigma_{(\alpha_{n-1}, \ldots, \alpha_0)} & \mbox{, if }i=0; \\
 \sigma_{(\alpha_n, \ldots, \alpha_{n-i+1}\alpha_{n-i}, \ldots, \alpha_0)} & \mbox{, if }i\geq 1.
 \end{array} \right.$$
Therefore
\begin{multline*}
\partial \circ \Psi_n \left( \left<(\alpha_n,\ldots, \alpha_0), X\right>\right) = \partial \left( [\sigma_{(\alpha_n, \ldots, \alpha_0)}\otimes \rho(\alpha_n \ldots \alpha_0)^{-1}X] \right) \\
= \sum_{i=0}^n (-1)^i [\restriction{\sigma_{(\alpha_n, \ldots, \alpha_0)}}{[v_0, \ldots, \check{v}_i, \ldots, v_n]} \otimes \rho(\alpha_n \ldots \alpha_0)^{-1}X] \\
= [\sigma_{\alpha_{n-1}, \ldots, \alpha_0} \otimes \rho(\alpha_n \ldots \alpha_0)^{-1}X] + \sum_{i=1}^n (-1)^i [ \sigma_{(\alpha_n, \ldots, \alpha_{n-i+1}\alpha_{n-i}, \ldots, \alpha_0)} \otimes \rho(\alpha_n \ldots \alpha_0)^{-1}X] \\
=\Psi_{n-1}\left( 
\left< (\alpha_{n-1}, \ldots, \alpha_0), \rho(\alpha_n)^{-1}X \right> + \sum_{j=1}^n (-1)^{j+1+n} \left<(\alpha_n, \ldots, \alpha_{j}\alpha_{j-1}, \ldots, \alpha_0), X \right>
 \right) 
 \\
  = \Psi_{n-1}\circ \partial  \left( \left<(\alpha_n,\ldots, \alpha_0), X\right>\right).
\end{multline*}
To prove that $\Psi_{\bullet}$ is a homotopy equivalence, let us construct a quasi-inverse $\Theta_{\bullet}: \mathrm{C}_{\bullet}(\Sigma; Ad_{\rho_v}) \to \mathrm{C}_{\bullet}(\Sigma; \rho)$. Consider a singular chain $\sigma^n: \Delta^n \to \widehat{\Sigma}$ and write $\widetilde{\sigma}^n:= \sigma^n \circ r : \Lambda^n \to \widehat{\Sigma}$. By parametrizing each arc $\beta_i$ of $\Lambda^n$, one defines  an element $\alpha_{\sigma}=(\alpha_n^{\sigma}, \ldots, \alpha_0^{\sigma})\in P^{(n)}$ such that $\sigma_{\alpha_{\sigma}}=\sigma$. Set 
$$ \Theta_n :  \mathrm{C}_{n}(\Sigma, Ad_{\rho_v}) \to \mathrm{C}_{n}(\Sigma, \rho), \quad \Theta_n: [\sigma \otimes X] \mapsto \left< \alpha_{\sigma}, \rho(\alpha_n^{\sigma} \ldots \alpha_0^{\sigma}) X\right>.$$
Clearly $\Theta_n \circ \Psi_n= \id$. Let us prove that $\Psi_{\bullet} \circ \Theta_{\bullet} \sim \id$. To $\sigma: \Delta^n \to \widehat{\Sigma}$, one can associate $h\circ(\sigma \times \mathds{1}): \Delta^n \times [0,1] \to \widehat{\Sigma}$. Imitating the construction of the Prism operator in the proof of \cite[Theorem $2.10$]{Hatcher_Book} (to which we refer for further details), we subdivide $\Delta^n\times [0,1]$ into $n+1$-simplexes as follows. Write $\Delta^n\times \{0\}= [v_0, \ldots, v_n]$ and $\Delta^n \times \{1\}= [w_0, \ldots, w_n]$ (subsets of $\mathbb{R}^{n+2}$) and define the Prism operator
$$ \widetilde{H}_n: \mathrm{C}_n(\widehat{\Sigma}; \mathbb{Z}) \to \mathrm{C}_{n+1}(\widehat{\Sigma}; \mathbb{Z}), \quad \widetilde{H}_n: \sigma \mapsto \sum_i (-1)^i \restriction{h \circ (\sigma \times \mathds{1})}{[v_0, \ldots, v_i, w_i, \ldots, w_n]}.$$
By tensoring with $\mathfrak{g}$, it induces a map $H_n : \mathrm{C}_n(\Sigma; Ad_{\rho_v}) \to \mathrm{C}_{n+1}(\Sigma; Ad_{\rho_v})$. By a computation similar to the one done in \cite[Theorem $2.10$]{Hatcher_Book}, we see that 
$$ \Psi_n \circ \Theta_n - \id = \partial \circ H_n - H_{n-1} \circ \partial.$$
Thus $\Psi_{\bullet}$ is a homotopy equivalence. We obtain the similar result for cohomology by duality.

\end{proof}

\begin{lemma}\label{lemma_H2=0} Let $\mathbf{\Sigma}=(\Sigma, \mathcal{A})$ be a connected marked surface such that $\mathcal{A}\neq \emptyset$. Then for all $\rho \in \mathcal{R}_G(\mathbf{\Sigma})$, one has $\mathrm{H}^2(\Sigma, \mathcal{A}; \rho)=0$.
\end{lemma}

\begin{proof}
Using the pairing between homology and cohomology, it suffices to prove that $\mathrm{H}_2(\Sigma, \mathcal{A}; \rho)=0$.
Let us first suppose that $|\mathcal{A}|=1$ and let $v\in \mathcal{A}$. The homotopy equivalence $\Psi_{\bullet}: \mathrm{C}_{\bullet}(\Sigma; \rho) \to \mathrm{C}_{\bullet}(\Sigma; Ad_{\rho_v})$, sends the subcomplex $\mathrm{C}_{\bullet}(\mathcal{A}; \rho)$ to $\mathrm{C}_{\bullet}(\mathcal{A}; Ad_{\rho_v})$ so induces an isomorphism $\mathrm{H}^2(\Sigma, \mathcal{A}; \rho) \cong \mathrm{H}^2(\Sigma, \mathcal{A}; Ad_{\rho_v})=0$; the latter vanishes since $\Sigma$ retracts to a subgraph $\Gamma$ whose only vertex is $v$, so the pair $(\Sigma, \mathcal{A})$ retracts to the pair $(\Gamma, \{v\})$. When $|\mathcal{A}|\geq 2$, for each $a\in \mathcal{A}$ fix $v_a \in a$ and let $\underline{\Sigma}:= \quotient{\Sigma}{(v_a \sim v_b, a,b \in \mathcal{A})}$ be the surface obtained from $\Sigma$ by identifying all pairs $(v_a,v_b)$ to a single point $v$ and smoothing the corners. The marked surface $\underline{\mathbf{\Sigma}}$ is $\underline{\Sigma}$ with a single boundary arc $a$ containing $v$. The projection $\pi: \Sigma \to \underline{\Sigma}$ induces a functor $\pi_*: \Pi_1(\Sigma) \to \Pi_1(\underline{\Sigma})$ which is an equivalence. Let $\underline{\rho}: \Pi_1(\underline{\Sigma})\to G$ such that $\underline{\rho} \circ \pi_*= \rho$. The functor $\pi_*$ induces an isomorphism of chain complexes $\pi_{\bullet}: \mathrm{C}_{\bullet}(\Sigma, \mathcal{A}; \rho) \to \mathrm{C}_{\bullet}(\underline{\Sigma}, \{a\}; \underline{\rho})$ so $\mathrm{H}_2(\Sigma, \mathcal{A}; \rho) \cong \mathrm{H}_2(\underline{\Sigma}, \{a\}; \underline{\rho}) =0$ by the preceding case.

\end{proof}

\subsection{Relation with the (co)tangent spaces of the \stated character varieties}

\par We first recall some basic facts about tangent spaces (see \cite{Hart} for more details). Let $X$ be a complex affine variety and $x\in X$ be a point represented by a character $\chi_x : \mathbb{C}[X] \rightarrow \mathbb{C}$. We endow the space $\mathbb{C}$ with a structure of $\mathbb{C}[X]$-bimodule, denoted $\mathbb{C}_{\chi_x}$, defined for any $f\in \mathbb{C}[X]$ and $z\in \mathbb{C}$ by the formula $f\cdot z = z \cdot f := \chi_x(f)z$. The Zariski tangent space $T_xX$ is defined as the set $\mathrm{Der}(\mathbb{C}[X], \mathbb{C}_{\chi_x})$ of derivations, that is of linear maps $\varphi : \mathbb{C}[X] \rightarrow \mathbb{C}$ satisfying $\varphi(fh)=\chi_x(f)\varphi(h) +\varphi(f)\chi_x(h)$. If $f:X\rightarrow Y$ is a regular map defined by an algebra morphism $f^* : \mathbb{C}[Y] \rightarrow \mathbb{C}[X]$, define $D_x f : T_xX \rightarrow T_y Y$ as the map sending a derivation $\varphi$ to $\varphi \circ f^*$.

\vspace{2mm}
\par Let $\Pi_1(\Sigma, \mathcal{A})$ be the category whose objects are the points of $\Sigma$ and such that $\Pi_1(\Sigma, \mathcal{A})(x,y):= \quotient{ \Pi_1(\Sigma)(x,y)}{\Pi_1(\mathcal{A})(x,y)}$.
Consider the two affine schemes:
$$ X_1:= \Map(\Pi_1(\Sigma, \mathcal{A}), G), \quad X_2:= \Map(\Pi_1(\Sigma, \mathcal{A})\times_{\Sigma}\Pi_1(\Sigma, \mathcal{A}), G). $$
Here $\Pi_1(\Sigma, \mathcal{A})\times_{\Sigma}\Pi_1(\Sigma, \mathcal{A})$ is the set of pairs $(\alpha_2,\alpha_1)$ such that $s(\alpha_2)=t(\alpha_1)$. Set 
$$ \mathcal{R}: X_1 \to X_2, \quad \mathcal{R}(\rho): (\alpha_2, \alpha_1) \mapsto \rho(\alpha_2\alpha_1)^{-1}\rho(\alpha_2)\rho(\alpha_1).$$
Clearly, $X_1$ and $X_2$ are smooth schemes and $\mathcal{R}_G(\mathbf{\Sigma})=\mathcal{R}^{-1}(e)$ where $e$ is the constant map with value the neutral element of $G$. Fix $\rho \in \mathcal{R}_G(\mathbf{\Sigma})$ and define two isomorphisms
$$ \Lambda_1: T_{\rho} X_1 \xrightarrow{\cong} \mathrm{C}^1(\Sigma, \mathcal{A}; \rho) \quad \mbox{ and }\quad \Lambda_2: T_{e} X_2 \xrightarrow{\cong} \mathrm{C}^2(\Sigma, \mathcal{A}; \rho)$$
as follows. A derivation $\varphi \in T_{\rho}X_1= \mathrm{Der} \left( \mathbb{C}[\Map(\Pi_1(\Sigma, \mathcal{A}), G)], \mathbb{C}_{\rho} \right)$ induces on each factor $\mathbb{C}[G]_{\alpha} =\iota_{\alpha}(\mathbb{C}[G])\subset  \mathbb{C}[\Map(\Pi_1(\Sigma), G)]$ an element $X_{\alpha} \in \mathrm{Der}(\mathbb{C}[G], \mathbb{C}_{\chi_{\rho(\alpha)}})= T_{\rho(\alpha)}G$. We define $\Lambda_1(\varphi)$ as the map sending $(\alpha_1,\alpha_0)$ to $X_{\alpha_1}\rho(\alpha_0)^{-1}$. 
 Conversely, if $c^1\in  \mathrm{C}^1(\Sigma, \mathcal{A}; \rho)$ and $\alpha \in \Pi_1(\Sigma, \mathcal{A})$, consider the derivation $X_{\alpha}:= c^1(\alpha, 1_{t(\alpha)}) \in T_{\rho(\alpha)}G$. Define $\Lambda_1^{-1}(c^1)\in \mathrm{Der}\left( \mathbb{C}[\Map(\Pi_1(\Sigma, \mathcal{A}), G)], \mathbb{C}_{\chi_{\rho}} \right)$ by the formula $\Lambda_1^{-1}(c^1) (x_{\alpha}):= X_{\alpha} (x)$. The maps $\Lambda_1$ and $\Lambda_1^{-1}$ are inverse each other, hence are isomorphisms.
\par Similarly, a derivation $\psi \in T_{e}X_2= \mathrm{Der} \left( \mathbb{C}[\Map(\Pi_1(\Sigma, \mathcal{A})\times_{\Sigma}\Pi_1(\Sigma, \mathcal{A}), G)], \mathbb{C}_{e} \right)$ induces on each factor $\mathbb{C}[G]_{(\alpha_2,\alpha_1)} \subset  \otimes_{(\beta_2,\beta_1) \in\Pi_1(\Sigma, \mathcal{A})\times_{\Sigma}\Pi_1(\Sigma, \mathcal{A})}^{\vee}\mathbb{C}[G]_{(\beta_2,\beta_1)}$ an element $Y_{(\alpha_2,\alpha_1)} \in T_eG=\mathfrak{g}$. We define $\Lambda_2(\psi)$ as the map sending $(\alpha_2, \alpha_1,\alpha_0)$ to $\rho(\alpha_2\alpha_1)Y_{(\alpha_2, \alpha_1)}\rho(\alpha_0)$. We prove that $\Lambda_2$ is an isomorphism similarly than for $\Lambda_1$. 

\begin{lemma}\label{lemma_cohomology1}
In the following diagram
$$ \begin{tikzcd}
0 \ar[r] & T_{\rho}\mathcal{R}_G(\mathbf{\Sigma}) \ar[r] \ar[d, dotted, "\exists! \cong", "\Lambda'"'] &
T_{\rho}X_1 \ar[r, "D_{\rho}\mathcal{R}"] \ar[d, "\Lambda_1", "\cong"'] &
T_eX_2 \ar[d, "\Lambda_2", "\cong"'] \\
0 \ar[r] & \mathrm{Z}^1(\Sigma, \mathcal{A}; \rho) \ar[r] & 
 \mathrm{C}^1(\Sigma, \mathcal{A}; \rho) \ar[r, "d^1"]] & 
\mathrm{C}^2(\Sigma, \mathcal{A}; \rho)
\end{tikzcd}$$
one has $d^1\circ \Lambda_1 = \Lambda_2 \circ D_{\rho}\mathcal{R}$. Therefore, $\Lambda_1$ induces an isomorphism $\Lambda': T_{\rho}\mathcal{R}_G(\mathbf{\Sigma})  \xrightarrow{\cong} \mathrm{Z}^1(\Sigma, \mathcal{A}; \rho)$. Moreover $\rho$ is a smooth point of $\mathcal{R}_G(\mathbf{\Sigma})$ if and only if $\mathrm{H}^2(\Sigma, \mathcal{A}; \rho)=0$.

\end{lemma}

\begin{proof}
This is proved by a simple computation as follows: 
\begin{multline*} \mathcal{R}(\rho + \varepsilon X) (\alpha_2, \alpha_1)= (\rho(\alpha_2\alpha_1) + \varepsilon X(\alpha_2\alpha_1))^{-1}(\rho(\alpha_2) + \varepsilon X(\alpha_2))(\rho(\alpha_1)+ \varepsilon X(\alpha_1)) \\ \equiv e + \varepsilon 
\left( \rho(\alpha_2\alpha_1)^{-1}X(\alpha_2)\rho(\alpha_1) + \rho(\alpha_1)^{-1}X(\alpha_1) - X(\alpha_2\alpha_1) \rho(\alpha_2\alpha_1) \right) \pmod{\varepsilon^2}.\end{multline*}
So if $\Lambda_1(\varphi)=: c^1$ and $X_{\alpha}=c^1(\alpha, 1)$, then 
$$
 \Lambda_2\circ D_{\rho}\mathcal{R} (\varphi): (\alpha_2, \alpha_1, \alpha_0)   = 
 X_{\alpha_2} \rho(\alpha_1\alpha_0) =\rho(\alpha_2)X_{\alpha_1}\rho(\alpha_0) - X_{\alpha_2\alpha_1}\rho(\alpha_0)
 =
 d^1c^1(\alpha_2, \alpha_1, \alpha_0).
 $$

\end{proof}

\par Recall that we defined the gauge group via $\mathbb{C}[\mathcal{G}]= \mathbb{C}[\Map(\Sigma \setminus \mathcal{A}, G)]$. Its neutral element is represented by the character $\chi_0 : \mathbb{C}[\mathcal{G}] \rightarrow \mathbb{C}$ defined by $\chi_0(x_v)=\epsilon(x)$. Define an isomorphism $\Lambda^0 : \mathrm{Der}\left(\mathbb{C}[\mathcal{G}], \mathbb{C}_{\chi_0} \right)\xrightarrow{\cong} \mathrm{C}^0\left(\Sigma, \mathcal{A} ; \rho \right)$ as follows. A derivation $\varphi_0 \in \mathrm{Der}\left(\mathbb{C}[\mathcal{G}], \mathbb{C}_{\chi_0} \right)$ induces on each factor $\mathbb{C}[G]_v=\iota_v(\mathbb{C}[G]) \subset \mathbb{C}[\mathcal{G}]$ a derivation $X_v \in \mathrm{Der}(\mathbb{C}[G]_v, \mathbb{C}_{\epsilon})=\mathfrak{g}$. Define $\Lambda^0 (\varphi_0)=c^0$ by the formula $c^0(\alpha):= X_{s(\alpha)}\rho(\alpha)^{-1}$. The inverse map of $\Lambda^0$ sends a map $c^0\in  \mathrm{C}^0\left( \Sigma, \mathcal{A} ; \rho \right)$ for which $X_v:= c^0(1_v) \in \mathfrak{g}$, to the derivation $\varphi_0$ defined by $\varphi_0(x_v):= X_v(x)$. 
\vspace{2mm}
\par Consider the map $c_{\rho} : \mathcal{G} \rightarrow \mathcal{R}_G(\mathbf{\Sigma})$ defined via the gauge group action $c_{\rho}(g):= g\cdot \rho$. This map is better described as the algebra morphism $c_{\rho}^* : \mathbb{C}[\mathcal{R}_G(\mathbf{\Sigma})] \rightarrow \mathbb{C}[\mathcal{X}_G(\mathbf{\Sigma})]$ defined as the composition 
$$c_{\rho}^* : \mathbb{C}[\mathcal{R}_G(\mathbf{\Sigma})] \xrightarrow{\Delta_{\mathcal{G}}^L} \mathbb{C}[\mathcal{G}] \otimes \mathbb{C}[\mathcal{R}_G(\mathbf{\Sigma})] \xrightarrow{\id \otimes \chi_{\rho}} \mathbb{C}[\mathcal{G}].$$
\par The morphism $D_e c_{\rho}: \mathrm{Der}\left(\mathbb{C}[\mathcal{G}], \mathbb{C}_{\chi_0}\right) \rightarrow \mathrm{Der} \left(\mathbb{C}[\mathcal{R}_G(\mathbf{\Sigma})], \mathbb{C}_{\chi_{\rho}} \right)$ is the map sending a derivation $\varphi_0$ to the derivation $\varphi_0 \circ c_{\rho}^*$. 

\begin{lemma}\label{lemma_cohomology2}
The following diagram is commutative: 
$$\begin{tikzcd}
T_e\mathcal{G}
 \arrow[r, "D_e c_{\rho}"] \arrow[d, "\cong"', "\Lambda^0"] &
T_{\rho}\mathcal{R}_G(\mathbf{\Sigma})
 \arrow[d, "\cong"', "\Lambda'"]  \\
\mathrm{C}^0\left(\Sigma, \mathcal{A} ; \rho \right)
\arrow[r, "-d^0"] &
\mathrm{Z}^1\left(\Sigma, \mathcal{A} ; \rho \right)
\end{tikzcd}$$

\end{lemma}

\begin{proof}

Let $\varphi_0 \in  \mathrm{Der}\left( \mathbb{C}[\mathcal{G}], \mathbb{C}_{\chi_0} \right)=T_e\mathcal{G}$. For each $v\in \Sigma\setminus \mathcal{A}$ denote by $X_v \in \mathfrak{g}$ the vector induced by $\varphi_0$ and set $c^0 := \Lambda^0 (\varphi_0)$ the map defined by $c^0(1_v) = X_v$.  For each $\alpha \in \Pi_1(\Sigma)$ denote by $X_{\alpha} \in T_{\rho(\alpha)}G$ the derivation induced by $D_e c_{\rho} (\varphi_0)$ and set $c^1 := \Lambda'\circ D_e c_{\rho} (\varphi_0)$ the map defined by $c^1(\alpha, 1_{t(\alpha)})=X_{\alpha}$. Choose a path $\alpha \in \Pi_1(\Sigma)$ such that $s(\alpha), t(\alpha) \in \Sigma\setminus \mathcal{A}$ and an element $x\in \mathbb{C}[G]$. One has: 
\begin{eqnarray*}
X_{\alpha}(x) &=& D_e c_{\rho} (\varphi_0) (x_{\alpha}) = \varphi_0 \circ (\id \otimes \chi_{\rho}) \circ (\Delta_{\mathcal{G}}^L) (x_{\alpha}) \\
 &=&  \varphi_0 \circ (\id \otimes \chi_{\rho}) \left( \sum x'_{s(\alpha)} S(x''')_{t(\alpha)} \otimes x''_{\alpha} \right) \\
 &=& \sum \varphi_0 \left( x'_{s(\alpha)} S(x''')_{t(\alpha)} \right) \chi_{\rho}(x''_{\alpha}) \\
  &=& \sum \epsilon(x') \varphi_0 (S(x''')_{t(\alpha)}) \chi_{\rho} (x''_{\alpha}) + \sum \epsilon \circ S(x''') \varphi_0(x'_{s(\alpha)}) \chi_{\rho}(x''_{\alpha}) \\
   &=& (\rho(\alpha)X_{t(\alpha)}- X_{s(\alpha)}\rho(\alpha))(x)
   \end{eqnarray*}
   The above equalities imply that $c^1(\alpha, 1_{t(\alpha)}) = -d^0 c^0 (\alpha, 1_{t(\alpha)})$. The cases where $(s(\alpha), t(\alpha))\cap \mathcal{A} \neq \emptyset$ are proved similarly.  Hence we have the equality $c^1 = -d^0 (c^0)$ which proves that $\Lambda'\circ D_e c_{\rho} =- d^0 \circ \Lambda^0$ and concludes the proof.

\end{proof}

Putting Lemmas \ref{lemma_cohomology1} and \ref{lemma_cohomology2} together, we see that $\Lambda'$ induces an isomorphism 
$$\Lambda'': \quotient{ T_{\rho}\mathcal{R}_G(\mathbf{\Sigma})}{\Image(D_ec_{\rho})} \xrightarrow{\cong} \mathrm{H}^1(\Sigma, \mathcal{A}; \rho).$$

Let $\iota : \mathbb{C}[\mathcal{X}_G(\mathbf{\Sigma})] \hookrightarrow \mathbb{C}[\mathcal{R}_G(\mathbf{\Sigma})]$ be the inclusion map and $p: \mathcal{R}_G(\mathbf{\Sigma}) \to \mathcal{X}_G(\mathbf{\Sigma})$ be the projection map defined by $\iota$. 
Since $p\circ c_{\rho}$ is the constant map with value $[\rho]$, its derivative at $e$ vanishes and the chain rules implies that $T_ep$ passes to the quotient to a map (denoted by the same letter) $T_ep :  \quotient{ T_{\rho}\mathcal{R}_G(\mathbf{\Sigma})}{\Image(D_ec_{\rho})} \to T_{[\rho]}\mathcal{X}_{G}(\mathbf{\Sigma})$. Consider the composition 
$$ \Upsilon: \mathrm{H}^1(\Sigma, \mathcal{A}; \rho) \xrightarrow{ (\Lambda'')^{-1}} \quotient{ T_{\rho}\mathcal{R}_G(\mathbf{\Sigma})}{\Image(D_ec_{\rho})} \xrightarrow{ T_ep} T_{[\rho]}\mathcal{X}_{G}(\mathbf{\Sigma}).$$

Recall from Section \ref{sec_stabilizer} that $\rho$ is a good representation if for each connected component $\mathbf{\Sigma}^0=(\Sigma^0, \mathcal{A}^0)$ of $\mathbf{\Sigma}$ then either $\mathcal{A}^0\neq \emptyset$ or for $v\in \Sigma^0$ the induced representation $\rho_v: \pi_1(\Sigma, v) \to G$ is irreducible and has stabilizer equal to the center $C(G)$ of $G$. 

\begin{theorem}\label{theorem_cohomology}
If $\rho$ is a good representation then $\Upsilon: \mathrm{H}^1(\Sigma, \mathcal{A}; \rho)  \to T_{[\rho]}\mathcal{X}_{G}(\mathbf{\Sigma})$ is an isomorphism.
\end{theorem}

The proof is a direct adaptation of the arguments in \cite{Sikora} (made for classical character varieties) based on the following consequence of Luna's slice \'etale theorem \cite{Luna_SliceTheorem}.  

\begin{lemma}\label{lemma_LunaSlice} Let ${G}$ be a  reductive group which acts on an affine scheme ${R}$ and write $X:= R\sslash G$ and $\pi: R\to X$ the quotient map. Let $\rho\in R$ such that $(i)$ $\rho$ is a smooth point of $R$ and $(ii)$ the orbit $\mathcal{O}_{\rho}=G\cdot \rho \subset R$ is closed and $(iii)$ the stabilizer $S_{\rho}$ of $\rho$ acts trivially on $R$. Then $D_{\rho}\pi$ induces an isomorphism $\quotient{T_{\rho}R}{T_{\rho}\mathcal{O}} \xrightarrow{\cong} T_{[\rho]}X$. 
\end{lemma}

\begin{proof} By Luna's slice \'etale theorem \cite{Luna_SliceTheorem} (see also \cite{Drezet_LunaSlice}), conditions $(i)$ and $(ii)$ imply that  there exists a $\mathcal{G}$ invariant subscheme $S\subset R$ containing $\rho$ such that $(1)$ the map $p: G\times_{S_{\rho}}S \to R$, $p(g,s)=g\cdot s$ is \'tale and $(2)$ the projection map $q: S\sslash S_{\rho} \to X$ is \'etale. Therefore the following differentials of $p$ and $\pi$ are isomorphisms: 
$$ D_{(e,\rho)} p : T_{(e,\rho)} G\times_{S_{\rho}}S \xrightarrow{\cong} T_{\rho}R \quad \mbox{ and } \quad D_{[\rho]}q: T_{[\rho]} S\sslash S_{\rho} \xrightarrow{\cong} T_{[\rho]}X.$$
 By hypothesis $(iii)$, $G\times_{S_{\rho}}S= (\quotient{G}{S_{\rho}})\times S$ and $S\sslash S_{\rho}=S$ so we have a commutative diagram
 $$\begin{tikzcd}
T_e(\quotient{G}{S_{\rho}})\oplus T_{\rho}S   \ar[d, "D_{(e,\rho)}p", "\cong"'] \ar[r, "0 \oplus \id"] & T_{\rho}S \ar[d, "D_{[\rho]}q", "\cong"'] \\
T_{\rho}R \ar[r, "D_{\rho}\pi"] & T_{[\rho]}X
\end{tikzcd}$$
We conclude using the fact that $T_{\rho}\mathcal{O}_{\rho}$ is the image of $T_e(\quotient{G}{S_{\rho}})$ by $D_{(e,\rho)}p$.

\end{proof}

\begin{proof}[Proof of Theorem \ref{theorem_cohomology}] Without loss of generality, we suppose that $\mathbf{\Sigma}$ is connected.  Let $\rho$ be a good representation, $\mathcal{O}_{\rho}:=\mathcal{G}\cdot \rho \subset \mathcal{R}_G(\mathbf{\Sigma})$ its orbit and $S_{\rho}$ its stabilizer. Recall from Lemma \ref{lemma_stabilizer} that $S_{\rho}=\{e\}$ if $\mathcal{A}\neq \emptyset$ and that $S_{\rho}\cong C(G)$ else. In particular $S_{\rho}$ acts trivially on $\mathcal{R}_G(\mathbf{\Sigma})$. 
\par \textbf{Step 1:} Let us prove that $\Image(D_{\rho}c_{\rho})=T_{\rho}\mathcal{O}_{\rho}$. Clearly, the map $c_{\rho}: \mathcal{G}\to \mathcal{R}_G(\mathbf{\Sigma})$ induces an isomorphism $\mathcal{G}\sslash S_{\rho} \cong \mathcal{O}_{\rho}$. Decomposing $c_{\rho}$ as $c_{\rho}: \mathcal{G} \to \mathcal{G} \sslash S_{\rho} \cong \mathcal{O}_{\rho}$, we need to prove that the map $D_ec_{\rho} : T_e\mathcal{G} \to T_e \mathcal{G}\sslash S_{\rho} \cong T_e \mathcal{O}_{\rho}$ is an epimorphism. Let us apply Lemma \ref{lemma_LunaSlice} to $R= \mathcal{G}$ and $G= S_{\rho}$. Clearly $e\in \mathcal{G}$ is a smooth point (since $\mathcal{G}$ is smooth) and its orbit $S_{\rho}\subset \mathcal{G}$ is closed; so Lemma \ref{lemma_LunaSlice} implies that $T_e\mathcal{G} \to T_e \mathcal{G}\sslash S_{\rho}$ is an epimorphism. 
\par \textbf{Step 2:} Let us prove that $\mathcal{O}_{\rho} \subset \mathcal{R}_G(\mathbf{\Sigma})$ is closed. When $\mathcal{A}\neq \emptyset$, this follows from the fact that $\mathcal{G}\cong \mathcal{O}_{\rho}$. When $\mathcal{A}=\emptyset$, fix $v\in \Sigma$ and consider the map $p_v: \mathcal{R}_G(\mathbf{\Sigma}) \to \Hom(\pi_1(\Sigma, v), G)$ sending $\rho$ to its restriction $\rho_v$ at $\pi_1(\Sigma, v)$. Let $\mathcal{O}_{\rho, v}\subset \Hom(\pi_1(\Sigma, v), G)$ be the conjugacy $G$ orbit of $\rho_v$. By Lemma \ref{lemma_stabilizer}, $p_v^{-1}(\mathcal{O}_{\rho, v})=\mathcal{O}_{\rho}$ so $\mathcal{O}\subset \mathcal{R}_G(\mathbf{\Sigma})$ is closed whenever $\mathcal{O}_{\rho,v} \subset \Hom(\pi_1(\Sigma, v), G)$ is closed. We conclude using \cite[Theorem $30$]{Sikora}. 
\par \textbf{Step 3:} We show that $\mathrm{H}^2(\Sigma, \mathcal{A}; \rho)=0$. Indeed, if $\mathcal{A}\neq \emptyset$, this is proved in Lemma \ref{lemma_H2=0}. If $\mathcal{A}=\emptyset$, by Proposition \ref{prop_group_homology}, one has $\mathrm{H}^2(\Sigma; \rho)\cong \mathrm{H}^2( \Sigma; Ad_{\rho_v})$ and the latter vanishes when $\rho_v$ is a good representation by \cite{Sikora}.
\par 
 We can now conclude. Since $\mathrm{H}^2(\Sigma, \mathcal{A}; \rho)=0$, Lemma \ref{lemma_cohomology1} implies that $\rho$ is a smooth point. In Step $2$ we proved that its orbit is closed so Lemma \ref{lemma_LunaSlice} and step $1$ imply that $\Upsilon: \quotient{T_{\rho}\mathcal{R}_G(\mathbf{\Sigma})}{\Image(D_ec_{\rho})} = \quotient{T_{\rho}\mathcal{R}_G(\mathbf{\Sigma})}{T_e\mathcal{O}_{\rho}} \to T_{[\rho]}\mathcal{X}_G(\mathbf{\Sigma})$ is an isomorphism.  

\end{proof}

\begin{definition}
For $\rho\in \mathcal{R}_G(\mathbf{\Sigma})$ a good representation, we denote by $\Lambda : T_{[\rho]} \mathcal{X}_G(\mathbf{\Sigma}) \xrightarrow{\cong} \mathrm{H}^1\left(\Sigma, \mathcal{A} ; \rho \right)$ the inverse of $\Upsilon$.  The non-degenerate pairing $\left< \cdot, \cdot \right> : \mathrm{H}_{1}(\Sigma, \mathcal{A}; \rho) \otimes \mathrm{H}^{1}(\Sigma, \mathcal{A} ; \rho) \rightarrow \mathbb{C}$ induces an isomorphism between the space $ \mathrm{H}_{1}(\Sigma, \mathcal{A}; \rho) $ and the dual of $ \mathrm{H}^{1}(\Sigma, \mathcal{A} ; \rho)$. Identifying the cotangent space $\Omega^1_{[\rho]} \mathcal{X}_G(\mathbf{\Sigma})$ with the dual of the Zariski tangent space $T_{[\rho]}\mathcal{X}_G(\mathbf{\Sigma})$, the isomorphism  $\Lambda$ induces an isomorphism $\Lambda^* : \Omega^1_{[\rho]} \mathcal{X}_G(\mathbf{\Sigma}) \xrightarrow{\cong}  \mathrm{H}_{1}(\Sigma, \mathcal{A}; \rho) $.
 \end{definition}

  Note that we have proved Theorem \ref{theorem2}, i.e. if  $\rho \in \mathcal{R}_G(\mathbf{\Sigma})$ is a good representation with class $[\rho]\in \mathcal{X}_G(\mathbf{\Sigma})$, there exists canonical isomorphisms $\Lambda: T_{[\rho]} \mathcal{X}_{G}(\mathbf{\Sigma}) \xrightarrow{\cong} \mathrm{H}^1(\Sigma, \mathcal{A}; \rho)$ between the Zariski tangent space and the first twisted cohomological group, and $\Lambda^*: \Omega^1_{[\rho]} \mathcal{X}_G(\mathbf{\Sigma}) \xrightarrow{\cong} \mathrm{H}_1(\Sigma, \mathcal{A}; \rho)$ between the cotangent space and the first twisted homological group respectively.

\begin{notations} Let $f_{\mathcal{C}}$ be a curve function and $\alpha$ a path representative of $\mathcal{C}$. Define the vector $X_{f, \alpha} \in T_{\rho(\alpha)} G$ as the vector such that for all $Y \in T_{\rho(\alpha)}G$ one has $D_{\rho(\alpha)} f (Y) = \left( X_{f, \alpha} , Y\right)_{\rho(\alpha)}$. Recall that we denote by $[\alpha, X]\in \mathrm{H}_{1}(\Sigma, \mathcal{A}; \rho) $ the class of the element $\left< (\alpha, 1_{t(\alpha)}), X \right>$. Both the vector $X_{f, \alpha}$ and the class $[\alpha, X_{f, \alpha}]$ are independent of the choice of the path representative $\alpha$. We will denote them by $X_{f, \mathcal{C}}$ and $[\mathcal{C}, X_{f, \mathcal{C}}]$ respectively.
\end{notations}

\begin{lemma}\label{lemma_derivative}
The isomorphism  $\Lambda^* : \Omega^1_{[\rho]} \mathcal{X}_G(\mathbf{\Sigma}) \xrightarrow{\cong}  \mathrm{H}_{1}(\Sigma, \mathcal{A}; \rho) $ sends  the derivative $D_{[\rho]} f_{\mathcal{C}}$ to the class $[\mathcal{C}, X_{f, \mathcal{C}}]$. 
\end{lemma}

\begin{proof} Let $\varphi \in \mathrm{Der}\left( \mathbb{C}[\mathcal{R}_G(\mathbf{\Sigma})], \chi_{\rho} \right)$ and define $c^1:= \Lambda''' (\varphi)$ and $[\varphi]:= \varphi \circ \iota \in T_{[\rho]} \mathcal{X}_G(\mathbf{\Sigma})$ such that $\Lambda([\varphi])=[c^1]$. We must show that $D_{[\rho]}f_{\mathcal{C}} \cdot [\varphi] = \left< [\mathcal{C}, X_{f, \mathcal{C}}], [c^1] \right>$ to conclude. Unravelling the definitions, one has the following identities
\begin{eqnarray*}
\left< [\mathcal{C}, X_{f,\mathcal{C},}], [c^1] \right> &=& \left( X_{\alpha, f}, c^1(\alpha, 1_{t(\alpha)}) \right) = D_{\rho(\alpha)}f \cdot c^1(\alpha, 1_{t(\alpha)}) \\
 &=& D_{\rho} f_{\alpha} \cdot \varphi = D_{[\rho]}f_{\mathcal{C}} \cdot [\varphi] 
\end{eqnarray*}
This concludes the proof.

\end{proof}

\par Let $a$ and $b$ be two boundary arcs of $\mathbf{\Sigma}$ and consider the injective morphism $i_{a\# b} : \mathbb{C}[\mathcal{X}_G(\mathbf{\Sigma}_{ a\#b})] \hookrightarrow \mathbb{C}[\mathcal{X}_G(\mathbf{\Sigma})]$ of Proposition \ref{gluing_formula}, corresponding to a surjective regular map $\Psi_{a\#b} : \mathcal{X}_G(\mathbf{\Sigma}) \rightarrow \mathcal{X}_G(\mathbf{\Sigma}_{ a\#b})$. Write $\rho_{a\#b}= \Psi_{a\#b}(\rho)$. The derivative $D_{[\rho]} \Psi_{a\#b} : T_{[\rho]} \mathcal{X}_G(\mathbf{\Sigma}) \rightarrow T_{[\rho_{a\#b}]} \mathcal{X}_G(\mathbf{\Sigma}_{ a\#b})$, together with the isomorphism $\Lambda$, induce a surjective linear map $\eta_{a\#b} : \mathrm{H}^1 \left(\Sigma, \mathcal{A} ; \rho \right) \rightarrow \mathrm{H}^1 \left( \Sigma_{a\#b}, \mathcal{A}_{a\#b} ; \rho_{a\#b} \right)$. By duality, one obtains also an injective map $j_{a\#b} : \mathrm{H}_1 \left( \Sigma_{a\#b}, \mathcal{A}_{a\#b} ; \rho_{a\#b} \right) \rightarrow \mathrm{H}_1 \left( \Sigma, \mathcal{A} ; \rho \right)$ described as follows.
\vspace{2mm}
\par Note that it follows from Proposition \ref{prop_holonomy_functions} and Lemma \ref{lemma_derivative} that the space $\mathrm{H}_1 \left( \Sigma, \mathcal{A} ; \rho \right) $ is spanned by the classes of the form $[\mathcal{C}, X]$. Denote by $\pi : \Sigma \rightarrow \Sigma_{a\#b}$ the natural projection and by $c$ the image of $a$ and $b$ by $\pi$. Let $\mathcal{C}$ be a curve in $\mathbf{\Sigma}_{a\#b}$. Let $\alpha_{\mathcal{C}}$ be a path representative of $\mathcal{C}$ and choose a decomposition $\alpha_{\mathcal{C}} = \beta_1 \ldots \beta_n$ such that each path $\beta_i$ admits a geometric representative whose interior does not intersect $c$. Lifting each geometric representative through $\pi$, we obtain a collection $\alpha_1, \ldots, \alpha_n$ of paths in $\Pi_1(\Sigma)$ such that the path $\alpha_1\ldots \alpha_n$ is sent to $\alpha_{\mathcal{C}}$ through $\pi$.

\begin{lemma}\label{lemma_decomp_paths}
The linear map  $j_{a\#b} : \mathrm{H}_1 \left(\Sigma_{a\#b}, \mathcal{A}_{a\#b} ; \rho_{a\#b} \right) \rightarrow \mathrm{H}_1 \left( \Sigma, \mathcal{A} ; \rho \right)$ is characterized by the formula
$$ j_{a\#b} \left( [\mathcal{C}, X] \right) = \sum_i [\alpha_i, \rho\left( \alpha_1 \ldots \alpha_{i-1}\right)^{-1} X \rho\left(\alpha_{i+1}\ldots \alpha_{n}\right)^{-1} ].$$
\end{lemma} 

\begin{proof}
The proof is a straightforward consequence of the definition and of the following equality in homology
$$[\beta_1\ldots \beta_n, X] = \sum_i [\beta_i, \rho\left( \beta_1 \ldots \beta_{i-1}\right)^{-1} X \rho\left(\beta_{i+1}\ldots \beta_{n}\right)^{-1} ]. $$

\end{proof}

\subsection{Discrete versions of twisted cohomology}

\par Let $\mathbf{\Sigma}$ be a punctured surface, $\mathbb{P}$ a finite presentation of $\Pi_1(\Sigma)$ and $\rho \in \mathcal{R}_G(\mathbf{\Sigma}, \mathbb{P})$.  Recall that we defined a regular map $\mathcal{R} : G^{\mathbb{G}} \rightarrow G^{\mathbb{RL}}$ such that $\mathcal{R}_G(\mathbf{\Sigma}, \mathbb{P}):= \mathcal{R}^{-1}(e, \ldots, e)$ and that we defined the discrete gauge group as $\mathcal{G}_{\mathbb{P}}:= G^{\mathring{\mathbb{V}}}$. Define a regular map $c_{\rho} : G^{\mathring{\mathbb{V}}} \rightarrow G^{\mathbb{G}}$ by the gauge group action $c_{\rho}(g) = g\cdot \rho$.

\begin{definition} The cochain complex $\left(\mathrm{C}^{\bullet} (\mathbf{\Sigma}, \mathbb{P}; \rho), d^{\bullet} \right)$ is defined as follows.
 The graded space $\mathrm{C}^{\bullet} (\mathbf{\Sigma}, \mathbb{P}; \rho)$ has only non trivial graded parts in degree $0,1$ and $2$ which are defined by: 
\begin{equation*}
\mathrm{C}^{0} (\mathbf{\Sigma}, \mathbb{P}; \rho):= T_{(e, \ldots, e )} G^{\mathring{\mathbb{V}}}\cong \mathfrak{g}^{\oplus \mathring{\mathbb{V}}}\quad 
\mathrm{C}^{1} (\mathbf{\Sigma}, \mathbb{P}; \rho):= T_{\rho} G^{\mathbb{G}} \cong \oplus_{\beta \in \mathbb{G}} T_{\rho(\beta)}G \quad
\mathrm{C}^{2} (\mathbf{\Sigma}, \mathbb{P}; \rho):= T_{(e, \ldots, e)} G^{\mathbb{RL}} \cong \mathfrak{g}^{\oplus \mathbb{RL}} 
\end{equation*}
\par The co-boundary maps $d^0 : \mathrm{C}^{0} (\mathbf{\Sigma}, \mathbb{P}; \rho)\rightarrow \mathrm{C}^{1} (\mathbf{\Sigma}, \mathbb{P}; \rho)$ and $d^1 : \mathrm{C}^{1} (\mathbf{\Sigma}, \mathbb{P}; \rho)\rightarrow \mathrm{C}^{2} (\mathbf{\Sigma}, \mathbb{P}; \rho)$ are defined as the derivatives $d^0 := D_{(e, \ldots, e)} c_{\rho}$ and $d^1:= D_{\rho} \mathcal{R}$. 
\end{definition}

Since the composition $\mathcal{R} \circ c_{\rho}$ is a constant map, the chain-rule implies the equality $d^1 \circ d^0 = 0$ .
 The inclusion $\mathcal{R}_G(\mathbf{\Sigma}, \mathbb{P}) \subset G^{\mathbb{G}}$ induces an injective morphism $i: T_{\rho} \mathcal{R}_G(\mathbf{\Sigma}, \mathbb{P}) \hookrightarrow \mathrm{C}^1(\mathbf{\Sigma}, \mathbb{P} ; \rho)$.
Moreover the inclusion $\mathbb{C}[\mathcal{X}_G(\mathbf{\Sigma}, \mathbb{P}) ]\subset \mathbb{C}[\mathcal{R}_G(\mathbf{\Sigma}, \mathbb{P})]$ induces a surjective morphism $p : T_{\rho} \mathcal{R}_G(\mathbf{\Sigma}, \mathbb{P}) \rightarrow T_{[\rho]} \mathcal{X}_G(\mathbf{\Sigma}, \mathbb{P})$.

\begin{lemma}\label{lemma_homo_discrete}

\begin{enumerate}
\item The image of the morphism $i$ is the space $ \mathrm{Z}^1(\mathbf{\Sigma}, \mathbb{P}; \rho)$ of cocycles. Hence $i$ induces an isomorphism $j : T_{\rho} \mathcal{R}_G(\mathbf{\Sigma}, \mathbb{P}) \xrightarrow{\cong}  \mathrm{Z}^1(\mathbf{\Sigma}, \mathbb{P}; \rho)$.
\item The following diagram commutes: 
$$\begin{tikzcd}
 T_{(e, \ldots, e)}G^{\mathbb{G}} 
 \arrow[r, "D_e c_{\rho}"] \arrow[d, "="'] &
T_{\rho} \mathcal{R}_G(\mathbf{\Sigma}, \mathbb{P})
\arrow[r,"p"] \arrow[d, "\cong"', "j"] &
T_{[\rho]}\mathcal{X}_G(\mathbf{\Sigma}, \mathbb{P})
\arrow[r] & 0 \\
\mathrm{C}^0\left(\mathbf{\Sigma}, \mathbb{P} ; \rho \right)
\arrow[r, "d^0"] &
\mathrm{Z}^1\left(\mathbf{\Sigma}, \mathbb{P} ; \rho \right)
\arrow[r] &
\mathrm{H}^1\left(\mathbf{\Sigma}, \mathbb{P} ; \rho \right)
\arrow[r] & 0
\end{tikzcd}$$

\par Moreover both lines in the above diagram are exact if $\rho$ is a good representation.
\end{enumerate}
\end{lemma}

\begin{proof}
Denote by $\mathcal{R}^* : \mathbb{C}[G]^{\otimes \mathbb{G}} \rightarrow \mathbb{C}[G]^{\otimes \mathbb{RL}}$ the morphism defining the regular map $\mathcal{R}$. The representation variety is defined as the co-image
$$  \mathbb{C}[G]^{\otimes \mathbb{RL}} \xrightarrow{\mathcal{R}^* - \eta^{\otimes \mathbb{G}} \circ\epsilon^{\otimes \mathbb{RL}}} \mathbb{C}[G]^{\otimes \mathbb{G}} \rightarrow \mathbb{C}[\mathcal{R}_G(\mathbf{\Sigma}, \mathbb{P} )] \rightarrow 0.$$
\par Denote by $\kappa : T_{\rho} G^{\otimes \mathbb{G}} \rightarrow T_e G^{\mathbb{RL}}$ the map sending a derivation $\varphi$ to $D_{\rho} \mathcal{R} (\varphi) - \eta^{\otimes \mathbb{RL}}$. The above exact sequence induces the following one
$$ T_{\rho}\mathcal{R}_G(\mathbf{\Sigma}, \mathbb{P}) \xrightarrow{i} T_{\rho}G^{\mathbb{G}} \xrightarrow{\kappa} T_e G^{\mathbb{RL}} \rightarrow 0.$$
\par Since the maps $\kappa$ and $D_e\mathcal{R}= d^1$ have the same kernel, this proves the first assertion. The proof that the first line, in the diagram of the second assertion, is exact is a straightforward adaptation of the argument in the  proof of Theorem \ref{theorem_cohomology} using Luna slice theorem. The commutativity of the diagram follows from the definition $d^0 := D_e c_{\rho}$. This concludes the proof.

\end{proof}

\par Lemma \ref{lemma_homo_discrete} implies that the morphism $j$ induces an isomorphism $\Lambda^{\mathbb{P}} : T_{[\rho]} \mathcal{X}_G(\mathbf{\Sigma}, \mathbb{P}) \xrightarrow{\cong} \mathrm{H}^1\left( \mathbf{\Sigma}, \mathbb{P} ; \rho \right)$. The isomorphism  $\Psi^{\mathbb{P}} : \mathcal{X}_G(\mathbf{\Sigma}) \xrightarrow{\cong} \mathcal{X}_G(\mathbf{\Sigma}, \mathbb{P})$ of Proposition \ref{proposition_discrete_model} induces an isomorphism $T_{[\rho]} \Psi^{\mathbb{P}} : T_{[\rho]} \mathcal{X}_G(\mathbf{\Sigma}) \xrightarrow{\cong} T_{[\rho^{\mathbb{P}}]} \mathcal{X}_G(\mathbf{\Sigma}, \mathbb{P})$. 
Define the isomorphism $\theta^{\mathbb{P}} : \mathrm{H}^1\left( \Sigma, \mathcal{A} ; \rho \right) 
\xrightarrow{\cong} \mathrm{H}^1(\mathbf{\Sigma}, \mathbb{P} ; \rho{\mathbb{P}})$ as the composition

$$ \theta^{\mathbb{P}} :  \mathrm{H}^1\left( \Sigma, \mathcal{A} ; \rho \right)
 \xrightarrow{\Lambda^{-1}} T_{[\rho]} 
 \mathcal{X}_G(\mathbf{\Sigma})
 \xrightarrow{T_{[\rho]} \Psi^{\mathbb{P}}} T_{[\rho_{\mathbb{P}}]}
  \mathcal{X}_G(\mathbf{\Sigma}, \mathbb{P})
   \xrightarrow{\Lambda^{\mathbb{P}}} 
    \mathrm{H}^1\left( \mathbf{\Sigma}, \mathbb{P} ; \rho_{\mathbb{P}} \right).$$
Unfolding the definitions, the isomorphism $\theta^{\mathbb{P}}$ sends a class $[c^1]$ to a class $[c^1_{\mathbb{P}}]$ where $c^1_{\mathbb{P}}= \oplus_{\beta \in \mathbb{G}} c^1(\beta, 1_{t(\beta)})$.

\begin{remark}\label{remark_GHJW}
Let $\mathbf{\Sigma}$ be a connected marked surface  with non-trivial boundary and exactly one boundary arc per boundary component and consider the presentation $\mathbb{P}$ of $\Pi_1(\Sigma)$ defined in $6$-th item of Exemple \ref{example_finite_presentations}. In this case, the isomorphism  $\Lambda^{\mathbb{P}} : T_{[\rho]} \mathcal{X}_G(\mathbf{\Sigma}, \mathbb{P}) \xrightarrow{\cong} \mathrm{H}^1\left( \mathbf{\Sigma}, \mathbb{P} ; \rho \right)$ was defined by Guruprasad-Huebschmann-Jeffrey-Weinstein in \cite{GHJW_ModSpacesParBd}.
\end{remark}

\subsection{The intersection form}

\par We first recall from (\cite{DrinfeldrMatrix}, \cite[Section $2.1$]{ChariPressley})  the definition of a classical $r$-matrix. Let $\tau \in \mathfrak{g}^{\otimes 2}$ be the invariant bi-vector dual to the non-degenerate pairing $\left(\cdot, \cdot\right)$. A \textit{classical }$r$-\textit{matrix } is an element $r\in \mathfrak{g}^{\otimes 2}$ such that:

\begin{enumerate}
\item The symmetric part $\frac{r+\sigma(r)}{2}$ of $r$ is the invariant bi-vector $\tau$.
\item The bi-vector $r$ satisfies the following classical Yang-Baxter equation: $$ 0= [r_{12}, r_{13}] + [r_{12}, r_{23}] + [r_{13}, r_{23}] \in \mathfrak{g}^{\otimes 3}$$
\end{enumerate}

\par If $\mathfrak{g}$ is a simple Lie algebra and $\mathfrak{g}=\mathfrak{n}^-\oplus \mathfrak{h} \oplus \mathfrak{n}^+$ a Cartan decomposition, the invariant bi-vector decomposes as $\tau = \tau^0 + \tau^{-+} + \tau^{+-}$ where $\tau^0 \in \mathfrak{h}^{\otimes 2}$, $\tau^{-+}\in \mathfrak{n}^-\otimes \mathfrak{n}^+$ and $\tau^{+-}=\sigma(\tau^{-+}) \in \mathfrak{n}^+\otimes \mathfrak{n}^-$. 

\begin{definition}
We define the classical $r$-matrices $r^{\pm}$ by the formulas $r^+ := \tau^0 + 2 r^{+-}$ and $r^-:= \tau^0 +2 r^{-+}$. 
\end{definition}

For instance, suppose that $\mathfrak{g}=\mathfrak{sl}_2$ is identified with the space of traceless $2\times 2$ matrices and set $H:=\begin{pmatrix} 1 & 0 \\ 0& -1 \end{pmatrix}$, $E:= \begin{pmatrix} 0 & 1 \\ 0&0 \end{pmatrix}$ and $F:= \begin{pmatrix} 0&0 \\ 1 & 0 \end{pmatrix}$. Choosing the Killing form with invariant bi-vector $\tau = \frac{1}{2} H\otimes H + E\otimes F +F \otimes E$, we find $r^+ = \frac{1}{2} H\otimes H + 2 E\otimes F$ and $r^-= \frac{1}{2} H\otimes H +2 F\otimes E$.
\vspace{2mm}
\par If $\mathfrak{g}$ is abelian, we define $r^+ = r^- = \tau$. If $G$ is a complex reductive Lie group, its Lie algebra decomposes as a direct sum $\mathfrak{g}=\oplus_i \mathfrak{g}_i$ where each summand $\mathfrak{g}_i$ is either simple or abelian. In that case we define the classical $r$-matrices $r^{\pm}:= \oplus_i r_i^{\pm}$.

\vspace{2mm}
\par The goal of this subsection is to define a skew-symmetric pairing $\bigcap{}^{\mathfrak{o}} : \mathrm{H}_1(\Sigma, \mathcal{A}; \rho)^{\otimes 2} \rightarrow \mathbb{C}$ depending on a choice $\mathfrak{o}$ of orientation of each boundary arc of $\mathbf{\Sigma}$. 
\begin{notations} Given such an orientation $\mathfrak{o}$ and $a$ a boundary arc, we will write $\mathfrak{o}(a)=+$ if the $\mathfrak{o}$-orientation of $a$ agrees with the orientation  induced by the orientation of $\Sigma$ on its boundary, and write $\mathfrak{o}(a)=-$ if it is the opposite orientation. 
\end{notations}
\par Fix a representation $\rho\in \mathcal{R}_G(\mathbf{\Sigma})$. Let $\mathcal{C}_1 , \mathcal{C}_2$ be two curves of $\mathbf{\Sigma}$. Two geometric representatives $c_1$ and $c_2$ of  $\mathcal{C}_1$ and $ \mathcal{C}_2$ are said \textit{ in transverse position} if the images of $c_1$ and $c_2$ intersect transversally in $\Sigma \setminus \mathcal{A}$ along simple crossings. We denote by $\alpha_1$ and $\alpha_2$ the path representatives associated to $c_1$ and $c_2$.  For $i=1,2$, fix a vector $X_i \in T_{\rho(\alpha_i)}G$ with the additional assumption that $X_i$ is $G$-invariant if $\mathcal{C}_i$ is closed. If $v$ is a point of the image of $c_i$, it induces a decomposition $c_i = c_i^- \cdot c_i^+$ of the geometric representative and hence a decomposition $\alpha_i = \alpha_i^- \alpha_i^+$ of the path such that $t(\alpha_i^-)=s(\alpha_i^+)=v$. We denote by $X_i(v)\in \mathfrak{g}$ the vector $X_i(v):= \rho(\alpha_i^-)^{-1} X_i \rho(\alpha_i^+)^{-1}$. 
\vspace{2mm}
\par Let $v\in c_1 \cap c_2\subset\Sigma \setminus \mathcal{A}$ be an intersection point and denote by $e_1, e_2 \in T_v \Sigma$ the tangent vectors of $c_1$ and $c_2$ respectively at the point $v$. We define the sign intersection $\varepsilon(v)= +1$ if $(e_1, e_2)$ is an oriented basis of $T_v \Sigma$ and $\varepsilon(v)=-1$ else. Let $a$ be a boundary arc and denote by $S(a)$ the set of pairs $(v_1, v_2)$ of points such that $v_i \in c_i \cap a$. Note that $c_1$ and $c_2$ do not have intersection point in $a$ by definition. Given $(v_1, v_2) \in S(a)$, we define a sign $\varepsilon(v_1, v_2) \in \{ \pm 1 \}$ as follows. Isotope $c_1$ around $a$ to bring $v_1$ in the same position than $v_2$ and denote by $c_1'$ the new geometric curve. The isotopy should preserve the transversality condition and should not make appear any new inner intersection point. Define $e_1, e_2\in T_{v_2} \Sigma_{\mathcal{P}}$ the tangent vectors at $v_2$ of $c_1'$ and $c_2$ respectively. Define $\varepsilon(v_1, v_2) = +1$ if $(e_1, e_2)$ is an oriented basis and $\varepsilon(v_1, v_2)=-1$ else.
\vspace{2mm}
 \par Note that the orientation $\mathfrak{o}$ induces a total order $<_{\mathfrak{o}}$ on the set of elements of $a$. For $(v_1,v_2)\in S(a)$ we will write $\mathfrak{o}(v_1,v_2)= +1$ if $v_1 <_{\mathfrak{o}} v_2$ and $\mathfrak{o}(v_1,v_2)=-1$ if $v_2 <_{\mathfrak{o}} v_1$.

\begin{definition}
 Define a complex number $(c_1, X_1)\bigcap{}^{\mathfrak{o}} (c_2, X_2) \in \mathbb{C}$ by the formula

\begin{equation*} (c_1, X_1)\bigcap{}^{\mathfrak{o}} (c_2, X_2) := \sum_a \sum_{(v_1, v_2)\in S(a)} \varepsilon(v_1, v_2) \left( X_1(v_1)\otimes X_2(v_2), r^{\mathfrak{o}(v_1,v_2)} \right) +2\sum_{v\in c_1\cap c_2} \varepsilon(v) \left( X_1(v), X_2(v) \right).\end{equation*}
\end{definition}

In this formula, we have use the pairing $\left(\cdot, \cdot \right) :\mathfrak{g}^{\otimes 2} \otimes \mathfrak{g}^{\otimes 2} \rightarrow \mathbb{C}$ defined by $\left(x_1\otimes y_1, x_2 \otimes y_2 \right)= (x_1, x_2) (y_1, y_2)$. 

\begin{lemma}\label{lemma_intersection}
\begin{enumerate}
\item The number $(c_1, X_1)\bigcap{}^{\mathfrak{o}} (c_2, X_2)$ is independent on the choice of the geometric representative $c_1, c_2$ of $\mathcal{C}_1, \mathcal{C}_2$. Hence it induces a skew-symmetric pairing $\bigcap{}^{\mathfrak{o}} : \mathrm{Z}_1 \left(\Sigma, \mathcal{A} ; \rho \right)^{\otimes 2} \rightarrow \mathbb{C}$.
\item The pairing $\bigcap{}^{\mathfrak{o}}$ vanishes on the sub-space $(\mathrm{B}_1 \left(\Sigma, \mathcal{A} ; \rho \right)\otimes \mathrm{Z}_1 \left(\Sigma, \mathcal{A} ; \rho \right))\oplus (\mathrm{Z}_1 \left(\Sigma, \mathcal{A} ; \rho \right) \otimes \mathrm{B}_1 \left(\Sigma, \mathcal{A}; \rho \right))$. Therefore, it induces a skew-symmetric pairing
\begin{equation*}
\bigcap{}^{\mathfrak{o}} : \mathrm{H}_1(\Sigma, \mathcal{A}; \rho)^{\otimes 2} \rightarrow \mathbb{C}.\end{equation*}
\end{enumerate}
\end{lemma}

\begin{proof} First remark that if we denote by $c_1^{-1}$ the geometric curve defined by $c_1^{-1}(t) = c_1(1-t)$, then one has the equality
$$ (c_1, X_1)\bigcap{}^{\mathfrak{o}} (c_2, X_2) = (c_1^{-1}, -\rho(\alpha_1)^{-1} X_1 \rho(\alpha_1)^{-1}) \bigcap{}^{\mathfrak{o}} (c_2, X_2 )$$
\par Suppose that $(c_1, c_2)$ and $(c'_1, c'_2)$ are two pairs of geometric representatives of $\mathcal{C}_1, \mathcal{C}_2$. One can pass from the pair $(c_1,c_2)$ to the pair $(c'_1, c'_2)$ by a sequence of elementary moves which consist in the two moves drawn in Figure \ref{fig_moves_intersection} together with the elementary moves obtained from these two moves by changing the orientation of $c_1$ or $c_2$ or both. By the above formula, to prove the first point of the lemma,  it is sufficient to show the invariance of $ (c_1, X_1)\bigcap{}^{\mathfrak{o}} (c_2, X_2) $ by the two elementary moves of Figure \ref{fig_moves_intersection}.

\begin{figure}[!h] 
\centerline{\includegraphics[width=12cm]{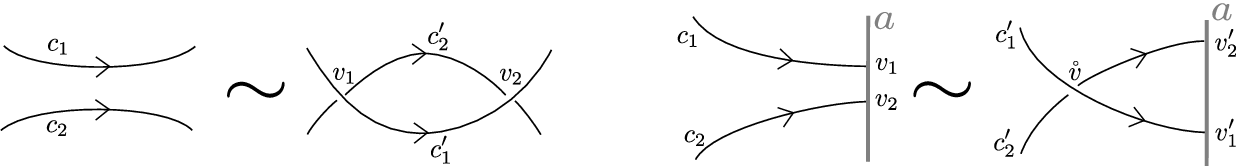} }
\caption{Two elementary moves for pairs of geometric representatives of a pair of  curves.  } 
\label{fig_moves_intersection} 
\end{figure} 

\vspace{2mm}
\par Suppose that $(c_1, c_2)$ and $(c'_1, c'_2)$ are two pairs which differ from the elementary move  drawn in the left part of Figure \ref{fig_moves_intersection}. Denote by $v_1$ and $v_2$ the two additional points induced by the move and, for $i=1,2$, decompose the paths $\alpha_i= \alpha_i^- \alpha_i^0 \alpha_i^+$ where $t(\alpha_i^-)=s(\alpha_i^0)= v_1$ and $t(\alpha_i^0)=s(\alpha_i^+)=v_2$. Note that $\alpha_1^0=\alpha_2^0=: \alpha_0$. We compute:
\begin{multline*}
(c'_1, X_1)\bigcap{}^{\mathfrak{o}} (c'_2, X_2)  - (c_1, X_1)\bigcap{}^{\mathfrak{o}} (c_2, X_2)  
=  2\varepsilon(v_1) (X_1(v_1), X_2(v_1)) + 2\varepsilon(v_2) (X_1(v_2), X_2(v_2)) \\
= (+2) \left( \rho(\alpha_1^-)^{-1} X_1 \rho(\alpha_1^+)^{-1} \rho(\alpha^0) ^{-1} , \rho(\alpha_2^-) ^{-1} X_2 \rho(\alpha_2^+)^{-1}\rho(\alpha^0)^{-1} \right) \\
  +(-2) \left( \rho(\alpha^0)^{-1}\rho(\alpha_1^-)^{-1} X_1 \rho(\alpha_1^+)^{-1}  ,\rho(\alpha^0)^{-1} \rho(\alpha_2^-) ^{-1} X_2 \rho(\alpha_2^+)^{-1} \right) =0
 \end{multline*}
 \par In the last line, we used the $G$-invariance of the pairing $\left(\cdot, \cdot \right)$. Next suppose that $(c_1, c_2)$ and $(c'_1, c'_2)$ are two pairs which differ from the elementary move drawn in the right part of Figure \ref{fig_moves_intersection}. Denote by $a$ the boundary arc and $\mathring{v}, v_1, v_2, v'_1, v'_2$ the points defined in Figure \ref{fig_moves_intersection}. For $i=1,2$, decompose the paths $\alpha'_i = \alpha_i^- \alpha_i^0$ such that $t(\alpha_i^-)=s(\alpha_i^0)=\mathring{v}$ and $t(\alpha_i^0)=v'_i$. Note that $\alpha_1^0=\alpha_2^0=: \alpha^0$. We compute:
 
 \begin{align*}
 &(c'_1, X_1)\bigcap{}^{\mathfrak{o}} (c'_2, X_2)  - (c_1, X_1)\bigcap{}^{\mathfrak{o}} (c_2, X_2)  \\
 & = 2\varepsilon(\mathring{v}) \left( X_1(\mathring{v}), X_2(\mathring{v}) \right) +\varepsilon(v_1', v_2') \left( X_1(v'_1)\otimes X_2(v'_2), r^{\mathfrak{o}(v_1', v_2')} \right) \\
 & -\varepsilon(v_1, v_2) \left(X_1(v_1)\otimes X_2(v_2), r^{\mathfrak{o}(v_1,v_2)} \right) \\
 & = 2 \left( \rho(\alpha_1^-)^{-1} X_1 \rho(\alpha^0)^{-1} , \rho(\alpha_1^-)^{-1} X_2 \rho(\alpha^0)^{-1} \right) - \left( \rho(\alpha_1)^{-1} X_1 \otimes \rho(\alpha_2)^{-1} X_2, r^{-\mathfrak{o}(a)} \right) \\ 
 &- \left( \rho(\alpha_1)^{-1} X_1 \otimes \rho(\alpha_2)^{-1} X_2 , r^{\mathfrak{o} (a)} \right) \\
  & = \left( \rho(\alpha_1)^{-1}X_1 \otimes \rho(\alpha_2)^{-1} X_2, 2 \tau - r^{-\mathfrak{o} (a)} -r^{\mathfrak{o}(a)} \right) = 0
   \end{align*}
In the above equalities we used both the $G$ invariance of the pairing $\left(\cdot, \cdot \right)$ and the fact that $\tau$ is the symmetric part of the $r$-matrix $r^{\mathfrak{o}(a)}$.

\vspace{2mm}
\par To prove the  second part of the lemma first note  that  $\mathrm{B}_1 \left(\Sigma, \mathcal{A} ; \rho \right)$ is spanned by co-boundary elements of the form $ \partial_2 \left< (\alpha_2, \alpha_1, 1_{t(\alpha_1)}), X \right>$ such that the paths $\alpha_2, \alpha_1$ and $\alpha_2 \alpha_1$ admit geometric representatives $c_1, c_2$ and $c_{12}$ respectively. Since 
$$\partial_2 \left< (\alpha_2, \alpha_1, 1), X \right> = \left< (\alpha_1, 1), \rho(\alpha_2)^{-1}X \right> + \left< (\alpha_2, 1), X\rho(\alpha_1)^{-1} \right> - \left< (\alpha_2\alpha_1, 1),X \right>$$
 we need to prove that for any geometric path $c'$ transverse to $c_1$, $c_2$ and $c_{12}$, one has the equality
$$ (c_{12}, X) \bigcap{}^{\mathfrak{o}} (c', Y) = (c_1, \rho(\alpha_2)^{-1}X) \bigcap{}^{\mathfrak{o}} (c', Y) + (c_2, X \rho(\alpha_1)^{-1}) \bigcap{}^{\mathfrak{o}} (c',Y).$$
This equality follows from a straightforward computation. 

\end{proof}

\begin{definition} We call \textit{intersection form} the skew-symmetric pairing:
$$\bigcap{}^{\mathfrak{o}} : \mathrm{H}_1(\Sigma, \mathcal{A}; \rho)^{\otimes 2} \rightarrow \mathbb{C}.$$
\end{definition}

\par We now show that the intersection form behaves well for the gluing operation. Let $a$ and $b$ be two boundary arcs of $\mathbf{\Sigma}$. Recall that we defined a map 
$j_{a\#b} : \mathrm{H}_1 \left( \Sigma_{a\#b}, \mathcal{A}_{a\#b} ; \rho_{a\#b} \right) \rightarrow \mathrm{H}_1 \left( \Sigma, \mathcal{A}; \rho \right)$ characterized by Lemma \ref{lemma_decomp_paths}. We choose an orientation $\mathfrak{o}$ of the boundary arcs of $\mathbf{\Sigma}$ such that the gluing map $\varphi$ preserves the orientation of $a$ and $b$. This is equivalent to the equality of the signs $\mathfrak{o}(a)=-\mathfrak{o}(b)$. Still denote by $\mathfrak{o}$ the induced orientation of the boundary arcs of $\mathbf{\Sigma}_{a\#b}$. 

\begin{lemma}\label{lemma_gluing_intersection} 
The following diagram commutes: 
$$\begin{tikzcd}
 \mathrm{H}_1 \left(\Sigma_{a\#b}, \mathcal{A}_{a\#b} ; \rho_{a\#b} \right)^{\otimes 2} 
\arrow[rr, "\bigcap{}^{\mathfrak{o}}"] \arrow[d, hook,  "(j_{a\#b})^{\otimes 2}"']  &&\mathbb{C} \\
 \mathrm{H}_1 \left( \Sigma, \mathcal{A} ; \rho \right)^{\otimes 2}
 \arrow[rru, "\bigcap{}^{\mathfrak{o}}"'] & &
 \end{tikzcd}$$

\end{lemma}

\begin{proof}
Denote by $c$ the image in $\Sigma_{a\#b}$ of the boundary arcs $a$ and $b$ and denote by $\pi$ the projection map. Consider $[\mathcal{C}_1, X_1], [\mathcal{C}_2, X_2] \in  \mathrm{H}_1 \left(\Sigma_{a\#b}, \mathcal{A}_{a\#b} ; \rho_{a\#b} \right)$ two generators and choose some geometric representatives $c_1$ and $c_2$ in transverse position such that $c_1\cap c_2\cap c = \emptyset$. Denote by $\alpha$ and $\beta$ the paths in $\Pi_1(\Sigma)$ representing the classes of $c_1$ and $c_2$ respectively. For $i=1, 2$ choose a decomposition $c_i = c_i^1 \ldots c_i^{n_i}$ such that the interior of each sub-arc $c_i^j$ does not intersect $c$. The arcs $c_i^j$ lift through the projection $\pi$ to arcs in $\Sigma$ whose classes in $\Pi_1(\Sigma)$ will be denoted $\alpha_i^j$. Denote by $X_{i,j}$ the vector $X_{i,j} := \rho(\alpha_i^1 \ldots \alpha_i^{j-1})^{-1} X \rho(\alpha_i^{j+1} \ldots \alpha_i^{n_i})^{-1} \in T_{\rho(\alpha_i^j)}G$.

By Lemma \ref{lemma_decomp_paths}, one has the equality for $i=1,2$
$$ j_{a\#b} ([\mathcal{C}_i, X_i]) = \sum_j [\alpha_i^j , X_{i,j}]. $$
Hence, one has 
\begin{equation}\label{equationalacon} j_{a\#b}([\mathcal{C}_1, X_1]) \bigcap{}^{\mathfrak{o}} j_{a\#b}([\mathcal{C}_2, X_2]) = \sum_{i,j} [\alpha_1^i, X_1^i ] \bigcap{}^{\mathfrak{o}} [\alpha_2^j , X_2^j]. \end{equation}
\par The projection map $\pi$ induces a bijection between the set of intersection points in $c_1 \cap c_2$ and the disjoint union of intersection points in  $\bigsqcup_{i,j} c_1^i \cap c_2^j$. Moreover the contribution of the points of $c_1 \cap c_2$ in the formula defining $[\mathcal{C}_1, X_1] \bigcap{}^{\mathfrak{o}} [\mathcal{C}_2, X_2 ]$ is equal to the contribution of the points of $\bigsqcup_{i,j} c_1^i \cap c_2^j$ in the right-hand-side of Equation \eqref{equationalacon}. If $d$ represents a boundary arc of $\mathbf{\Sigma}$ distinct from $a$ and $b$, the projection $\pi$ induces a bijection between the set of pairs $(v_1, v_2) \in d\cap c_1 \times d\cap c_2$ and the disjoint union over the indexes $i$ and $j$ of set of pairs $(v_1, v_2) \in \pi(d) \cap c_1^i \times \pi(d)\cap c_2^j$. 
Moreover the contribution of the pairs of points in $d$ in the formula defining $[\mathcal{C}_1, X_1] \bigcap{}^{\mathfrak{o}} [\mathcal{C}_2, X_2 ]$ is equal to the contribution of the points in $\pi(d)$ in the right-hand-side of Equation \eqref{equationalacon}. Denote by $S(c)$ the set of pairs $(v_1, v_2) \in c_1 \cap c \times c_2 \cap c$. Also define $S(a)$ the set of pairs $(v_1, v_2)$ such that there exists some indexes $i$ and $j$ such that $v_1 \in c_1^i \cap a$ and $v_2 \in c_2^j \cap a$. 
Define $S(b)$ in the same manner. Each pair $(v_1, v_2) \in S(c)$ induces exactly two pairs $(v_1^a, v_2^a) \in S(a)$ and $(v_1^b, v_2^b) \in S(b)$ corresponding to the lifts by the projection $\pi$ of the points $v_1$ and $v_2$. Hence the sets $S(c), S(a)$ and $S(b)$ are in natural bijection. By the preceding discussion, one has:

\begin{align*}
& j_{a\#b}([\mathcal{C}_1, X_1]) \bigcap{}^{\mathfrak{o}} j_{a\#b}([\mathcal{C}_2, X_2])  - [\mathcal{C}_1, X_1] \bigcap{}^{\mathfrak{o}} [\mathcal{C}_2, X_2]  = \sum_{(v_1, v_2)\in S(a)\cup S(b)} \varepsilon(v_1, v_2) \left( X_1(v_1)\otimes  X_2(v_2) , r^{\mathfrak{o}(v_1, v_2)} \right) \\
& = \sum_{(v_1, v_2) \in S(c) } \varepsilon(v_1^a, v_2^a) \left( X_1(v_1^a)\otimes  X_2(v_2^a) , r^{\mathfrak{o}(v_1^a, v_2^a)} \right) +\varepsilon(v_1^b, v_2^b ) \left( X_1(v_1^b)\otimes  X_2(v_2^b) , r^{\mathfrak{o}(v_1^b, v_2^b)} \right)
\end{align*}

\par Remark that for each pair $(v_1, v_2)\in S(c)$,  one has the equalities $\mathfrak{o}(v_1^a, v_2^a) = \mathfrak{o}(v_1^b, v_2^b)$, $X_i(v_i^a)= X_i (v_i^b)$ and $\varepsilon(v_1^a, v_2^a) = - \varepsilon(v_1^b, v_2^b)$. Therefore, the above sum vanishes and we have proved the lemma.

\end{proof}

\section{Poisson structure}\label{sec_Poisson}

\subsection{Definition of the Poisson bracket}
\par Given $M$ a smooth manifold, a Poisson structure on $M$ is a bi-vector field $w\in \Lambda^2 TM$ such that the Schouten bracket $[w, w]_S$ vanishes. Such a bi-vector endows the algebra $C^{\infty}(M)$ of smooth functions with a Poisson bracket $\left\{ \cdot , \cdot \right\}$ defined by the formula $\{ f, h \} (x) := \left< D_xf \otimes D_x h , w_x\right>$ (see e.g. \cite{ChariPressley, GPVanhaecke} for details). In this section, given an orientation $\mathfrak{o}$ of the boundary arcs of a punctured surface $\mathbf{\Sigma}$ and a finite presentation $\mathbb{P}$ of the fundamental groupoid, we want to define a Poisson bracket $\left\{\cdot, \cdot \right\}^{\mathfrak{o}}$ on the algebra $\mathbb{C}[\mathcal{X}_G(\mathbf{\Sigma}, \mathbb{P})]$. Since the tangent space at a point $[\rho] \in \mathcal{X}_G(\mathbf{\Sigma}, \mathbb{P})$ identifies with the twisted groupoid homology, the naive idea is to define an element $w_{\rho}^{\mathfrak{o}} \in \Lambda^2 \mathrm{H}_1(\mathbf{\Sigma}, \mathbb{P}; \rho) $ and then define a Poisson bracket using the formula $\left\{ f, h\right\}^{\mathfrak{o}} ([\rho])= \left< (\Lambda^{\mathbb{P}})^{\otimes 2} (D_{[\rho]}f \otimes D_{[\rho]}h), w_{\rho}^{\mathfrak{o}} \right>$. However, since we deal with an affine variety rather than a smooth manifold and we care about the algebra of regular functions rather than the algebra of smooth functions,  we need to formulate the construction in the algebraic setting. 

\vspace{2mm}
\par We first recall some basic algebraic facts from \cite{GPVanhaecke}. Let $A$ be a commutative algebra and $M$ a bimodule. Denote by $\mathfrak{X}^n (A, M)$ the $A$-module of $n$ skew symmetric forms $P\in \Hom_A\left( \Lambda^n A, M\right)$ such that $P$ is a derivation in each of its variables. The graded algebra $\mathfrak{X}^{\bullet}(A,M):= \oplus_{n\geq 0} \mathfrak{X}^n(A,M)$ has a structure of Gerstenhaber algebra $\left(\mathfrak{X}^{\bullet}, \wedge, [\cdot, \cdot]_S\right)$ where $\wedge$ represents the wedge product and $[\cdot, \cdot]_S$ is the Schouten bracket. If $\chi : M_1 \rightarrow M_2$ is a morphism of $A$-bimodules, there is a well defined morphism $\chi_* : \mathfrak{X}^{\bullet}(A,M_1)\rightarrow \mathfrak{X}^{\bullet}(A,M_2)$ sending $P$ to $\chi\circ P$. If $M=A$, we simply denote by $\mathfrak{X}^{\bullet}(A)$ the algebra $\mathfrak{X}^{\bullet}(A,A)$. If $X$ is an affine variety, the algebra $\mathfrak{X}^{\bullet}(\mathbb{C}[X])$ plays the same role than the Gerstenhaber algebra $\Lambda^{\bullet} TM$ in differential geometry. It follows from the definition of the Schouten bracket that a bi-derivation $P\in \mathfrak{X}^2(\mathbb{C}[X])$ is a Poisson bracket if and only if the Schouten bracket $[P,P]_S$ vanishes. If $x\in X$ is represented by a character $\chi_x : \mathbb{C}[X]\rightarrow \mathbb{C}$ and $\mathbb{C}_{\chi_x}$ represents the corresponding $\mathbb{C}[X]$ bimodule structure on $\mathbb{C}$, we will denote by $P_x \in \mathfrak{X}^n(\mathbb{C}[X], \mathbb{C}_{\chi_x})$ the derivation associated to an element $P\in \mathfrak{X}^n(\mathbb{C}[X])$ by the formula $P_x:= \chi_x \circ P$. Note that $\mathfrak{X}^1(\mathbb{C}[X], \mathbb{C}_{\chi_x})$ is, by definition, the Zariski tangent space $T_x X$.

\begin{notations} Note that the orientation of $\Sigma$ induces an orientation of its boundary and thus of its boundary arcs. For $\mathfrak{o}$ an orientation of the boundary arcs of $\mathbf{\Sigma}$ and $a$ a boundary arc, we write $\mathfrak{o}(a)=+1$ of the $\mathfrak{o}$-orientation of $a$ coincides with the one induced by $\Sigma$ and write $\mathfrak{o}(a)=-1$ elsewhere. 
\end{notations}

\subsubsection{The case of the bigon}\label{sec_bigon_Poisson} We first consider the case where $\mathbf{\Sigma}= \mathbb{B}$. Recall from Example \ref{example_finite_presentations} that the bigon is endowed with a canonical presentation $\mathbb{P}^{\mathbb{B}}$ with only generators $\alpha^{\pm}$, where $\alpha$ is a path such that $s(\alpha)$ lies in some boundary arc $a$ and $t(\alpha)$ lies in the other boundary arc $b$. The morphism $\mathcal{X}_G(\mathbb{B})\rightarrow G$ sending a class $[\rho]$ to $\rho(\alpha)$ is an isomorphism by Proposition \ref{proposition_discrete_model}. Let $\mathfrak{o}$ be an orientation of the boundary arcs, writing $\varepsilon_1:= \mathfrak{o}(a)$ and $\varepsilon_2:=\mathfrak{o}(b)$, we want to define a Poisson bracket $\left\{ \cdot, \cdot \right\}^{\varepsilon_1, \varepsilon_2}$ on the algebra $\mathbb{C}[G]\cong \mathbb{C}[\mathcal{X}_G(\mathbb{B})]$. The Lie group $G$ has a Poisson bi-vector field $w^{\varepsilon_1, \varepsilon_2} \in \Lambda^2 TG$ defined at $g\in G$ by the formula $w_g^{\varepsilon_1, \varepsilon_2}:= \overline{r}^{\varepsilon_1} (g\otimes g) + (g\otimes g) \overline{r}^{\varepsilon_2}$. Here we denoted by $\overline{r}^{\pm}$ the skew-symmetric part of $r^{\pm}$. It is a classical fact (\cite[Proposition $2.2.2$]{ChariPressley}) that the classical Yang-Baxter equation implies that the Schouten bracket $[w^{\varepsilon_1, \varepsilon_2}, w^{\varepsilon_1, \varepsilon_2}]$ vanishes, hence the algebra $C^{\infty}(G)$ has a Poisson bracket defined by

 $$\left\{ f, h \right\}^{\varepsilon_1, \varepsilon_2} (g) = \left< D_g f \otimes D_g h, \overline{r}^{\varepsilon_1} (g\otimes g) + (g\otimes g) \overline{r}^{\varepsilon_2} \right>. $$

\par Remark that only the brackets $\left\{ \cdot, \cdot \right\}^{-, +}$ and $\left\{ \cdot, \cdot \right\}^{+,-}$ endow $G$ with a Poisson Lie structure, i.e. are such that the product in $G$ is a Poisson morphism. We now translate the preceding discussion in algebraic terms. Denote by $\mathbb{C}_{\epsilon}$ the $\mathbb{C}[G]$ bimodule associated to the neutral element $e\in G$, that is such that $f\cdot z = z\cdot f = \epsilon(f)z$. Let $r\in \mathfrak{g}^{\otimes 2}$ be a classical $r$-matrix and $\overline{r}$ its skew-symmetric part. Fix $(X_i)_i$ a basis of the Lie algebra $\mathfrak{g}= \mathfrak{X}^1(\mathbb{C}[G], \mathbb{C}_{\epsilon})$ and decompose $\overline{r}$ as $\overline{r}=\sum_{i,j} \overline{r}^{ij} X_i \wedge X_j$. The left-translation map $L_g : G\rightarrow G$ defined by $L_g(h)=gh$ is a regular map with associated algebra morphism $L_g^* :\mathbb{C}[G]\rightarrow \mathbb{C}[G]$ defined by $L_g^* = (\chi_g \otimes \id)\circ \Delta$. Define a bi-derivation $P^{L,\overline{r}} \in \mathfrak{X}^2(\mathbb{C}[G])$ by the formula

\begin{equation*}
 P^{L,\overline{r}}:= \sum_{i,j} \overline{r}^{ij} \mu\circ \left[ \left( (\id\otimes X_i)\circ \Delta\right) \otimes \left( (\id\otimes X_j)\circ \Delta \right)  - \left( (\id\otimes X_j)\circ \Delta\right) \otimes \left( (\id\otimes X_i)\circ \Delta \right) \right]. 
 \end{equation*}

 
For $g\in G$ represented by a character $\chi_g$,  denote by $P^{L,\overline{r}}_g\in \mathfrak{X}^2(\mathbb{C}[G], \mathbb{C}_{\chi_g})$ the bi-derivation $P^{L,\overline{r}}:= \chi_g \circ P^{L,\overline{r}}$.  One has the equalities

\begin{eqnarray*} 
P^{L,\overline{r}}_g &=& \sum_{i,j} \overline{r}^{ij} \mu\circ \left[ \left( (\chi_g\otimes X_i)\circ \Delta\right) \otimes \left( (\chi_g\otimes X_j)\circ \Delta \right) - \left( (\chi_g\otimes X_j)\circ \Delta\right) \otimes \left( (\chi_g\otimes X_i)\circ \Delta \right) \right] \\
&=& \sum_{i,j} \overline{r}^{ij} \left[ (X_i \wedge X_j ) \circ (L_g^*)^{\otimes 2} \right] =  (D_e L_g)^{\otimes 2} (\overline{r})= (g\otimes g)\overline{r} 
\end{eqnarray*}

\par Similarly, define the bi-derivation $P^{R,\overline{r}} \in \mathfrak{X}^2(\mathbb{C}[G])$ by the formula

\begin{equation*}
 P^{R,\overline{r}}:= \sum_{i,j} \overline{r}^{ij} \mu\circ \left[ \left( (X_i \otimes \id)\circ \Delta\right) \otimes \left( (X_j \otimes \id)\circ \Delta \right)  - \left( (X_j \otimes \id)\circ \Delta\right) \otimes \left( (X_i \otimes \id)\circ \Delta \right) \right]. 
 \end{equation*}

\par A similar computation shows that $P_g^{R, \overline{r}} = \overline{r} (g\otimes g)$. 

\begin{definition}
Define the Poisson bracket $\left\{ \cdot, \cdot \right\}^{\varepsilon_1, \varepsilon_2}\in \mathfrak{X}^2(\mathbb{C}[G])$ by the formula $\left\{ \cdot, \cdot \right\}^{\varepsilon_1, \varepsilon_2} := P^{R, \overline{r}^{\varepsilon_1}} + P^{L, \overline{r}^{\varepsilon_2}}$. 
\end{definition}

This is an algebraic analog of the differential Poisson structure defined previously and the vanishing of the Schouten bracket $[ \left\{ \cdot, \cdot \right\}^{\varepsilon_1, \varepsilon_2}, \left\{ \cdot, \cdot \right\}^{\varepsilon_1, \varepsilon_2} ]_S$ follows from the classical Yang-Baxter equation by a similar argument than in the differential geometric setting. We denote by $\mathbb{C}[G]^{\varepsilon_1, \varepsilon_2}$ the algebra $\mathbb{C}[G]$ equipped with the Poisson bracket $\left\{\cdot, \cdot, \right\}^{\varepsilon_1, \varepsilon_2}$.

\begin{lemma}\label{lemma_poissonlie}
For $\varepsilon_1, \varepsilon_2, \varepsilon \in \{\pm \}$, the following assertions hold:
\begin{enumerate}
\item The co-product $\Delta : \mathbb{C}[G]^{\varepsilon_1, \varepsilon_2} \rightarrow \mathbb{C}[G]^{\varepsilon_1, \varepsilon} \otimes \mathbb{C}[G]^{-\varepsilon, \varepsilon_2}$ is a Poisson morphism.
\item The antipode $S: \mathbb{C}[G]^{\varepsilon_1, \varepsilon_2} \rightarrow \mathbb{C}[G]^{-\varepsilon_1, -\varepsilon_2}$ is Poisson morphism.
\end{enumerate}
\end{lemma}

\begin{proof} The proof is a straightforward computation. \end{proof}

\begin{remark}\label{remark_PoissonBigon}
Consider an embedding $G\subset \GL_N(\mathbb{C})$ so that $\mathbb{C}[G]$ is a quotient of $\mathbb{C}[GL_N]=\quotient{\mathbb{C}[x_{i,j}, \det^{-1}| 1\leq i,j \leq N]}{(\det \det^{-1}=1)}$ and $\mathbb{C}[G]$ is generated by the classes of the matrix coefficient functions $x_{i,j}$. Consider the $N\times N$ matrix $M(\alpha)=\{ x_{i,j} \}_{1\leq i,j\leq N}$ with coefficients in $\mathbb{C}[G]$. Then the Poisson bracket $\{ \cdot, \cdot \}^{\varepsilon_1, \varepsilon_2}$ is described by the down-to-earth formula:
$$ \{ M(\alpha) \otimes M(\alpha) \}^{\varepsilon_1, \varepsilon_2} = \overline{r}^{\varepsilon_1} (M(\alpha)\otimes M(\alpha)) +(M(\alpha)\otimes M(\alpha)) \overline{r}^{\varepsilon_2}.$$
Here we used the classical notation $\{N \otimes N\}$ to denote the matrix defined by $\{N\otimes N \}_{i j k l}= \{x_{i,j}, x_{k,l} \}$ and $\overline{r}$ is the skew-symmetric part of $r$.
\end{remark}

\subsubsection{The case of the triangle}  We next consider the triangle $\mathbb{T}$. Fix an orientation $\mathfrak{o}$ of its three boundary arcs. Recall from Example \ref{example_finite_presentations} that the fundamental groupoid of the triangle has a canonical presentation $\mathbb{P}^{\mathbb{T}}$ with six generators $\mathbb{G}=\{\beta_1^{\pm 1},\beta_2^{\pm 1},\beta_3^{\pm 1}\}$ and a unique non-trivial relation $R=\beta_1\star\beta_2\star \beta_3$. If $v\in \mathbb{V}$ belongs to a boundary arc $a$, we write $r_v:=r^{\mathfrak{o}(a)}$. 

\vspace{2mm}\par
We first define a Poisson structure on the affine variety $G^{\mathbb{G}}$ as follows. If $A$ and $B$ are two commutative algebras, there is a morphism $\mathfrak{X}^1(A)\otimes \mathfrak{X}^1(B) \rightarrow \mathfrak{X}^1(A\otimes B)$ sending $v_1\oplus v_2$ to the derivation $v$ defined by $v(a\otimes b) := v_1(a)\otimes b + a\otimes v_2(b)$. Hence we have a natural morphism $\oplus_{\beta \in \mathbb{G}} \mathfrak{X}^1(\mathbb{C}[G]) \rightarrow \mathfrak{X}^1(\mathbb{C}[G]^{\otimes \mathbb{G}})$ (corresponding to the morphism $\oplus_{\beta \in \mathbb{G}} TG \rightarrow T G^{ \mathbb{G}}$ in differential geometry). For each generator $\delta \in \mathbb{G}$ denote by $i_{\delta} : \mathfrak{X}^1(\mathbb{C}[G]) \rightarrow \mathfrak{X}^1(\mathbb{C}[G]^{\otimes \mathbb{G}})$ the corresponding embedding. If $X\in \mathfrak{g}$ and $\delta \in \mathbb{G}$,  denote by $X^{\delta} \in \mathfrak{X}^1(\mathbb{C}[G]^{\otimes \mathbb{G}})$ the sum $X^{\delta}:= i_{\delta} \left( (\id\otimes X)\circ \Delta \right) + i_{\delta^{-1}}\left((X\otimes \id)\circ \Delta \right)$. Given $r=\sum_{i,j}r^{ij} X_i \wedge X_j$ a classical $r$-matrix  and $\delta_1, \delta_2 \in \mathbb{G}$ two generators, define the bi-derivation $P_r^{\delta_1, \delta_2} \in \mathfrak{X}^2(\mathbb{C}[G]^{\otimes \mathbb{G}} )$ by the formula
$$ P_r^{\delta_1, \delta_2} = \sum_{i,j} r^{ij} X_i^{\delta_1} \wedge X_j^{\delta_2}. $$

\begin{definition}
We eventually define a Poisson bracket $P^{\mathbb{T}} \in \mathfrak{X}^2(\mathbb{C}[G]^{\otimes \mathbb{G}})$ by the formula
$$ P^{\mathbb{T}} := \frac{1}{2} \left(\sum_{\delta \in \mathbb{G}} P_{s(\delta)}^{\delta, \delta}  \right) + \sum_{i\in \mathbb{Z}/3\mathbb{Z}} P_{r_{v_{i+1}}}^{\beta_i, \beta_{i+1}^{-1}}.
$$
\end{definition}

\par Note that if $\rho= (\rho(\delta))_{\delta\in \mathbb{G}} \in \mathcal{R}_G(\mathbb{T}, \mathbb{P}^{\mathbb{T}}) \subset G^{\mathbb{G}}$, then $X^{\alpha}_{\rho} = (\rho(\alpha)X)\oplus (-X\rho(\alpha)^{-1}) \in T_{\rho(\alpha)}G \oplus T_{\rho(\alpha^{-1})}G \subset T_{\rho} G^{\mathbb{G}}$ and the above formula defining $P^{\mathbb{T}}$ is similar to the formula used by Fock and Rosly in \cite{FockRosly} to define a Poisson structure in the geometric differential context. Remark that if $(\delta_1, \delta_2)$ and $(\delta'_1, \delta'_2)$ are two distinct pairs of generators in $\mathbb{G}$, then the Schouten bracket $[P_v^{\delta_1, \delta_2}, P_{v'}^{\delta'_1, \delta'_2} ]_S$ vanishes for any $v,v'$. Moreover, it follows from the fact that $r$ is a classical $r$-matrix that the Schouten bracket $[P_v^{\delta_1, \delta_2}, P_v^{\delta_1, \delta_2} ]_S$ also vanishes. Hence we have $[P^{\mathbb{T}}, P^{\mathbb{T}}]_S=0$ and $P^{\mathbb{T}}$ is a Poisson bracket on the algebra $\mathbb{C}[G]^{\otimes \mathbb{G}}$.

\vspace{2mm}
\par Let $Y$ be an affine Poisson variety and $X\subset Y$ be a sub-variety whose closed embedding is defined by a surjective morphism $p: \mathbb{C}[Y] \rightarrow \mathbb{C}[X]$. The Poisson structure $P$ of $Y$ induces a Poisson structure on $X$ if and only if the ideal $\mathcal{I}=\ker(p)$ is Poisson ideal, \emph{i.e}. if $P\left( \mathcal{I}\otimes \mathbb{C}[Y]\right) \subset \mathcal{I}$. This condition is equivalent to the fact that for any $x\in X$ one has $P_x\left( \mathcal{I} \otimes \mathbb{C}[Y] \right) =0$ or equivalently to the fact that for any $x\in X$ one has $P_x \in \Lambda^2 T_x X \subset \Lambda^2 T_x Y$. Specialising the preceding discussion to the embedding $\mathcal{X}_G(\mathbb{T}, \mathbb{P}^{\mathbb{T}})=\mathcal{R}_G(\mathbb{T}, \mathbb{P}^{\mathbb{T}}) \subset G^{\mathbb{G}}$, to prove that the Poisson bracket $P^{\mathbb{T}}$ induces a Poisson bracket on $\mathbb{C}[\mathcal{X}_G(\mathbb{T}, \mathbb{P}^{\mathbb{T}})]$, we need to show that for any $\rho \in \mathcal{X}_G(\mathbb{T}, \mathbb{P}^{\mathbb{T}})$, the bi-vector $(P^{\mathbb{T}})_{\rho} \in \Lambda^2 T_{\rho} G^{\mathbb{G}} = \Lambda^2 \mathrm{C}^1(\mathbb{T}, \mathbb{P}^{\mathbb{T}}; \rho)$ lies in $\Lambda^2 \mathrm{Z}^1(\mathbb{T}, \mathbb{P}^{\mathbb{T}} ; \rho)= \Lambda^2 T_{\rho} \mathcal{R}_G(\mathbb{T}, \mathbb{P}^{\mathbb{T}})$ by Lemma \ref{lemma_homo_discrete}. 

\vspace{2mm}
\par
Let us state a more explicit description of $(P^{\mathbb{T}})_{\rho}$. Consider the embedding $\Lambda^2 \mathrm{C}^1(\mathbb{T}, \mathbb{P}^{\mathbb{T}}; \rho) \subset \mathrm{C}^1(\mathbb{T}, \mathbb{P}^{\mathbb{T}}; \rho)^{\otimes 2} = \oplus_{\delta_1, \delta_2 \in \mathbb{G}} T_{\rho(\delta_1)}G\otimes T_{\rho(\delta_2)}G$ and denote by $w(\delta_1, \delta_2) \subset T_{\rho(\delta_1)}G \otimes T_{\rho(\delta_2)}G$ the projection of $(P^{\mathbb{T}})_{\rho}$ in the corresponding summand,  such that $(P^{\mathbb{T}})_{\rho}= \oplus_{(\delta_1, \delta_2) \in \mathbb{G}^2} w(\delta_1, \delta_2)$. Then $(P^{\mathbb{T}})_{\rho}$ is characterized by the following equalities, where we denote by $\overline{r}$ the skew-symmetric part of $r$ and where $\sigma(x\otimes y)=y\otimes x$: 

\begin{align}
\label{eq_a}&w(\delta_1, \delta_2) = -\sigma (w(\delta_2, \delta_1)) &\mbox{, for all }\delta_1, \delta_2 \in \mathbb{G}; \\
\label{eq_b}&w(\delta_1^{-1}, \delta_2) = -(\rho(\delta_1)^{-1}\otimes 1)w(\delta_1, \delta_2)(\rho(\delta_1)^{-1}\otimes 1)&\mbox{, for all }\delta_1, \delta_2 \in \mathbb{G}; \\
\label{eq_c}&w(\delta, \delta) = \overline{r}_{s(\delta)} (\rho(\delta)\otimes \rho(\delta)) + (\rho(\delta)\otimes \rho(\delta))\overline{r}_{t(\delta)} & \mbox{, for all }\delta\in \mathbb{G}; \\
\label{eq_d}&w(\beta_{i}, \beta_{i+1}^{-1}) = (\rho(\beta_{i})\otimes \rho(\beta_{i+1})^{-1})r_{v_{i+1}}  &\mbox{,for all } i\in \mathbb{Z}/3\mathbb{Z}.
\end{align}

\begin{lemma}\label{lemma_poisson_triangle}
For any $\rho \in \mathcal{R}_G(\mathbb{T}, \mathbb{P}^{\mathbb{T}})$, one has $(P^{\mathbb{T}})_{\rho} \in \Lambda^2 \mathrm{Z}^1 (\mathbb{T}, \mathbb{P}^{\mathbb{T}} ; \rho)$.
\end{lemma}

\begin{proof} Since $(P^{\mathbb{T}})_{\rho}$ is skew-symmetric, it is sufficient to prove that $(d^1 \otimes \id) (P^{\mathbb{T}})_{\rho} = 0$. Decompose $d^1=D_{\rho}\mathcal{R} : \oplus_{\delta \in \mathbb{G}}T_{\rho(\delta)}G \rightarrow \mathfrak{g}^{\oplus \mathbb{RL}}$ as $d^1= (d^1_{R_1}, \ldots , d^1_{R_m})$ where $R_i\in \mathbb{RL}$. For $\delta \in \mathbb{G}$, we first consider the trivial relation $R_{\delta}:=\delta\star \delta^{-1}$ for which $d^1_{R_{\delta}}\left(\oplus_{\eta \in \mathbb{G}} X_{\eta} \right) = X_{\delta}\rho(\delta^{-1}) + \rho(\delta)X_{\delta^{-1}}$. One has 
$$ (d^1_{R_{\delta}} \otimes \id) \left(P^{\mathbb{T}}\right)_{\rho} = \oplus_{\eta \in \mathbb{G}} \left( w(\delta, \eta)(\rho(\delta^{-1})\otimes 1) + (\rho(\delta)\otimes 1)w(\delta^{-1}, \eta) \right) =0. $$
 Next consider the only non-trivial relation $R= \beta_1 \star \beta_2 \star \beta_3 \in \mathbb{RL}$. By definition, one has
$$d^1_R \left(\oplus_{\delta \in \mathbb{G}} X_{\delta} \right) = X_{\beta_1}\rho(\beta_2\beta_3) + \rho(\beta_1)X_{\beta_2}\rho(\beta_3) + \rho(\beta_1 \beta_2)X_{\beta_3}. $$
\par For $\eta\in \mathbb{G}$, denote by $\pi_{\eta} : \oplus_{\delta \in \mathbb{G}} T_{\rho(\delta)}G \rightarrow T_{\rho(\eta)}G$ the associated projection. To prove that $(d^1_R \otimes \id)\left(P^{\mathbb{T}} \right)_{\rho} \in \oplus_{\delta \in \mathbb{G}} \mathfrak{g} \otimes T_{\rho(\delta)}G$ vanishes, we need to show that for each generator $\eta \in \mathbb{G}$,  one has $(\id \otimes \pi_{\eta})(d^1_R \otimes \id) \left(P^{\mathbb{T}}\right)_{\rho} = 0$. 
We detail the computation for the generator $\beta_3$; the computations for the other generators are similar and left to the reader. 
Write $S:= (\id \otimes \pi_{\beta_3})(d^1_R \otimes \id) \left(P^{\mathbb{T}}\right)_{\rho} $ and let us prove that $S=0$. First using the above expression for $d^1_R$, we find:
\begin{equation}\label{eq_machinbidule}
S = w(\beta_1, \beta_3)(\rho(\beta_2\beta_3) \otimes 1) + (\rho(\beta_1)\otimes 1) w(\beta_2, \beta_3)(\rho(\beta_3)\otimes 1) + (\rho(\beta_1\beta_2)\otimes 1)w(\beta_3, \beta_3).
\end{equation}
Let us compute each summand in the right-hand-side of Equation \ref{eq_machinbidule}.
First using Equations \eqref{eq_a}, \eqref{eq_b} and \eqref{eq_d}, we find
$$ w(\beta_1, \beta_3) = -(\rho(\beta_1)\otimes 1) w(\beta_1^{-1}, \beta_3) (\rho(\beta_1)\otimes 1) = +(\rho(\beta_1)\otimes 1) \sigma(w(\beta_3, \beta_1^{-1})) (\rho(\beta_1)\otimes 1) = (1\otimes \rho(\beta_3)) \sigma(v_{v_1}) (\rho(\beta_1)\otimes 1).$$
Thus 
\begin{equation}\label{term1}
w(\beta_1, \beta_3) (\rho(\beta_1)^{-1} \otimes 1) = (1\otimes \rho(\beta_3)) \sigma(r_{v_1}).
\end{equation}
Using Equations \eqref{eq_a}, \eqref{eq_b} and \eqref{eq_d} again, we find
$$ w(\beta_2, \beta_3) = -\sigma(w(\beta_3,\beta_2))= - (1\otimes \rho(\beta_3)) w(\beta_2, \beta_3^{-1}) (1\otimes \rho(\beta_3)) =-(\rho(\beta_2)\otimes 1) r_{v_3} (1\otimes \rho(\beta_3)).$$
Thus 
\begin{equation}\label{term2}
(\rho(\beta_1)\otimes ) w(\beta_2, \beta_3) (\rho(\beta_3) \otimes 1) = -(\rho(\beta_1\beta_2) \otimes 1) r_{v_3}(\rho(\beta_3) \otimes \rho(\beta_3)).
\end{equation}
Using  Equation  \eqref{eq_c} and the fact that $\rho(\beta_1\beta_2\beta_3)=1$, we find
\begin{equation}\label{term3}
(\rho(\beta_1\beta_2)\otimes 1) w(\beta_3, \beta_3) = (\rho(\beta_3)^{-1}\otimes 1) \overline{r}_{v_3}(\rho(\beta_3) \otimes \rho(\beta_3)) +(1\otimes \rho(\beta_3))\overline{r}_{v_1}.
\end{equation}

Using Equations \eqref{term1}, \eqref{term2}, \eqref{term3}, Equation \eqref{eq_machinbidule} simplifies to
$$ S= (1\otimes \rho(\beta_3)) (\sigma(r_{v_1}) +\overline{r}_{v_1}) + (\rho(\beta_3)^{-1}\otimes 1) (\overline{r}_{v_3} - r_{v_3})(\rho(\beta_3) \otimes \rho(\beta_3)).$$
Remember that $r_{v_i} = \tau + \overline{r}_{v_i}$, where  the symmetric part $\tau$ is the dual of the invariant pairing, so does not depend on $i$ and $\overline{r}_{v_i}$ is the skew-symmetric part. We thus have $\sigma(r_{v_1})+\overline{r}_{v_1} = \tau$ and $\overline{r}_{v_3} - r_{v_3}= -\tau$ so
$$S= (1\otimes \rho(\beta_3)) \tau - (\rho(\beta_3)^{-1}\otimes 1) \tau (\rho(\beta_3)\otimes \rho(\beta_3)) = 0, $$
where we used the fact that $\tau$ is Ad-invariant. 
This concludes the proof.
\end{proof}

Lemma \ref{lemma_poisson_triangle} shows that $\mathcal{X}_G(\mathbb{T}, \mathbb{P}^{\mathbb{T}}) \subset G^{\mathbb{G}}$ is a sub Poisson variety, so we can state the

\begin{definition} We denote by $\{ \cdot , \cdot \}^{\mathfrak{o}}$ the Poisson bracket on  $\mathbb{C}[\mathcal{X}_G(\mathbb{T}, \mathbb{P}^{\mathbb{T}})]$ induced by $P^{\mathbb{T}}$.
\end{definition}
 
 \begin{remark}\label{remark_PoissonTriangle}
 Suppose that $G\subset \GL_N(\mathbb{C})$ and, using notations similar to Remark \ref{remark_PoissonBigon}, for $\delta \in \mathbb{G}$, denote by $M(\delta)$ the $N\times N$ matrix with coefficient in $\mathcal{X}_G(\mathbb{T}, \mathbb{P}^{\mathbb{T}})$ whose $(k,l)$ entry is the function sending a representation $\rho$ to the $(k,l)$ entry of $\rho(\delta)$. Then the Poisson bracket $\{\cdot , \cdot \}^{\mathfrak{o}}$ is described by the down-to-earth formulas, where $i\in \mathbb{Z}/3\mathbb{Z}$:
 \begin{align*}
& \left\{ M(\beta_i) \otimes M(\beta_i) \right\}^{\mathfrak{o}} =  \overline{r}^{\mathfrak{o}(s(\beta_i))} (M(\beta_i) \otimes M(\beta_i)) + (M(\beta_i)\otimes M(\beta_i))\overline{r}^{\mathfrak{o}(t(\beta_i))}, \\
&  \{ M(\beta_i) \otimes M(\beta_{i+1})^{-1} \}^{\mathfrak{o}} =  (M(\beta_i) \otimes M(\beta_{i+1}^{-1})) r_{v_{i+1}}.
 \end{align*}
 
 \end{remark}
 
 \subsubsection{The general case} Consider a punctured surface $\mathbf{\Sigma}$ equipped with topological triangulation $\Delta$ and an orientation $\mathfrak{o}_{\Delta}$ of the edges of $\Delta$.  In Example \ref{example_finite_presentations} we defined a finite presentation $\mathbb{P}^{\Delta}$ of the fundamental groupoid of $\Sigma$ made by gluing the canonical presentation of the triangle in each face of the triangulation. In particular the presentation has exactly one puncture $v_e$ in each edge $e\in \mathcal{E}(\Delta)$ of the triangulation. Consider the punctured surface $\mathbf{\Sigma}^{\Delta}= \bigsqcup_{\mathbb{T} \in F(\Delta)} \mathbb{T}$ which is the disjoint union of the triangles of the faces of the triangulation. Then $\mathbf{\Sigma}$ is obtained from $\mathbf{\Sigma}^{\Delta}$ by gluing the pair of faces corresponding to the edges of the triangulation. Note that the choice of an orientation $\mathfrak{o}_{\Delta}$ of each edge of the triangulation induces an orientation of the boundary arcs of $\mathbf{\Sigma}^{\Delta}$. Hence the algebra $\mathbb{C}[\mathcal{X}_G(\mathbf{\Sigma}^{\Delta})] = \otimes_{\mathbb{T} \in F(\Delta)} \mathbb{C}[\mathcal{X}_G(\mathbb{T}, \mathbb{P}^{\mathbb{T}})]$ inherits a Poisson bracket from this choice of orientation. 
\vspace{2mm}\par

By Proposition \ref{gluing_formula}, one has the exact sequence: 
\begin{multline}\label{suite_exacte}
 0 \rightarrow \mathbb{C}[\mathcal{X}_G(\mathbf{\Sigma}, \mathbb{P}^{\Delta})] \xrightarrow{i^{\Delta}} \otimes_{\mathbb{T} \in F(\Delta)} \mathbb{C}[\mathcal{X}_G(\mathbb{T}, \mathbb{P}^{\mathbb{T}})] 
 \xrightarrow{\Delta^L - \sigma \circ \Delta^R} \left( \otimes_{e\in \mathring{\mathcal{E}}(\Delta)} \mathbb{C}[G] \right) \otimes \left(  \otimes_{\mathbb{T} \in F(\Delta)} \mathbb{C}[\mathcal{X}_G(\mathbb{T}, \mathbb{P}^{\mathbb{T}})] \right). 
 \end{multline}
 \par Each inner edge $e\in \mathring{\mathcal{E}}(\Delta)$ corresponds to two edges  in the disjoint union $ \bigsqcup_{\mathbb{T} \in F(\Delta)} \mathbb{T}$, that is to two boundary arcs $e', e''$ of $\mathbf{\Sigma}^{\Delta}$. The co-modules maps $\Delta^L$ and $\Delta^R$ depend on the choice of which of these two boundary arcs we consider as being on the left or on the right, that is weather we consider the gluing $\mathbf{\Sigma}^{\Delta}_{|e' \# e''}$ or $\mathbf{\Sigma}^{\Delta}_{|e'' \#e'}$. Note also that the two signs $\mathfrak{o}_{\Delta}(e')$ and $\mathfrak{o}_{\Delta}(e'')$ are distinct. We will follow the convention that we choose the gluing $\mathbf{\Sigma}^{\Delta}_{|e' \# e''}$ for which $\mathfrak{o}_{\Delta}(e')= + $ and $\mathfrak{o}_{\Delta}(e'')= -$. Moreover, we equip the algebra  $\left( \otimes_{e\in \mathring{\mathcal{E}}(\Delta)} \mathbb{C}[G] \right)$ with the Poisson structure obtained by choosing the bracket $\left\{ \cdot, \cdot\right\}^{-, +}$ in each factor.
 
 \begin{lemma}\label{lemma_poisson_gluing}
 The comodules maps $\Delta^L$ and $\Delta^R$ in the exact sequence \eqref{suite_exacte} are Poisson morphisms.
 \end{lemma}
 \begin{proof} The proof is a straightforward consequence of Lemma \ref{lemma_poissonlie} and of the signs convention.
 \end{proof}
 
 \par It follows from Lemma \ref{lemma_poisson_gluing} and the exact sequence \eqref{suite_exacte}, that the algebra $ \mathbb{C}[\mathcal{X}_G(\mathbf{\Sigma}, \mathbb{P}^{\Delta})]$ is a Poisson sub-algebra of $\otimes_{\mathbb{T} \in F(\Delta)} \mathbb{C}[\mathcal{X}_G(\mathbb{T}, \mathbb{P}^{\mathbb{T}})] $, hence inherits a Poisson bracket. 
 
 \begin{definition} We denote by $\left\{ \cdot, \cdot \right\}^{\Delta, \mathfrak{o}_{\Delta}}$ the Poisson bracket on $ \mathbb{C}[\mathcal{X}_G(\mathbf{\Sigma})]$ induced by the isomorphism $ \mathbb{C}[\mathcal{X}_G(\mathbf{\Sigma})]\cong  \mathbb{C}[\mathcal{X}_G(\mathbf{\Sigma}, \mathbb{P}^{\Delta})]$ of Proposition \ref{proposition_discrete_model}.
 \end{definition}
  Note that at this stage, the Poisson structure seems to depend on both the choice of a triangulation and on the choice of an orientation of the edges.

\subsection{The generalized Goldman formula}

\par We first re-write Theorem \ref{theorem3} in a concise form. Suppose that $\mathbf{\Sigma}$ is either the bigon $\mathbb{B}$ equipped with an orientation $\mathfrak{o}$ of its boundary edges,  or that $\mathbf{\Sigma}$ is equipped with a topological triangulation $\Delta$ and an orientation $\mathfrak{o}_{\Delta}$ of the edges of $\Delta$. Note that in the latter case, the orientation of the edges induces an orientation  $\mathfrak{o}$ of the boundary arcs of $\mathbf{\Sigma}$.
Let $f_{\mathcal{C}_1}$ and $h_{\mathcal{C}_2}$ be two curve functions and $c_1, c_2$ be two geometric representatives of $\mathcal{C}_1$ and $\mathcal{C}_2$ respectively, in transverse position. Fix $\rho \in \mathcal{R}_G(\mathbf{\Sigma})$.

\begin{theorem}[Generalized Goldman formula]\label{goldman_formula}
 The following equality holds
$$
 \{ f_{\mathcal{C}_1}, h _{\mathcal{C}_2} \} ([\rho]) = [\mathcal{C}_1, X_{f, \mathcal{C}_1}] \bigcap{}^{\mathfrak{o}} [\mathcal{C}_2, X_{h, \mathcal{C}_2}].
$$

\end{theorem}

\par Theorem \ref{goldman_formula} is just a reformulation of Theorem \ref{theorem3}. Since the intersection form only depends on the choice of orientation of the boundary arcs and the holonomy functions generate the algebra of regular functions,Theorem \ref{goldman_formula} implies that the Poisson structure is independent of the triangulation and of the orientation of its inner edges.

\begin{proof} We first consider the case where $\mathbf{\Sigma}$ is the bigon $\mathbb{B}$ and $\mathcal{C}_1=\mathcal{C}_2=:\mathcal{C}$ is the curve represented by the path $\alpha$. Consider two geometric representative $c_1$ and $c_2$ which do not intersect and such that $c_1$ lies on the top of $c_2$. Write $S(a)= (v_1, v_2)$ and $S(b)=(v'_1, v'_2)$ where $v_i = c_i \cap a$ and $v'_i= c_i \cap b$. One has

\begin{align*}
&\left\{ f_{\mathcal{C}}, h_{\mathcal{C}} \right\}^{\varepsilon_1, \varepsilon_2} ([\rho]) = \left< D_{\rho(\alpha)}f \otimes D_{\rho(\alpha)}g, \overline{r}^{\varepsilon_1}(\rho(\alpha)\otimes \rho(\alpha)) +  (\rho(\alpha)\otimes \rho(\alpha))\overline{r}^{\varepsilon_2} \right> \\
& = \left< D_{\rho(\alpha)}f \otimes D_{\rho(\alpha)}g, r^{\mathfrak{o}(v_1,v_2)}(\rho(\alpha)\otimes \rho(\alpha)) -  (\rho(\alpha)\otimes \rho(\alpha))r^{\mathfrak{o}(v'_1,v'_2)} \right>\\
&=  \left(  (X_{\mathcal{C}, f}\otimes X_{\mathcal{C}, h} )(\rho(\alpha)^{-1} \otimes \rho(\alpha)^{-1}), r^{\mathfrak{o}(v_1,v_2)} \right)
 -\left( (\rho(\alpha)^{-1} \otimes \rho(\alpha)^{-1}) (X_{\mathcal{C}, f}\otimes X_{\mathcal{C}, h} ), r^{\mathfrak{o}(v'_1, v'_2)} \right) \\
&= \varepsilon(v_1, v_2) \left( X_{\mathcal{C}, f}(v_1) \otimes X_{\mathcal{C}, h}(v_2), r^{\mathfrak{o}(v_1,v_2)} \right) + \varepsilon(v'_1, v'_2) \left( X_{\mathcal{C}, f}(v'_1) \otimes X_{\mathcal{C}, h}(v'_2), r^{\mathfrak{o}(v'_1,v'_2)} \right) \\
 &= [\mathcal{C}, X_{\mathcal{C}, f} ] \bigcap{}^{\mathfrak{o}} [\mathcal{C}, X_{\mathcal{C}, h} ] 
 \end{align*}
\par To pass from the first to the second line, we used the facts that the symmetric parts of $r^{\pm}$ are equal to the $G$-invariant bi-vector $\tau$, and that $\overline{r}^{\epsilon}= -\overline{r}^{-\epsilon} $.  If $\mathcal{C}^{-1}$ is the curve represented by the path $\alpha^{-1}$, we have $f_{\mathcal{C}} = S(f)_{\mathcal{C}^{-1}}$ hence we obtain similar equalities for pairs of curves $(\mathcal{C}_1, \mathcal{C}_2)= (\mathcal{C}_1^{\pm}, \mathcal{C}_2^{\pm})$ and the proof for the bigon is completed.

\vspace{2mm}
\par Next consider the case where $\mathbf{\Sigma}$ is the triangle $\mathbb{T}$. For $\delta \in \{\beta_1^{\pm 1}, \beta_2^{\pm 1}, \beta_3^{\pm 1} \}$, denote by $\mathcal{C}_{\delta}$ the associated curve. The equality $\{ f_{\mathcal{C}_{\delta}} , h_{\mathcal{C}_{\delta}} \} ([\rho]) = [\mathcal{C}_{\delta}, X_{\mathcal{C}_{\delta}, f} ] \bigcap{}^{\mathfrak{o}} [\mathcal{C}_{\delta}, X_{\mathcal{C}_{\delta}, h} ] $ is proved by the same computation than in the case of the bigon. Next consider the case where $(\mathcal{C}_1, \mathcal{C}_2)=(\mathcal{C}_{\beta_i^{-1}}, \mathcal{C}_{\beta_{i+1}})$. Choose some geometric representatives $c_1$ and $c_2$ of $\mathcal{C}_{\beta_i^{-1}}$ and $\mathcal{C}_{\beta_{i+1}}$ respectively which do not intersect and denote by $(w_1, w_2)$ the intersection points $w_i := c_i \cap b$. We compute: 
\begin{align*}
& \left\{ f_{\mathcal{C}_{\beta_i^{-1}}}, h_{\mathcal{C}_{\beta_{i+1}}} \right\} ([\rho]) = \left< D_{\rho} f_{\mathcal{C}_{\beta_i^{-1}}} \otimes D_{\rho}h_{\mathcal{C}_{\beta_{i+1}}} , w(\beta_i^{-1}, \beta_{i+1}) \right>  \\
& = \left< D_{\rho(\beta_i^{-1})}f \otimes D_{\rho(\beta_{i+1})^{-1}}h , (\rho(\beta_i^{-1})\otimes \rho(\beta_{i+1}))r^{\mathfrak{o}(w_1,w_2)} \right> \\
& = \varepsilon(w_1, w_2) \left( X_{\mathcal{C}_{\beta_i^{-1}}, f}(w_1) \otimes X_{\mathcal{C}_{\beta_{i+1}}, h}(w_2) , r^{\mathfrak{o}(w_1,w_2)} \right)
=  [\mathcal{C}_{\beta_i^{-1}}, X_{\mathcal{C}_{\beta_i^{-1}}, f} ] \bigcap{}^{\mathfrak{o}} [\mathcal{C}_{\beta_{i+1}^{-1}}, X_{\mathcal{C}_{\beta_{i+1}^{-1}}, h} ] 
\end{align*}
\par The other cases are obtained from the above computation changing the orientations using $f_{\mathcal{C}}= S(f)_{\mathcal{C}^{-1}}$. Hence the proof in the case of the triangle is completed.
\vspace{2mm}
\par Eventually suppose that $\mathbf{\Sigma}$ is a punctured surface with topological triangulation $\Delta$, fix $\rho \in \mathcal{R}_G(\mathbf{\Sigma})$  and consider the Poisson embedding $i^{\Delta} : \mathbb{C}[\mathcal{X}_G(\mathbf{\Sigma})] \hookrightarrow \otimes_{\mathbb{T} \in F(\Delta)} \mathbb{C}[\mathcal{X}_G(\mathbb{T}, \mathbb{P}^{\mathbb{T}})]$. The Poisson bracket $\{ \cdot, \cdot \}^{\Delta, \mathfrak{o}_{\Delta}}$ is, by definition, the restriction of a Poisson bracket $\hat{P}=\otimes_{\mathbb{T}}P^{\mathbb{T}}$ on $\otimes_{\mathbb{T} \in F(\Delta)} \mathbb{C}[\mathcal{X}_G(\mathbb{T}, \mathbb{P}^{\mathbb{T}})]$, hence the bi-derivation $\chi_{\rho}\circ \{ \cdot, \cdot \}^{\Delta, \mathfrak{o}_{\Delta}}$ is the restriction of $\hat{P}_{\rho}$. The skew-symmetric bilinear form $\Theta$ on $\mathbb{C}[\mathcal{X}_G(\mathbf{\Sigma})]$ defined by $\Theta(f\otimes h)= \Lambda^*(D_{[\rho]}f) \bigcap{}^{\mathfrak{o}} \Lambda^*(D_{[\rho]}h)$ is, by Lemma \ref{lemma_gluing_intersection}, the restriction of a bilinear form $\hat{\Theta}$ on $\otimes_{\mathbb{T} \in F(\Delta)} \mathbb{C}[\mathcal{X}_G(\mathbb{T}, \mathbb{P}^{\mathbb{T}})]$. By the above proof for the triangle, the two forms $\hat{\Theta}$ and $\hat{P}_{\rho}$ are equal, hence their restrictions to $\mathbb{C}[\mathcal{X}_G(\mathbf{\Sigma})] $ agree. This concludes the proof.
\end{proof}

\section{The case $G=\mathbb{C}^*$}\label{sec_abelian}

\par When $G=\mathbb{C}^*$, the character varieties have a simple description and are closely related to the quantum Teichm\"uller spaces (see \cite{KojuQuesneyQNonAb}).  Let $c_1, c_2$ be two geometric curves in $\Sigma$ in transverse position  and denote by $\sigma_1, \sigma_2$ the cycles in  $\mathrm{Z}_1(\Sigma, \mathcal{A}; \mathbb{Z})$ represented by $c_1$ and $c_2$.
\begin{definition}
 We define the skew-symmetric pairing $\left( \cdot, \cdot \right) : \mathrm{H}_1\left( \Sigma, \mathcal{A}; \mathbb{Z} \right)^{\otimes 2} \rightarrow \frac{1}{2}\mathbb{Z}$ by the formula
$$ \left( [\sigma_1], [\sigma_2] \right) := \sum_a \sum_{(v_1, v_2) \in S(a) } \frac{1}{2} \varepsilon(v_1, v_2) + \sum_{v \in c_1 \cap c_2} \varepsilon(v). $$
\end{definition}

The classes $[\sigma_i]$ associated to such cycles $\sigma_i$ span the module  $\mathrm{Z}_1(\Sigma, \mathcal{A}; \mathbb{Z})$ and an argument similar to the proof of Lemma \ref{lemma_intersection} shows that the pairing $ \left( [\sigma_1], [\sigma_2] \right)$ only depends on the homology classes $[\sigma_i]$, hence the pairing is well-defined in homology. Note that when $\partial \Sigma = \emptyset$, this pairing is the classical intersection pairing. Define the Poisson bracket $\left\{ \cdot, \cdot \right\}$ on the group algebra $\mathbb{C}[\mathrm{H}_1\left( \Sigma, \mathcal{A}; \mathbb{Z}\right) ]$ by the formula 
$$ \{ [\sigma_1], [\sigma_2] \} := ([\sigma_1], [\sigma_2]) [\sigma_1 + \sigma_2]. $$
\par Equip the Lie algebra $\mathbb{C}$ of $\mathbb{C}^*$ with the invariant bi-vector and the $r$-matrices $\tau= r^+ = r^- = \frac{1}{2} 1\otimes 1 \in \mathbb{C}^{\otimes 2}$. The algebra of regular functions $\mathbb{C}[\mathbb{C}^*]= \mathbb{C}[X^{\pm 1}]$ is generated by the elements $X$ and $X^{-1}$. Let $\mathcal{C}$ be a curve in $\Sigma_{\mathcal{P}}$,   $c$  a geometric representative of $\mathcal{C}$ and $\sigma_c \in \mathrm{Z}_1(\Sigma, \mathcal{A} ; \mathbb{Z})$ the induced cycle. The homology class $[\sigma_c] \in \mathrm{H}_1(\Sigma, \mathcal{A}; \mathbb{Z})$ does not depend on the choice of the geometric representative $c$ and will be denoted by $[\mathcal{C}]$. 

\begin{proposition}\label{prop_abelian}
There exists a Poisson isomorphism of algebra $\Psi : \mathbb{C}[\mathcal{X}_{\mathbb{C}^*}(\mathbf{\Sigma}) ]\xrightarrow{\cong} \mathbb{C}[\mathrm{H}_1(\Sigma, \mathcal{A}; \mathbb{Z})]$ characterized by the formula $\Psi( X_{\mathcal{C}}) = [\mathcal{C}]$.
\end{proposition}

\begin{proof} If $\sigma\in \mathrm{C}_1(\Sigma; \mathbb{Z})$ is a  singular $1$-chain, denote by $\alpha_{\sigma}\in \Pi_1(\Sigma)$ its homotopy class. Define a morphism $\phi_1 : \mathbb{C}[\mathrm{C}_1(\Sigma ; \mathbb{Z})] \rightarrow \mathbb{C}[\Map(\Pi_1(\Sigma), \mathbb{C}^*)]$ by the formula $\phi_1(\sigma):= X_{\alpha_{\sigma}}$. By definition of the boundary arcs, we have the inclusion $\phi_1\left( \mathrm{C}_1(\mathcal{A}; \mathbb{Z}) \right) \subset \mathcal{I}_{\epsilon}$, hence $\phi_1$ induces a morphism $\phi_2: \mathbb{C}[\mathrm{C}_1(\Sigma, \mathcal{A}; \mathbb{Z})] \rightarrow \mathbb{C}[\mathcal{R}_{\mathbb{C}^*}(\mathbf{\Sigma})]$. Since $\mathbb{C}^*$ is abelian, the action of $\mathbb{C}^*$ on itself by conjugacy is trivial, and we have the inclusion $\phi_2\left( \mathrm{Z}_1(\Sigma, \mathcal{A}; \mathbb{Z} ) \right) \subset \mathbb{C}[\mathcal{X}_{\mathbb{C}^*}(\mathbf{\Sigma})]$. Denote by $\phi_3 : \mathbb{C}[ \mathrm{Z}_1(\Sigma, \mathcal{A}; \mathbb{Z} )] \rightarrow  \mathbb{C}[\mathcal{X}_{\mathbb{C}^*}(\mathbf{\Sigma})]$ the induced morphism. Since we work in dimension $2$, the space $\mathrm{B}_1(\Sigma, \mathcal{A}; \mathbb{Z})$ is spanned by elements of the form $\partial S$ where $S\subset \Sigma_{\mathcal{P}}$ is an embedded surface. Given such a surface $S$, decompose $\partial S = c_1 \ldots c_n$ into geometric arcs. Since $\alpha:= \alpha_{c_1} \ldots \alpha_{c_n}$ is a trivial path, then $[X_{\alpha}]=0 \in  \mathbb{C}[\mathcal{X}_{\mathbb{C}^*}(\mathbf{\Sigma})]$. Moreover since $\Delta^{(n-1)}(X)=X^{\otimes n}$, one has the equalities $\phi_3(\partial S) = \sum_i X_{\alpha_{c_i}} = X_{\alpha} = 0$. Thus $\phi_3$ induces a morphism $\phi : \mathbb{C}[\mathrm{H}_1(\Sigma, \mathcal{A}; \mathbb{Z} )] \rightarrow  \mathbb{C}[\mathcal{X}_{\mathbb{C}^*}(\mathbf{\Sigma})]$. Note that $\phi( [\mathcal{C}]) = X_{\mathcal{C}}$.
\vspace{2mm}
\par Next define a morphism $\psi_1 : \mathbb{C}[\Map(\Pi_1(\Sigma), \mathbb{C}^*)] \rightarrow  \mathbb{C}[\mathrm{H}_1(\Sigma, \mathcal{A}; \mathbb{Z} )]$ by the formula $\psi_1( X^{\pm 1}_{\alpha}) := \pm [\sigma_{\alpha}]$, where $[\sigma_{\alpha}]$ is the homology class of the singular $1$-chain associated to an arbitrary geometric representative of $\alpha$. It follows from the definitions that we have  $\psi_1 \left( \mathcal{I}_{\varepsilon} + \mathcal{I}_{\Delta} \right) =0$, hence $\psi_1$ induces a morphism $\Psi :  \mathbb{C}[\mathcal{X}_{\mathbb{C}^*}(\mathbf{\Sigma}) ]\rightarrow \mathbb{C}[\mathrm{H}_1(\Sigma, \mathcal{A}; \mathbb{Z})]$. Since $\Psi(X_{\mathcal{C}}) = [\mathcal{C}]$, the morphisms $\Psi$ and $\phi$ are inverse to each other, thus are isomorphisms.  The fact that $\Psi $ preserves the Poisson brackets results from Proposition \ref{goldman_formula}.

\end{proof}

\section{Alekseev-Malkin's fusion operation}\label{sec_fusion}

\begin{definition} Let $\mathbf{\Sigma}=(\Sigma, \mathcal{A})$ be a marked surface and $a,b\in \mathcal{A}$ two boundary arcs. Recall that the triangle $\mathbb{T}$ is a disc with three boundary arcs, say $i, j, k$. The \textit{fusion} of  $\mathbf{\Sigma}$ along $a,b$ is the marked surface 
   $$ \mathbf{\Sigma}_{a \circledast b}:= \left( \mathbf{\Sigma}\bigsqcup \mathbb{T} \right) _{a\#i, b\#j}.$$
obtained by gluing a triangle to $\mathbf{\Sigma}$.
\end{definition}

The \stated character varieties $\mathcal{X}_G(\mathbf{\Sigma})$ and $\mathcal{X}_G(\mathbf{\Sigma}_{a\circledast b})$ are related as follows. 
 A $G$-\textit{Poisson affine variety} is a complex affine variety $X$ with an algebraic Poisson action $G\times X \to X$. 
  
  \begin{definition}\label{def_fusion}
  Let $G$ be an algebraic Poisson Lie group with classical $r$-matrix $r^+$.
  Let $X$ be a $G^2$-Poisson affine variety and denote by $\Delta_{G\times G} : \mathbb{C}[X]\to \mathbb{C}[G]^{\otimes 2}\otimes \mathbb{C}[X]$ it comodule map. Wite $\Delta^1:= (\id\otimes \epsilon\times \id)\circ \Delta_{G\times G}$ and $\Delta^2:= (\epsilon \otimes \id \otimes \id)$.
  The \textit{fusion} of $X$ is the $G$-Poisson affine variety $X^{\circledast}$ defined by:
  \begin{enumerate}
  \item As a $\mathbb{C}$-algebra, $\mathbb{C}[X^{\circledast}]=\mathbb{C}[X]$.
  \item  For $x\in \mathbb{C}[X]$ and $i=1,2$, write $\Delta^i (x) = \sum x_{(i)}'\otimes x_{(i)}''$. The Poisson bracket is defined by
  $$\{ x, y \}^{\circledast} := \{x,y\} + \sum r^+ (y'_{(2)} \otimes x'_{(1)}) x_{(1)}''y_{(2)}'' - \sum r^+ (x'_{(2)}\otimes y'_{(1)})y''_{(1)}x''_{(2)}.$$
   \item The $G$ action is given by the comodule map $\Delta_G:= (\mu_G \otimes \id) \circ \Delta_{G\times G} $.
  \end{enumerate}
  \end{definition}
  In the above formula, we have considered $r^+ \in \mathfrak{g}^{\otimes 2}$ as a derivation $r^+\in \mathrm{Der}(\mathbb{C}[G]^{\otimes 2}, \mathbb{C})$. 
  In the particular case where $X$ is smooth, consider $X$ as a smooth manifold and denote by $\pi_X $ the Poisson bivector field defining the Poisson structure (i.e. $\{f,g\}(x)=\left< D_xf \otimes D_xg, \pi_{X,x}\right>$). Let $r^- := \sigma(r^+)$, where $\sigma(x\otimes y)=y\otimes x$. Let $a_{G\times G}: \mathfrak{g}\otimes \mathfrak{g} \to \Gamma(X, T_X)$ the infinitesimal action induced by the action of $G^2$ on $X$. Then the fusion $X^{\circledast}$ is the manifold $X$ with the Poisson bivector field
  $$ \pi_{X^{\circledast}} = \pi_X + a_{G\times G} (r^- - r^+).$$
 This is using this formula that the concept of fusion was introduced in the work of Alekseev-Malkin \cite{AlekseevMalkin_PoissonCharVar}.  Fix $\mathfrak{o}$ such that $\mathfrak{o}(a)=\mathfrak{o}(b)=+$. The comodule maps $\Delta^L_a$ and $\Delta^L_b$ induce a structure of $G^2$-Poisson variety on $\mathcal{X}_G(\mathbf{\Sigma})$, where $G$ is equipped with the bracket $\{\cdot, \cdot\}_{-,+}$, through $\Delta_{G\times G}:= (\id \otimes \Delta_b^L)\circ \Delta_a^L$.

\begin{figure}[!h] 
\centerline{\includegraphics[width=8cm]{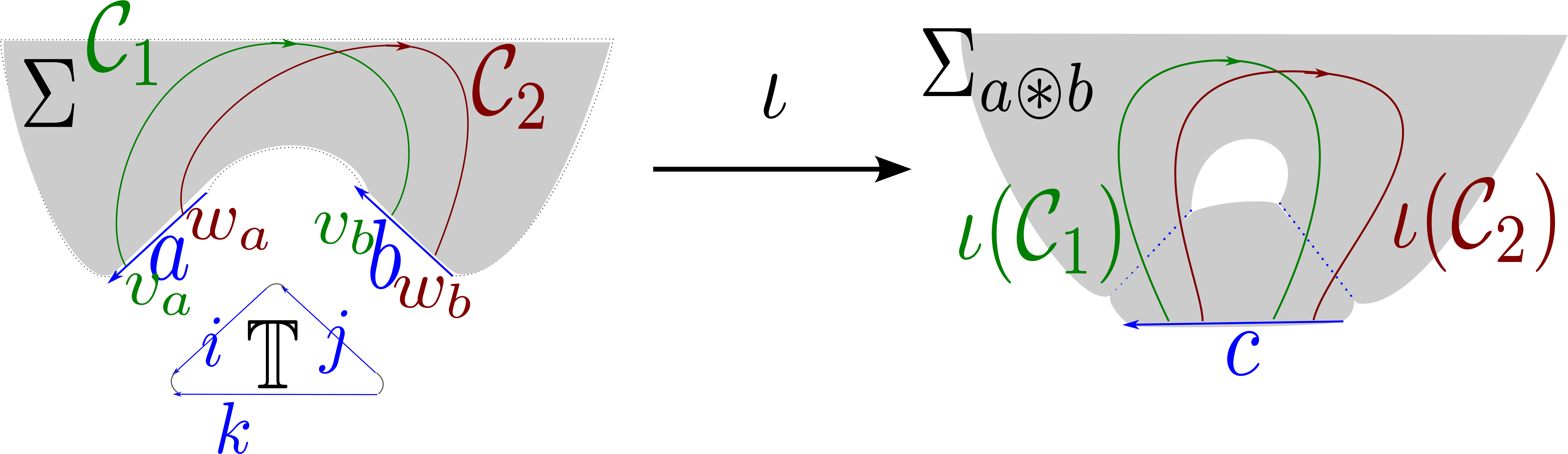} }
\caption{The marked surface $\mathbf{\Sigma}_{a\circledast b}$ is obtained from $\mathbf{\Sigma}$ by gluing a triangle $\mathbb{T}$. The figure illustrates how the embedding $\iota: \Sigma \hookrightarrow \Sigma_{a \circledast b}$ acts on curves.  } 
\label{fig_fusion} 
\end{figure} 
 
 \begin{theorem}\label{theorem_fusion}
 One has an isomorphism $\Psi: \mathcal{X}_G(\mathbf{\Sigma}_{a\circledast b})\cong \mathcal{X}_G(\mathbf{\Sigma})^{a\circledast b}$ of Poisson varieties.
 \end{theorem}

\begin{proof}
Let $c$ be the boundary arc of $\Sigma_{a\circledast b}$ which corresponds to the edge $k$ of $\mathbb{T}$.
Let $\iota: \Sigma \hookrightarrow \Sigma_{a\circledast b}$ be the embedding which is the identity outside disjoint collar  neighborhoods $N(a)$ and $N(b)$ of $a$ and $b$, sends both $a$ and $b$ to $c$ and which sends $N(a)$ and $N(b)$ to disjoint strips inside $\mathbb{T}$ as illustrated in Figure \ref{fig_fusion}. The convention is such that $v_a\in a$, $v_b \in b$ implies $\iota(v_b) <_{\mathfrak{o}_c} \iota(v_a)$. $\iota$ induces an equivalence $\iota_*: \Pi_1(\Sigma) \to \Pi_1(\Sigma_{a \circledast b})$ and thus an isomorphism of varieties $\Psi: \mathcal{X}_G(\mathbf{\Sigma}_{a\circledast b}) \xrightarrow{\cong} \mathcal{X}_G(\mathbf{\Sigma})=\mathcal{X}_G(\mathbf{\Sigma})^{a \circledast b}$ such that $\Psi^*$ sends a curve function $f_{\mathcal{C}}$ to $f_{\iota(\mathcal{C})}$. To prove that $\Psi: \mathcal{X}_G(\mathbf{\Sigma}_{a\circledast b}) \to \mathcal{X}_G(\mathbf{\Sigma})^{a \circledast b}$ is equivariant, it suffices to note the following equality
$$ \Delta_c^L = (\mu_G \otimes \id) (\id \otimes \Delta_b^L)\circ \Delta_a^L \circ( \Psi^* \otimes \Psi^*).$$
To prove that $\Psi$ is Poisson, let $x=f_{\mathcal{C}_1}$ and $y=h_{\mathcal{C}_2}$ be two curve functions in $\mathbb{C}[\mathcal{X}_G(\mathbf{\Sigma})]$, $\rho: \Pi_1(\Sigma)\to G$ and $\underline{\rho}:\Pi_1(\Sigma_{a\circledast b})\to G$ such that $\underline{\rho}\circ \iota_*=\rho$ (so $\Psi([\underline{\rho}])=[\rho]$). We need to prove that  
\begin{equation*}
\{f_{\iota(\mathcal{C}_1)}, h_{\iota(\mathcal{C}_2)} \}([\underline{\rho}]) = \{f_{\mathcal{C}_1}, h_{\mathcal{C}_2}\}^{\circledast}([\rho]).
\end{equation*}
The computation of both sides depends on the cardinality of the set $\mathcal{C}_i \cap a$ and $\mathcal{C}_i \cap b$. For instance, if $\mathcal{C}_1$ (or $\mathcal{C}_2$) does not intersect $a\cup b$, then by the generalized Goldman formula, both sides of the above equality are equal to $\{f_{\mathcal{C}_1}, h_{\mathcal{C}_2}\}(\rho)$. Let us suppose that $\mathcal{C}_1$ is oriented from an endpoint $v_a \in a$ to $v_b \in b$ and that $\mathcal{C}_2$ is oriented from an endpoint $w_a \in a$ to $w_b\in b$ and that $v_a>_a w_a$ and $v_b>_b w_b$. The other cases are handled similarly and left to the reader. On the one hand, the generalized Goldman formula tells us that 
\begin{multline*}
 \{f_{\iota(\mathcal{C}_1)}, h_{\iota(\mathcal{C}_2)} \}([\underline{\rho}]) - \{f_{\mathcal{C}_1}, h_{\mathcal{C}_2}\}([\rho])= \left( X_{f, \mathcal{C}_1}(v_2)\otimes X_{h, \mathcal{C}_2}(w_1), r^-\right) - \left(X_{f, \mathcal{C}_1}(v_1)\otimes X_{h, \mathcal{C}_2}(w_2), r^+\right) \\
 = \left( \rho(\mathcal{C}_1)^{-1}X_{f, \mathcal{C}_1}\otimes X_{h, \mathcal{C}_2}\rho(\mathcal{C}_2)^{-1}, r^-\right) - \left(X_{f, \mathcal{C}_1}\rho(\mathcal{C}_1)^{-1}\otimes \rho(\mathcal{C}_2)^{-1}X_{h, \mathcal{C}_2}, r^+\right) \\
  = \left( X_{f, \mathcal{C}_1} \rho(\mathcal{C}_1)\otimes \rho(\mathcal{C}_2)X_{h, \mathcal{C}_2}, r^-\right) - \left(\rho(\mathcal{C}_2)X_{f, \mathcal{C}_1}\otimes X_{h, \mathcal{C}_2}\rho(\mathcal{C}_1), r^+\right). 
 \end{multline*}
On the other hand, using that $\Delta^L_a(f_{\mathcal{C}_1})= \sum f' \otimes f''_{\mathcal{C}_1}$ and $\Delta^L_b(f_{\mathcal{C}_1})= \sum f'' \otimes f'_{\mathcal{C}}$ (with similar formulas for $h_{\mathcal{C}_2}$),
by Definition \ref{def_fusion}, one has: 
\begin{multline*}  \{f_{\mathcal{C}_1}, h_{\mathcal{C}_2}\}^{\circledast}([\rho]) - \{f_{\mathcal{C}_1}, h_{\mathcal{C}_2}\}([\rho])=
\left(\{ x, y \}^{\circledast} - \{x,y\}\right)([\rho]) \\
=  \left(\sum r^+ (y'_{(2)} \otimes x'_{(1)}) x_{(1)}''y_{(2)}'' - \sum r^+ (x'_{(2)}\otimes y'_{(1)})y''_{(1)}x''_{(2)}\right)(\rho) \\
= \left( \sum r^+(h'' \otimes f') f''_{\mathcal{C}_1}h''_{\mathcal{C}_2} - \sum r^+(f'' \otimes h')f'_{\mathcal{C}_1}h''_{\mathcal{C}_2} \right)(\rho) \\
= \left( \sum h'_{\mathcal{C}_2} r^{-}(f'\otimes h'') f''_{\mathcal{C}_1} - \sum f'_{\mathcal{C}_1} r^+(f'' \otimes h') h''_{\mathcal{C}_2} \right)(\rho) \\
= \left( X_{f, \mathcal{C}_1} \rho(\mathcal{C}_1)\otimes \rho(\mathcal{C}_2)X_{h, \mathcal{C}_2}, r^-\right) - \left(\rho(\mathcal{C}_2)X_{f, \mathcal{C}_1}\otimes X_{h, \mathcal{C}_2}\rho(\mathcal{C}_1), r^+\right).
\end{multline*}
We thus have proved that $\{f_{\iota(\mathcal{C}_1)}, h_{\iota(\mathcal{C}_2)} \}([\underline{\rho}]) = \{f_{\mathcal{C}_1}, h_{\mathcal{C}_2}\}^{\circledast}([\rho])$ so $\Psi$ is Poisson.

\end{proof}

\appendix

\section{Character varieties of graphs and proof of Proposition \ref{prop_holonomy_functions}}\label{appendix_graph}

In order to prove Proposition \ref{prop_holonomy_functions}, we introduce the notion of character varieties associated to a graph. 

\begin{definition}
A \textit{graph} $\Gamma$ is a $4$-tuple $(V(\Gamma), \mathcal{E}(\Gamma), t, \sigma)$, where $V(\Gamma)$ (the vertices) and $\mathcal{E}(\Gamma)$ (the oriented edges) are finite sets, $t:\mathcal{E}(\Gamma)\rightarrow V(\Gamma)$ is a surjective map and $\sigma: \mathcal{E}(\Gamma) \rightarrow \mathcal{E}(\Gamma)$ is a free involution. 
\end{definition}

We define the map $s:=t\circ \sigma$ and write $\alpha^{-1}$ the element $\sigma(\alpha)$. The geometric realization of $\Gamma$ is the CW-complex with set of $0$-cells $V(\Gamma)$ obtained by attaching of copy $I_{\alpha}$ of $[0,1]$ by gluing $\{1\}$ to $t(\alpha)$ and by identifying $I_{\alpha}$ with $I_{\alpha^{-1}}$ by the map sending $t$ to $1-t$. A graph is connected if its geometric realization is connected. The groupoid $\Pi_1(\Gamma)$ is the sub-category of the fundamental groupoid of the geometric realization of $\Gamma$ whose objects are the elements of $V(\Gamma)$ and morphisms are paths $\alpha$ which decompose as $\alpha=\alpha_1\dots \alpha_n$ where $\alpha_i \in \mathcal{E}(\Gamma)$. The set $\mathcal{E}(\Gamma)$ is naturally identified with a subset of the set of morphisms of $\Pi_1(\Gamma)$. We denote by $V^{\partial}(\Gamma) \subset V(\gamma)$ the sub-set of vertices with valence one and denote by $\mathring{V}(\Gamma)$ its complementary. An orientation of the edges of $\Gamma$ is a sub-set $\mathcal{E}_o(\Gamma)\subset \mathcal{E}(\Gamma)$ such that for each edge $e\in \mathcal{E}(\Gamma)$  the intersection $\mathcal{E}_o(\Gamma) \cap \{ e, e^{-1} \}$ contains exactly one element.

\begin{definition}\label{def_RepVarGraph}
The algebra $\mathbb{C}[\mathcal{R}_G(\Gamma)]$ is the quotient of the algebra $\mathbb{C}[G]^{\otimes \mathcal{E}}$ by the ideal generated by elements $x_e - S(x)_{e^{-1}}$ for $x\in \mathbb{C}[G]$ and $e\in \mathcal{E}(\Gamma)$. 
\end{definition}

Remark that if $\mathcal{E}_o$ is an orientation of $\Gamma$, there is a canonical isomorphism $\mathbb{C}[\mathcal{R}_G(\Gamma)] \cong \mathbb{C}[G]^{\otimes \mathcal{E}_o}$. Define the Hopf algebra $\mathbb{C}[\mathcal{G}_{\Gamma}]:= \mathbb{C}[G]^{\otimes \mathring{V}}$ and the co-module map $\Delta_{\Gamma} : \mathbb{C}[\mathcal{R}_G(\Gamma)] \rightarrow \mathbb{C}[\mathcal{G}_{\Gamma}]\otimes \mathbb{C}[\mathcal{R}_G(\Gamma)]$ by the formulas
$$\Delta_{\Gamma}(x_{\alpha}) := \left\{ 
\begin{array}{ll}
\sum x^{(1)}_{s(\alpha)}\cdot S(x^{(3)})_{t(\alpha)} \otimes x^{(2)}_{\alpha} & \mbox{, if }s(\alpha), t(\alpha) \in \mathring{V}; \\
\sum x^{(1)}_{s(\alpha)} \otimes x^{(2)}_{\alpha} & \mbox{, if }s(\alpha) \in \mathring{V}, t(\alpha) \in V^{\partial}; \\
\sum S(x^{(2)})_{t(\alpha)} \otimes x^{(1)}_{\alpha} & \mbox{, if } s(\alpha) \in V^{\partial} , t(\alpha) \in \mathring{V}; \\
1 \otimes x_{\alpha} &\mbox{, if } s(\alpha), t(\alpha) \in V^{\partial}. 
\end{array}\right. $$
\begin{definition}\label{def_CharVarGraph}
The algebra $\mathbb{C}[\mathcal{X}_G(\Gamma)]$ is the sub-algebra of co-invariant vectors of $\mathbb{C}[\mathcal{R}_G(\Gamma)] $, that is as the kernel of $\Delta_{\Gamma} - \eta\otimes \id$. The character variety $\mathcal{X}_G(\Gamma)$ is the maximal spectrum of $\mathbb{C}[\mathcal{X}_G(\Gamma)]$. 
\end{definition}
Remark that if $\mathbf{\Sigma}$ is a punctured surface with a finite presentation $\mathbb{P}$ without non-trivial relations, the character variety $\mathcal{X}_G(\mathbf{\Sigma}, \mathbb{P})$ is canonically isomorphic to the character variety of its associated presenting graph.

\vspace{2mm}
\par A \textit{curve} is an element $\mathcal{C}$ of $\Pi_1(\Gamma)$ such that either $s(\mathcal{C})=t(\mathcal{C})$ or $s(\mathcal{C}), t(\mathcal{C}) \in V^{\partial}$. Given a curve $\mathcal{C}$ which decomposes as $\mathcal{C}= \alpha_1 \ldots \alpha_n$ with $\alpha_i \in \mathcal{E}(\Gamma)$ and a regular function $f\in \mathbb{C}[G]$, which is further assumed to be $G$ invariant if $s(\mathcal{C})\neq t(\mathcal{C})$, we define the curve function $f_{\mathcal{C}}\in \mathbb{C}[\mathcal{X}_G(\Gamma)]$ as the class of the element $\sum (f^{(1)})_{\alpha_1} \ldots (f^{(n)})_{\alpha_n}$. Proposition \ref{prop_holonomy_functions} will be deduced from the following:

\begin{proposition}\label{prop_graph}
Let $\Gamma$ be a connected graph and $G$ a standard group. The following assertions hold:
\begin{enumerate}
\item The algebra $\mathbb{C}[\mathcal{X}_G(\Gamma)]$ is generated by its curve functions.
\item If $V^{\partial}(\Gamma)\neq \emptyset$, there exists an integer $d\geq 1$ such that $\mathbb{C}[\mathcal{X}_G(\Gamma)]\cong \mathbb{C}[G]^{\otimes d}$.
\item If $V^{\partial}(\Gamma)$ is empty, the character variety $\mathcal{X}_G(\Gamma)$ is isomorphic to the (Culler-Shalen) character variety of a free group $\mathbb{F}_m$ for some $m\geq 1$.
\end{enumerate}
\end{proposition}

\par When the set $\mathring{V}$ is empty, the above proposition is trivial. We first consider the case where $\mathring{V}$ has a single element. Denote by $\Gamma(n,m)$ the  graph defined by $\mathring{V}=\{v\}$, $V^{\partial}=\{ v_1, \ldots, v_n\}, \mathcal{E}=\{ \beta_1^{\pm 1}, \ldots, \beta_n^{\pm 1}, \gamma_1^{\pm 1}, \ldots, \gamma_m^{\pm 1} \}$, $\sigma(\beta_i)= \beta_i^{-1}, \sigma(\gamma_i)=\gamma_i^{-1}$ and $t(\beta_i)=v_i$, $s(\beta_i)= s(\gamma_j)=t(\gamma_j) = v$. 
\vspace{2mm}
\par Let $\mathbb{F}_m$ represents the free group generated by elements $\gamma_1, \ldots, \gamma_m$. By definition, the variety $\mathcal{X}_G(\Gamma(0, m))$ is canonically isomorphic to the character variety $\mathcal{X}_G(\mathbb{F}_m) :=  {\Hom (\mathbb{F}_m, G )} \sslash G $ and we called the group $G$ standard if the algebra of $\mathcal{X}_G(\mathbb{F}_m)$ is generated by curve functions for any $m\geq 1$.

\vspace{2mm}
\par Given $m\geq 0$ and $n\geq 1$, consider the subset $\mathcal{D}=\{ \alpha_2, \ldots, \alpha_n, \theta_1, \ldots, \theta_m \} \subset \Pi_1(\Gamma(n,m))$ of curves defined by $\alpha_i:= \beta_1^{-1}\beta_i$ and $\theta_j:= \beta_1^{-1}\gamma_j \beta_1$. Define a morphism $\phi : \mathbb{C}[G]^{\otimes \mathcal{D}} \rightarrow \mathbb{C}[\mathcal{X}_G(\Gamma(n,m))]$ by the formulas $\phi(x_{\alpha_i}):= \sum (S(x^{(1)}))_{\beta_1}x^{(2)}_{\beta_i}$ and $\phi(x_{\theta_j}):= \sum (S(x^{(1)})x^{(3)})_{\beta_1} x^{(2)}_{\gamma_j}$.

\begin{lemma}\label{lemma_appendix1}
The morphism $\phi : \mathbb{C}[G]^{\otimes \mathcal{D}} \rightarrow \mathbb{C}[\mathcal{X}_G(\Gamma(n,m)) ]$  is an isomorphism. In particular the graph $\Gamma(n,m)$ satisfies the conclusion of Proposition \ref{prop_graph}.
\end{lemma}

\begin{proof}
Define a morphism $\psi : \mathbb{C}[\mathcal{R}_G(\Gamma(n,m))] \rightarrow \mathbb{C}[G]^{\otimes \mathcal{D}}$ by the formulas $\psi(x_{\beta_1}):= \eta\circ \epsilon(x)$, $\psi(x_{\beta_i}):= x_{\alpha_i}$ for $2\leq i \leq n$ and $\psi(x_{\gamma_j}):= x_{\theta_j}$ for $1\leq j \leq m$. Define a map $f : \mathbb{C}[G] \otimes \mathbb{C}[\mathcal{R}_G(\Gamma(n,m))] \rightarrow \mathbb{C}[\mathcal{R}_G(\Gamma(n,m))]$ by the formula 
\begin{equation*} f\left( (x_0)_v \otimes (x_1)_{\beta_1} \ldots (x_n)_{\beta_n} (y_1)_{\gamma_1}\ldots (y_m)_{\gamma_m}\right) := \\
\left( S(x_0) x_1 \right)_{\beta_1} (x_2)_{\beta_2}  \ldots (x_n)_{\beta_n} (y_1)_{\gamma_1}\ldots (y_m)_{\gamma_m}.  \end{equation*}
\par Straightforward computations show that $\psi \circ \phi = \id$, hence $\phi$ is injective, and that $\phi \circ \psi= f \circ \Delta_{\Gamma(n,m)}$. If $X \in \mathbb{C}[\mathcal{X}_G(\Gamma(n,m))]$, we have $\Delta_{\Gamma(n,m)} (X) =1\otimes X$ by definition. We deduce from the equality $f\circ \Delta_{\Gamma(n,m)}(X) = f(1\otimes X) =X$ that $\phi(\psi(X))=X$, hence $X$ belongs to the image of $\phi$. This proves the surjectivity of $\phi$ and concludes the proof.
\end{proof}

\par The strategy to prove Proposition \ref{prop_graph} is to show that the character variety of any connected graph is isomorphic to the character variety of a graph $\Gamma(n,m)$, through an isomorphism that preserves the set of curve functions and the cardinal of $V^{\partial}$. Let $\Gamma=(V(\Gamma), \mathcal{E}(\Gamma), t, \sigma)$ be a connected graph such that $\mathring{V}(\Gamma)$ has at least cardinal two, and fix $v\in \mathring{V}(\Gamma)$. Fix an edge $\beta_1$ such that $s(\beta_1)=v$ and $t(\beta_1)\in \mathring{V}(\Gamma)\setminus \{v\}$.  We define a graph $\Gamma(v)$, obtained from $\Gamma$ by contracting the edge $\beta_1$,  as follows. Partition the set of edges of $\Gamma$ as $\mathcal{E}(\Gamma)= \mathcal{E}' \bigsqcup \mathcal{E}''$ where $\mathcal{E}'$ is the set of edges $e$ such that $v \in \{ s(e), t(e) \}$. Denote by $\{ \beta_1, \ldots, \beta_n \}$ the set of elements of $\mathcal{E}'$ such that $s(\beta_i)=v$ and $t(\beta_i)\neq v$. Since $\Gamma$ is connected and $\mathring{V}(\Gamma)$ has cardinal at least two, we can suppose that $t(\beta_1)\in \mathring{V}(\Gamma)$. Denote by $\{ \gamma_1^{\pm 1}, \ldots, \gamma_m^{\pm 1} \} $ the set of elements of $\mathcal{E}'$ such that $v=s(\gamma_j^{\pm 1})= t(\gamma_j^{\pm 1})$.

 \begin{definition}
  The graph $\Gamma(v)$ is defined by the combinatorial data $(V(\Gamma(v)), \mathcal{E}(\Gamma(v)), \widetilde{t}, \widetilde{\sigma})$ where $V(\Gamma(v)):= V(\Gamma)\setminus \{v\}$, $\mathcal{E}(\Gamma(v)) := \mathcal{E}'' \bigsqcup \{ \alpha_2^{\pm 1}, \ldots, \alpha_n^{\pm 1}, \theta_1^{\pm 1}, \ldots, \theta_m^{\pm 1} \}$. The restrictions of $t$ and $\widetilde{t}$ to $\mathcal{E}''$ coincide and we set $\widetilde{t} (\alpha_i):= t(\beta_i)$ and $\widetilde{t}(\alpha_i^{-1})=\widetilde{t}(\theta_j)= \widetilde{t}(\theta_j^{-1}) := t(\beta_1)$. The free involution $\widetilde{\sigma}$ coincides with $\sigma$ on $\mathcal{E}''$ and satisfies $\widetilde{\sigma}(\alpha_i):= \alpha_i^{-1}$ and $\widetilde{\sigma}(\theta_j):= \theta_j^{-1}$. 
 \end{definition}
 
 In short, the contracting operation sending $\Gamma$ to $\Gamma(v)$ consists in removing the sub-graph of $\Gamma$ adjacent to $v$, seen as an embedding of $\Gamma(n,m)$,  and replacing it by a graph whose edges are the paths $\alpha_i^{\pm 1}$ and $\gamma_j^{\pm 1}$ of Lemma \ref{lemma_appendix1}. Figure \ref{fig_graphs} illustrates this operation. 

\begin{figure}[!h] 
\centerline{\includegraphics[width=8cm]{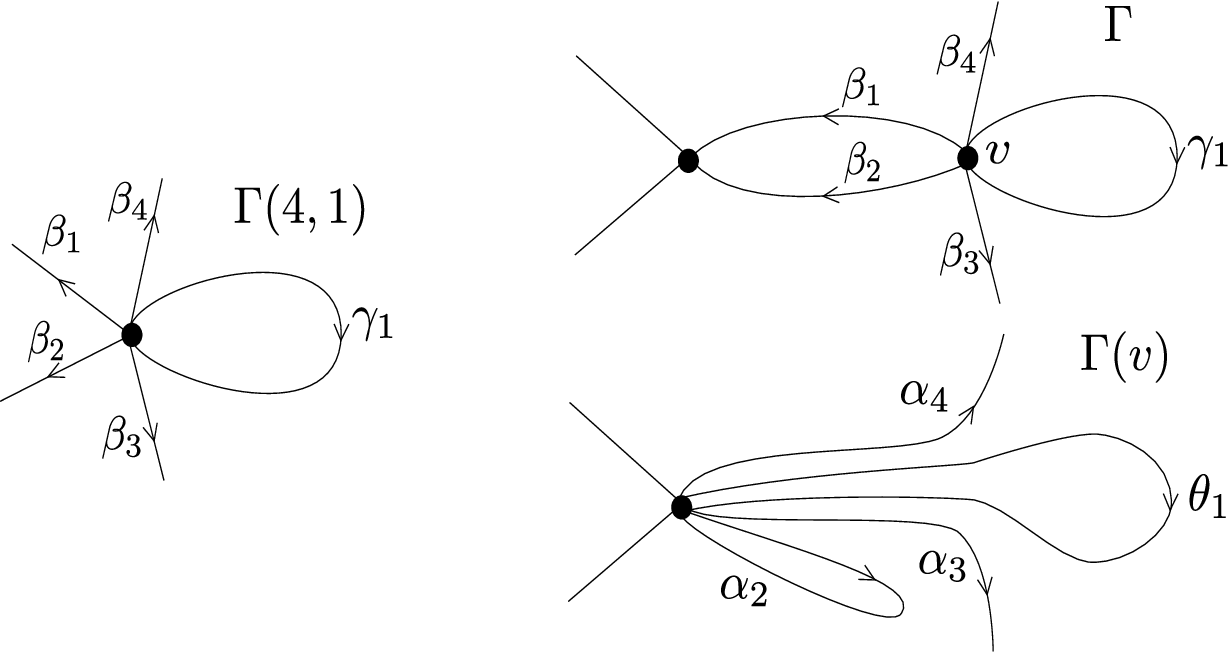} }
\caption{On the left, the graph $\Gamma(4,1)$. On the right, a graph $\Gamma$ and a contracted graph $\Gamma(v)$.} 
\label{fig_graphs} 
\end{figure}

\vspace{2mm}
\par Define a morphism $\Phi : \mathbb{C}[\mathcal{R}_G(\Gamma(v))] \rightarrow \mathbb{C}[\mathcal{R}_G(\Gamma)]$ by the formulas $\Phi(x_e)= x_e$ for $e\in \mathcal{E}''$, $\Phi(x_{\alpha_i}):= \sum (S(x^{(1)}))_{\beta_1}x^{(2)}_{\beta_i}$ and $\Phi(x_{\theta_j}):= \sum (S(x^{(1)}x^{(3)}))_{\beta_1} x^{(2)}_{\gamma_j}$.

\begin{lemma}\label{lemma_appendix2}
The morphism $\Phi$ induces an isomorphism $\mathbb{C}[\mathcal{X}_G(\Gamma(v))] \xrightarrow{\cong} \mathbb{C}[\mathcal{X}_G(\Gamma)]$ which preserves the sets of curve functions. 

\end{lemma}

\par We first introduce a notation. For $w \in \mathring{V}(\Gamma)$, define a Hopf co-module map $\Delta_w^{\Gamma} : \mathbb{C}[\mathcal{R}_G(\Gamma)] \rightarrow \mathbb{C}[G] \otimes \mathbb{C}[\mathcal{R}_G(\Gamma)]$ by the formulas
$$\Delta_w^{\Gamma}(x_{\alpha}) := \left\{ 
\begin{array}{ll}
\sum x^{(1)}_{s(\alpha)}\cdot S(x^{(3)})_{t(\alpha)} \otimes x^{(2)}_{\alpha} & \mbox{, if }s(\alpha)=t(\alpha) =w; \\
\sum x^{(1)}_{s(\alpha)} \otimes x^{(2)}_{\alpha} & \mbox{, if }s(\alpha) =w, t(\alpha) \neq w; \\
\sum S(x^{(2)})_{t(\alpha)} \otimes x^{(1)}_{\alpha} & \mbox{, if } s(\alpha) \neq w , t(\alpha) =w; \\
1 \otimes x_{\alpha} &\mbox{, if } s(\alpha), t(\alpha) \neq w. 
\end{array}\right. $$
By definition, the space of co-invariant vectors for $\Delta_{\Gamma}$ is the intersection over $w\in \mathring{V}$ of the spaces of co-invariant vectors for $\Delta_w^{\Gamma}$.

\begin{proof}
By Lemma \ref{lemma_appendix1}, the following sequence is exact
$$ 0 \rightarrow \mathbb{C}[G]^{\otimes \mathcal{D}} \xrightarrow{\phi} \mathbb{C}[\mathcal{R}_G(\Gamma(n,m))] \xrightarrow{\Delta_{\Gamma(n,m)} -\eta\otimes \id} \mathbb{C}[G] \otimes  \mathbb{C}[\mathcal{R}_G(\Gamma(n,m))]. $$
Fix a subset $\mathcal{E}'_o \subset \mathcal{E}'$ which intersects once each set $\{e, e^{-1} \}$ for any $e\in \mathcal{E}'$. We have natural isomorphisms $\varphi_1: \mathbb{C}[\mathcal{R}_G(\Gamma)] \cong \mathbb{C}[\mathcal{R}_G(\Gamma(n,m))] \otimes \mathbb{C}[G]^{\otimes \mathcal{E}'_o}$ and $\varphi_2 : \mathbb{C}[\mathcal{R}_G(\Gamma(v))] \cong  \mathbb{C}[G]^{\otimes \mathcal{D}} \otimes \mathbb{C}[G]^{\otimes \mathcal{E}'_o}$ making the following diagram commuting: 
$$
\begin{tikzcd}
0 \arrow[r,""] &\mathbb{C}[\mathcal{R}_G(\Gamma(v))] \arrow[r,"\Phi"] \arrow[d,"\cong", "\varphi_2"']& \mathbb{C}[\mathcal{R}_G(\Gamma)] \arrow[r, "\Delta_v^{\Gamma} - \eta\otimes \id"] \arrow[d,"\cong", "\varphi_1"'] & \mathbb{C}[G]\otimes \mathbb{C}[\mathcal{R}_G(\Gamma)] \arrow[d,"\cong", "\id \otimes \varphi_1"'] \\
0 \arrow[r,""] &  \mathbb{C}[G]^{\otimes \mathcal{D}} \otimes \mathbb{C}[G]^{\otimes \mathcal{E}'_o} \arrow[r,"\phi\otimes \id"] & \mathbb{C}[\mathcal{R}_G(\Gamma(n,m))] \otimes \mathbb{C}[G]^{\otimes \mathcal{E}'_o} \arrow[r,"(\Delta_{\Gamma(n,m)} -\eta\otimes \id) \otimes \id"] & \mathbb{C}[G] \otimes  \mathbb{C}[\mathcal{R}_G(\Gamma(n,m))] \otimes \mathbb{C}[G]^{\otimes \mathcal{E}'_o}
\end{tikzcd}
$$
The exactness of the second line implies the exactness of the first line, hence $\Phi$ sends injectively $\mathbb{C}[\mathcal{R}_G(\Gamma(v))]$ to the sub-algebra of co-invariant vectors of $\mathbb{C}[\mathcal{R}_G(\Gamma)]$ for the co-action $\Delta_v^{\Gamma}$. Moreover for any $w\in \mathring{V}(\Gamma)\setminus \{v\}$, the morphism $\phi$ intertwines the Hopf co-actions of $\Delta_w^{\Gamma(v)}$ and $\Delta_w^{\Gamma}$, hence  induces an isomorphism between the character varieties. The fact that $\phi$ sends curve functions to curve functions follows from the definitions.

\end{proof}

\begin{proof}[Proof of Proposition \ref{prop_graph}] Let $\Gamma$ be a connected graph and write $\mathring{V}(\Gamma) = \{v_1, \ldots, v_k\}$. If $k=0$ the proposition is trivial. If $k=1$, it follows from the fact that $G$ is standard and \ref{lemma_appendix1}. Suppose $k\geq 2$ and let $\Gamma':= \Gamma(v_2)(v_3)\ldots (v_k)$ be the graph obtained from $\Gamma$ by performing the contracting operation repeatedly on the vertices $v_2, \ldots, v_k$. By definition, $\mathring{V}(\Gamma')$ has  one element, hence $\Gamma'$ is isomorphic to a graph $\Gamma(n,m)$, and $V^{\partial}(\Gamma)$ has the same cardinal than $V^{\partial}(\Gamma')$. By Lemma \ref{lemma_appendix2}, there exists an isomorphism $\mathbb{C}[\mathcal{X}_G(\Gamma)] \cong \mathbb{C}[\mathcal{X}_G(\Gamma')] $ preserving the set of curve functions. We conclude using  \ref{lemma_appendix1}.
\end{proof}

\begin{proof}[Proof of Proposition \ref{prop_holonomy_functions}] Let $x=\sum_i (x_{i_1})_{\alpha_{i_1}} \ldots (x_{i_{k_i}})_{\alpha_{i_{k_i}}} \in \otimes^{\vee}_{\Pi_1(\Sigma)} \mathbb{C}[G]$ be an element such that its class $[x]$ belongs to $\mathbb{C}[\mathcal{X}_G(\mathbf{\Sigma})]$. Let $\mathcal{E}$ be the set of paths $\alpha_{i_j}^{\pm 1}$ appearing in the expression of $x$, together with their inverse. Let $V:= \{ s(\alpha), t(\alpha), \alpha \in \mathcal{E} \}$ be the set of endpoints and define $V^{\partial} := V \cap \mathcal{A}$ and $\mathring{V}:= V \cap (\Sigma \setminus \mathcal{A})$. Define the graph $\Gamma= (V, \mathcal{E}, t_{| \mathcal{E}}, \sigma)$ where $\sigma(\alpha_{i_j}):= \alpha_{i_j}^{-1}$. If necessary, we modify the polynomial expression of $x$ without changing the class $[x]$, such that the elements of $V^{\partial}$ have valency one and such that the elements of $\mathring{V}$ have valency bigger than one. There is a well defined morphism $\mathbb{C}[\mathcal{R}_G(\Gamma)] \rightarrow \mathbb{C}[\mathcal{R}_G(\mathbf{\Sigma})]$, sending a generator $x_{\alpha}$ to the generator denoted by the same symbol, which induces a morphism $\phi : \mathbb{C}[\mathcal{X}_G(\Gamma)] \rightarrow \mathbb{C}[\mathcal{X}_G(\mathbf{\Sigma})]$. By definition, the morphism $\phi$ sends curve functions to curve functions and $[x]$ belongs to its image. Hence, by Proposition \ref{prop_graph}, the element $[x]$ belongs to the algebra generated by the curve functions. This concludes the proof.

\end{proof}

\section{Comparison with Fock-Rosly constructions}\label{sec_FR}

The constructions of Fock and Rosly in \cite{FockRosly} are based on ciliated graphs. As we now explain, to a ciliated graph $(\Gamma,c)$ one can  associate a marked surface $\mathbf{\Sigma}^0$ together with a finite presentation $\mathbb{P}$ of its associated groupoid. 

\begin{definition}
\begin{enumerate}
\item A \textit{ribbon graph} $\Gamma$ is a finite graph together with the data, for each vertex, of a cyclic ordering of its adjacent half-edges. An \textit{orientation} for a ribbon graph is the choice of an orientation for each of its edges. 
\item  A \textit{ciliated ribbon graph} $(\Gamma, c)$ is a ribbon graph $\Gamma$ together with a lift, for each vertex, of the cyclic ordering of the adjacent half-edges, to a linear ordering. In pictures, if the half-edges adjacent to a vertex have the cyclic ordering $e_1<e_2<\ldots <e_n <e_1$ that we lift to the linear ordering $e_1<e_2<\ldots <e_n$, we draw a \textit{cilium} between $e_n$ and $e_1$. 
\item We associate surfaces to ribbon graphs as follows. 
\begin{itemize}
\item[(i)]
Place a  a disc $D_v$ on top of  each vertex $v$ and a band $B_e$ on top of each edge $e$,  then glue the discs to the band using the cyclic ordering: we thus get a surface $\Sigma(\Gamma)$ named the \textit{ fattening of }$\Gamma$. The \textit{unmarked surface associated to }$\Gamma$ is  $\mathbf{\Sigma}(\Gamma)=(\Sigma(\Gamma), \emptyset)$. 
 \item[(ii)]
The \textit{marked surface} $\mathbf{\Sigma}^0(\Gamma, c)=(\Sigma(\Gamma), \mathcal{A}(c))$ \textit{associated to }$(\Gamma,c)$ has the same underlying surface $\Sigma(\Gamma)$ and for each vertex $v$ adjacent to half-edges ordered as $e_1<e_2<\ldots <e_n$ place one boundary arc $a_v$ on the boundary of $D_v$ between $e_n$ and $e_1$ and set $\mathcal{A}(c)=\{a_v\}_{v\in V(\Gamma)}$. By isotoping each vertex $v$ of $\Gamma \subset \Sigma(\Gamma)$ to $a_v$, we get the generating graph of a set of generators $\mathbb{G}= \mathcal{E}(\Gamma)$ (the oriented edges) of $\Pi_1(\Sigma(\Gamma))$ relatively to $\mathcal{A}(c)$ such that $\mathbb{P}(\Gamma, c):=(V(\Gamma), \mathcal{E}(\Gamma), \emptyset)$ is a finite presentation of $\Pi_1(\Sigma(\Gamma))$ relatively to $\mathcal{A}(c)$ with no non-trivial relation.

\end{itemize}
\item Fix an arbitrary classical $r$-matrix $r(v)$ for each vertex $v\in V(\Gamma)$ and write $r(v)=\sum_{i,j} r^{i,j}(v)X_i \wedge X_j$ in some basis $(X_i)_i$ of $\mathfrak{g}$.
 For two oriented edges $\alpha, \beta \in \mathcal{E}(\Gamma)$ with common target endpoint $t(\alpha)=t(\beta)=v$, we write $\alpha <_c \beta$ if the target half edge of $\alpha$ is smaller than the target half edge of $\beta$ in the total ordering given by the cilium $c$. Consider the representation variety $\mathcal{R}_G(\Gamma)$ as defined in Definition \ref{def_RepVarGraph}. By definition, it is the (smooth) subvariety of $G^{\mathcal{E}(\Gamma)}$ of elements $g=(g_e)_{e\in \mathcal{E}(\Gamma)}$ such that $g_e^{-1}=g_{e^{-1}}$, so it is a smooth manifold as well that we denote by $\mathcal{M}_G(\Gamma)$. In \cite{FockRosly}, Fock and Rosly endowed the smooth manifold $\mathcal{M}_G(\Gamma)$ with a Poisson structure, depending on the cilium $c$, by defining a Poisson bivector field on $G^{\mathcal{E}(\Gamma)}$:
 $$ B := \sum_{v\in V(\Gamma)} \left( \sum_{t(\alpha)=t(\beta)=v, \alpha<_c \beta} r^{i,j}(v) X_i^{\alpha} \wedge X_j^{\beta} + \frac{1}{2}\sum_{t(\alpha)=v} r^{i,j}(v) X_i^{\alpha} \wedge X_j^{\alpha} \right).$$
 This defines a Poisson bracket $\{\cdot, \cdot\}^{FR}$ on $C^{\infty}(\mathcal{M}_G(\Gamma))$ which depends on the cilium and the $r$-matrices $r(v)$. 
 \item
 Now consider the discrete gauge group $\mathcal{G}:= G^{V(\Gamma)}$. It acts on $\mathcal{M}_G(\Gamma)$ by the classical formula
 \begin{equation}\label{eq_GaugeActionGraph}
  g\cdot \rho (\alpha) = g(s(\alpha)) \rho(\alpha) g(t(\alpha))^{-1} \quad, \mbox{ for all }g\in \mathcal{G}, \alpha\in \mathcal{E}(\Gamma), \rho \in \mathcal{X}_G(\Gamma), 
  \end{equation}
 and the quotient $\quotient{\mathcal{M}_G(\Gamma)}{\mathcal{G}}$ identifies with the (singular) moduli space $\mathcal{M}_G(\Sigma(\Gamma))$. Let $p: \mathcal{M}_G(\Gamma) \to \mathcal{M}_G(\Sigma(\Gamma))$ denote the projection map and $\mathcal{M}^0_G(\Sigma(\Gamma))$ the smooth locus and write
 $\mathcal{M}^0_G(\Gamma):= p^{-1}(\mathcal{M}^0_G(\Sigma(\Gamma)))$ and $p^* : C^{\infty}(\mathcal{M}^0_G(\Gamma)) \hookrightarrow C^{\infty}(\mathcal{M}_G(\Gamma) )$ the injective morphism induced by $p$. Fock and Rosly proved in \cite{FockRosly} that the image of $p^*$ is a Poisson subalgebra of $C^{\infty}(\mathcal{M}_G(\Gamma) )$, so $\mathcal{M}^0_G(\Sigma(\Gamma))$ inherits a structure of Poisson manifold. Moreover the authors proved that this Poisson structure is independent on the choice of the ciliated graph $(\Gamma,c)$ but only depends on $\Sigma(\Gamma)$.

\end{enumerate}
\end{definition}

Since $\Gamma$ is the generating graph of the presentation $\mathbb{P}(\Gamma)$ of $\Pi_1(\Sigma^0(\Gamma,c))$ which has no non-trivial relation, we have $\mathcal{X}_G(\mathbf{\Sigma}^0(\Gamma,c)) \cong \mathcal{X}_G(\mathbf{\Sigma}^0(\Gamma,c), \mathbb{P}(\Gamma)) \cong \mathcal{R}_G(\Gamma)$ so  $\mathbb{C}[\mathcal{X}_G(\mathbf{\Sigma}^0(\Gamma,c))]$ identifies with the subalgebra of $C^{\infty}(\mathcal{M}_G(\Gamma))$ of regular functions. We denote by $\iota : \mathbb{C}[\mathcal{X}_G(\mathbf{\Sigma}^0(\Gamma,c))] \hookrightarrow C^{\infty}(\mathcal{M}^0_G(\Gamma))$ the induced embedding.

The \textit{monogon} $\mathbb{D}$ is the marked surface made of a disc with one boundary arc. Its character variety has only one point and $\mathbb{C}[\mathcal{X}_G(\mathbb{D})]\cong\mathbb{C}$.  By gluing a monogon $\mathbb{D}_v$ to each boundary arc $a_v$ of the marked surface $\mathbf{\Sigma}^{0}(\Gamma, c)$, we obtain the unmarked surface $\mathbf{\Sigma}(\Gamma)$.  Still writing $\mathcal{G}= G^{V(\Gamma)}$, Proposition \ref{gluing_formula} implies that we have an exact sequence
$$ 0 \to \mathbb{C}[\mathcal{X}_G(\mathbf{\Sigma}(\Gamma))] \to \mathbb{C}[\mathcal{X}_G(\mathbf{\Sigma}^0(\Gamma, c))]\xrightarrow{\Delta^L -\sigma \circ \Delta^R} \mathbb{C}[\mathcal{G}] \otimes \mathbb{C}[\mathcal{X}_G(\mathbf{\Sigma}^0(\Gamma, c))].$$
The left comodule map $\Delta^L$ identifies with the left group action defined by Equation \eqref{eq_GaugeActionGraph} whereas, since the character variety of the monogon is trivial, the map $\sigma \circ \Delta^R$ identifies with the counit $\epsilon \times \id$ so 
$$ \mathcal{X}_G(\mathbf{\Sigma}(\Gamma)) = \mathcal{X}_G(\mathbf{\Sigma}^0(\Gamma,c)) \sslash \mathcal{G}.$$
We denote by $j : \mathbb{C}[\mathcal{X}_G(\mathbf{\Sigma})] \hookrightarrow C^{\infty}(\mathcal{M}^0_G(\Sigma(\Gamma)))$ the inclusion morphism.

\begin{proposition}\label{prop_FR}
Let $(\Gamma, c)$ be a ciliated graph. For each vertex $v\in V(\Gamma)$, choose an orientation $\mathfrak{o}(v)$ of the corresponding boundary arc $a_v$ of $\mathbf{\Sigma}^0(\Gamma,c)$ and consider the Fock-Rosly Poisson structures on $\mathcal{M}_G(\Gamma)$ and $\mathcal{M}^0_G(\Sigma(\Gamma,c))$ induced by the classical r-matrices $r(v):= r^{\mathfrak{o}(a_v)}$. Then both morphisms 
$\iota : \mathbb{C}[\mathcal{X}_G(\mathbf{\Sigma}^0(\Gamma,c))] \hookrightarrow C^{\infty}(\mathcal{M}^0_G(\Gamma))$ and $j : \mathbb{C}[\mathcal{X}_G(\mathbf{\Sigma})] \hookrightarrow C^{\infty}(\mathcal{M}^0_G(\Sigma(\Gamma)))$ are Poisson.
\end{proposition}

\begin{proof}
We first prove that $\iota$ is Poisson. As in Remarks \ref{remark_PoissonBigon} and \ref{remark_PoissonTriangle}, we consider an embedding $G\subset \GL_N(\mathbb{C})$. For $\alpha \in \mathcal{E}(\Gamma)$, we consider the $N\times N$ matrix $M(\alpha)$ with coefficients in  $\mathbb{C}[\mathcal{X}_G(\mathbf{\Sigma}^0(\Gamma,c))]$ whose $(i,j)$ entry is the regular function sending a representation $\rho$ to the $(i,j)$ entry of $\rho(\alpha)$. 
Consider $\alpha, \beta \in \mathcal{E}(\Gamma)$ two generating paths. Replacing $\alpha$ and or $\beta$ by $\alpha^{-1}, \beta^{-1}$ if necessary, we have $10$ possible configuration for the pair $(\alpha, \beta)$ illustrated in Figure \ref{fig_configurations}, depending on which pairs of element of $\{ s(\alpha), t(\alpha), s(\beta), t(\beta) \}$ are equal or not. For each configuration, we need to show that the formula for $\{ M(\alpha) \otimes M(\beta) \}$ obtained by the Fock-Rosly Poisson bracket is the same as the one obtained by the generalized Goldman formula. For instance, in case $(i)$ where  $\{ s(\alpha), t(\alpha), s(\beta), t(\beta) \}$ has cardinal four, we obtain $\{ M(\alpha) \otimes M(\beta) \} = 0$ for both Poisson structures. In case $(ii)$ where $t(\alpha)=t(\beta)$ with $\alpha <_c \beta$ and $\{ s(\alpha), t(\alpha), s(\beta), t(\beta) \}$ has cardinal $3$, we obtain 
$$ \{M(\alpha) \otimes M(\beta) \} = (M(\alpha)\otimes M(\beta)) r(t(\alpha)) $$
 in both cases. In case $(iii)$, where $s(\alpha)=s(\beta)=v_1$ and $t(\alpha)=t(\beta)=v_2\neq v_1$ and $\alpha >_c \beta$ at $v_1$ and $\alpha<_c \beta$ at $v_2$ (this case covers also the case where $\alpha = \beta$), one finds
 $$ \{M(\alpha)\otimes M(\beta) \} = (M(\alpha) \otimes M(\beta))r(v_1) - r(v_2) (M(\alpha)\otimes M(\beta))$$
 for both Poisson brackets.
 The remaining $7$ cases are done by a similar case-by-case analysis left to the reader.
 
\begin{figure}[!h] 
\centerline{\includegraphics[width=10cm]{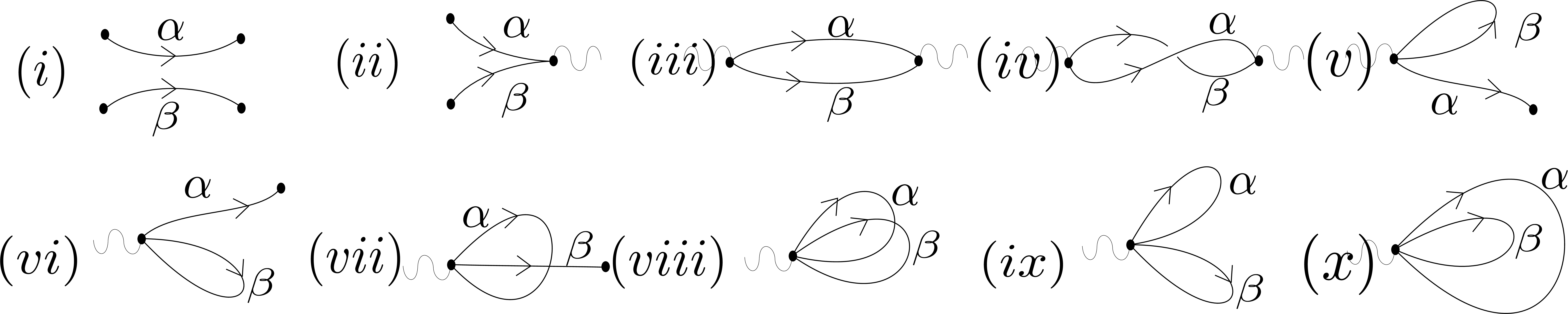} }
\caption{Ten different configurations for a pair of oriented edges in some ciliated graph.} 
\label{fig_configurations} 
\end{figure} 
 
 To prove that $j$ is Poisson, we simply remark that it fits in the commutative diagram
 
 $$ \begin{tikzcd}
\mathbb{C}[\mathcal{X}_G(\mathbf{\Sigma}(\Gamma))] \arrow[r, hook,  "i"] \arrow[d, hook, "j"] &  \mathbb{C}[\mathcal{X}_G(\mathbf{\Sigma}^0(\Gamma, c))] \arrow[d, hook, "\iota"] \\C^{\infty}(\mathcal{M}^0_G(\Sigma(\Gamma))) \arrow[r, hook, "p^*"] & C^{\infty}(\mathcal{M}^0_G(\Gamma))
\end{tikzcd} $$

Since the maps $i, \iota$ and $p^*$ are Poisson, so is $j$.

\end{proof}

Proposition \ref{prop_FR} shows that our definition of relative character varieties is essentially a reformulation of the construction of Fock-Rosly. Let us stress some advantages of our construction.

\begin{enumerate}
\item The Fock-Rosly Poisson varieties are essentially the same than our discrete models $\mathcal{X}_G(\mathbf{\Sigma}^0(\Gamma, c), \mathbb{P})$. The main novelty of our approach is the fact that we also consider some continuous model $\mathcal{X}_G(\mathbf{\Sigma})$, which are independent on the choice of finite presentation of the fundamental groupoid, or equivalently on the choice of ciliated graph. In \cite[Proposition $4$]{FockRosly}, in order to prove that the Poisson structure on $\mathcal{X}_G(\mathbf{\Sigma}(\Gamma))$ does only depend on the surface $\Sigma(\Gamma)$ and not on the ciliated graph $(\Gamma,c)$, the authors need to prove the invariance of the Poisson bracket on a set of elementary moves on ciliated graphs that preserve the underlying surface (the annoying proof is actually left to the reader). In our approach, this is done by identifying  each discrete model with a canonical continuous model which only depends on the marked surface by definition.
\item Consider a compact oriented connected Riemann surface $\Sigma$ with non trivial boundary and equip the Banach space $\Omega^1_F(\Sigma)$ of flat connections (with suitable Sobolev regularity) with the Atiyah-Bott Poisson structure defined for two smooth functions $F,G$ on $\Omega^1_F(\Sigma)$, and seeing the differentials $D_AF$ and $D_AG$ as elements of $\mathrm{Z}^1_A(\Sigma, \mathfrak{g})\cong T_{A}\Omega^1_F(\Sigma) $ (as explained in the introduction), by the formula
$$ \{ F,G\}(A) = \int_{\Sigma} \left( D_AF \wedge D_AG\right) .$$
Restricting this Poisson bracket to functions invariant under the gauge group $\mathcal{G}=\{g: \Sigma \to G\}$, we get a Poisson structure on the smooth locus of $\mathcal{M}_G(\Sigma)=\quotient{ \Omega^1_F(\Sigma) }{\mathcal{G}}$. In \cite[Proposition $5$]{FockRosly} Fock and Rosly proved that this Poisson structure coincides, through the holonomy map, with the one they defined on $\mathcal{M}^0_G(\Sigma(\Gamma))$ using a ciliated graph whose thickening is $\Sigma$  (this does not give an alternative proof of its independence with respect to the choice of $(\Gamma,c)$ since the proof strongly uses this latter fact).

 It was commonly admitted in the community that Goldman's arguments in \cite{Goldman_Symplectic} generalize to non closed surfaces in order to prove that the above bracket on the moduli space of flat connections taken on curve functions is given by the same expression than Goldman's formula in \cite{Goldman86}.  This fact is proved by Roche-Szenes in \cite{RocheSzenes} and also, apparently independently,  by Lawton
  in  \cite[Theorem $15$]{Lawton_PoissonGeomSL3} when  $G=\SL_N$ though the author explained in Comment $18$ how to generalize his proof for general $G$.  Together with \cite[Proposition $5$]{FockRosly}, this gives a gauge theoretic proof that the Fock-Rosly Poisson bracket is described by Goldman's formula for standard groups $G$.

 In this paper, putting together Proposition \ref{prop_FR} and Theorem \ref{goldman_formula}, we obtain an alternative algebraic proof of this fact, which does not rely on gauge theory.

\end{enumerate}

\begin{remark}The Fock-Rosly Poisson varieties admit quantization deformations named \textit{quantum moduli algebras} defined independently by Alekseev-Grosse-Schomerus in \cite{AlekseevGrosseSchomerus_LatticeCS1,AlekseevGrosseSchomerus_LatticeCS2, AlekseevSchomerus_RepCS} and Buffenoir-Roche in \cite{BuffenoirRoche, BuffenoirRoche2} based on \cite{FockRosly} and indexed by a ciliated graph. The relative character varieties defined in the present paper are designed to admit the Bonahon-Wong-L\^e stated skein algebras (indexed by marked surfaces) as deformation quantizations (see \cite{KojuQuesneyClassicalShadows}). Since we proved that the relative character varieties are isomorphic to the Fock-Rosly moduli space, it is natural to expect that stated skein algebras are isomorphic to the quantum moduli spaces using the same correspondence ciliated graphs vs marked surfaces with finite presentations. This was proved in the particular case of marked surfaces with exactly one boundary arc by Faitg in \cite{Faitg_LGFT_SSkein} and can be alternatively and independently derived from the works of Ben-Zvi, Brochier, Jordan \cite{BenzviBrochierJordan_FactAlg1} and  Gunningham, Jordan, Safronov \cite{GunninghamJordanSafranov_FinitenessConjecture} (see the end of \cite{KojuPresentationSSkein}). The general case was proved in \cite{KojuPresentationSSkein}. The quantum moduli spaces can be seen as discrete models for the stated skein algebras in the same way that for relative character varieties.

\end{remark}

\section{Comparison with the constructions of Alekseev-Kosmann-Malkin-Meinreken}\label{sec_Alekseev}

We now compare \stated character varieties with the moduli spaces $\mathcal{M}_{g,n}$ which appear in \cite{AlekseevMalkin_PoissonCharVar, AlekseevMalkin_PoissonLie, AlekseevKosmannMeinrenken}. For $\mathbf{\Sigma}$ and $\mathbf{\Sigma}'$ two marked surfaces, each having exactly one boundary arc, say $a$ and $a'$, we denote by $\mathbf{\Sigma}\circledast \mathbf{\Sigma}'$ the marked surface obtained from $\mathbf{\Sigma}\sqcup \mathbf{\Sigma}'$ by fusioning $a$ and $a'$. 

Let $\mathbf{\Sigma}_{g,n}^*=(\Sigma_{g,n+1}, \{a\})$ be a genus $g$ surface with $n+1$ boundary components and a single boundary arc $a$. Then $\mathbf{\Sigma}_{g,n}^* \circledast \mathbf{\Sigma}_{g',n'}^* \cong \mathbf{\Sigma}_{g+g', n+n'}^*$, so 
$$ \mathbf{\Sigma}_{g,n}^*\cong (\mathbf{\Sigma}_{1,0}^*)^{\circledast g} \circledast (\mathbf{\Sigma}_{0,1}^*)^{\circledast n}.$$
Theorem \ref{theorem_fusion} implies 
$$
\mathcal{X}_G(\mathbf{\Sigma}_{g,n}^*) \cong \mathcal{X}_G(\mathbf{\Sigma}_{1,0}^*)^{\circledast g} \circledast \mathcal{X}_G(\mathbf{\Sigma}_{0,1}^*)^{\circledast n}.
$$
Here the Poisson bracket is chosen by the orientation $\mathfrak{o}$ such that $\mathfrak{o}(a)=+$. Note that $\mathbf{\Sigma}_{0,1}^*$ is obtained from the bigon $\mathbb{B}$ by fusioning its two boundary arcs together. So, as a variety, $\mathcal{X}_G(\mathbf{\Sigma}_{0,1}^*)=G$ and the Poisson bracket $\{\cdot, \cdot \}^{STS}$ is given (using the generalized Goldman formula), in matrix notations, by 
$$ \{N\otimes N\}^{STS} = - (\mathds{1}_1 \odot N)r^+ (N\odot \mathds{1}_2) +\tau (N\odot N) \tau r^+ - r^- (N\odot N) + (N\odot \mathds{1}_2)r^- (\mathds{1}_2 \odot N),$$
where, as before, we use an embedding $G\subset \GL_n$ and denote by $N$ the $n\times n$ matrix whose $(i,j)$ entry is the function $x_{i,j}: G \to \mathbb{C}$ sending the $i,j$ matrix coefficient of an element $g\in G \subset \GL_n$. 
The Poisson variety $G^{STS}= (G, \{\cdot, \cdot\}^{STS}) \cong \mathcal{X}_G(\mathbf{\Sigma}_{0,1}^*)$ was studied in great details in \cite[Section $4$]{AlekseevMalkin_PoissonCharVar} where its symplectic leaves were computed (they are the intersection of the conjugacy classes of $G$ with the so-called dressing orbits). 
\par Let $\mathbb{D}_1^+$ be an annulus with one boundary arc in each of its boundary component. So $\mathbb{D}_1^+$ has two boundary arcs an by fusioning these two boundary arcs, we get $\mathbf{\Sigma}_{1,0}^*$. 
Let $\alpha$, $\beta$ be the two arcs in $\mathbb{D}_1^+$ of Figure \ref{fig_D1+}, so $\mathbb{G}={\alpha^{\pm 1}, \beta^{\pm 1}}$ forms a generating set of $\Pi_1(\mathbb{D}_1^+)$ with no non trivial relations.

\begin{figure}[!h] 
\centerline{\includegraphics[width=3cm]{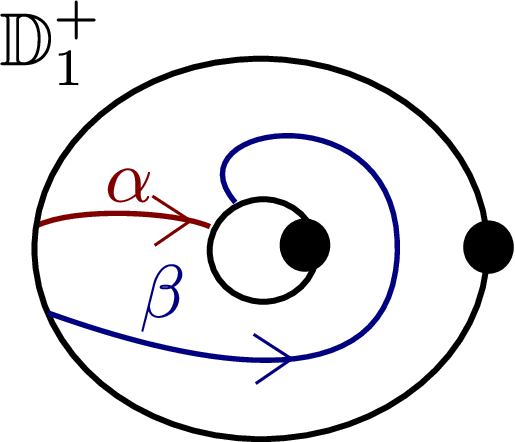} }
\caption{Two arcs in $\mathbb{D}_1^+$.} 
\label{fig_D1+} 
\end{figure}

 So $\mathcal{X}_G(\mathbb{D}_1^+)\cong G\times G$ through the map sending $\rho$ to $(\rho(\alpha), \rho(\beta))$ and, by the generalized Goldman formula, the Poisson bracket is given in matrix notations by 
$$ \{N(\alpha) \otimes N(\beta) \}^+ = r_+ (N(\alpha)\odot N(\beta)) + (N(\alpha)\odot N(\beta))r_+, \quad \{N(\delta)\otimes N(\delta)\}^+ = r^+ (N(\delta) \odot N(\delta)) - (N(\delta)\odot N(\delta))r^+,$$
for $\delta=\alpha, \beta$.
The Poisson variety $D_+(G):=(G\times G, \{\cdot, \cdot\}^+)$ was studied by Alekseev-Malkin in  \cite{AlekseevMalkin_PoissonLie} inspired by the work of Semenov-Tian-Shansky, and is called the \textit{twisted Heisenberg double}. In particular the authors computed its symplectic leaves. More precisely, they consider the bracket $\{\cdot, \cdot\}^{AM}:= - \{\cdot, \cdot\}^+$ for which, using the notations $r:= r^+$ and $r^*:= -r^-$, one has (compare with \cite[Equation $(80$)]{AlekseevMalkin_PoissonLie}):
$$ \{N(\alpha)\otimes N(\beta) \}^{AM}= - \left( r ( N(\alpha) \odot N(\beta)) + (N(\alpha)\odot N(\beta))r^* \right).$$
Therefore $\mathcal{X}_G(\mathbf{\Sigma}_{1,0}^*)$ is isomorphic to the fusion $(D_+G)_{1 \circledast 2}$ of $D_+G$ (seen as a $G\times G$ variety). We thus have proved that 
$$\mathcal{X}_G(\mathbf{\Sigma}_{g,n}^*) \cong ( (D_+G)_{1 \circledast 2})^{\circledast g} \circledast (G^{STS})^{\circledast n}.$$
This is precisely the moduli space studied in \cite{AlekseevKosmannMeinrenken}. In particular, we have proved that this moduli space is a particular case of Fock-Rosly moduli spaces; this fact is part of the folklore on the subject though, at the author's knowledge, no proof had been written yet. 

\section{Comparison with the constructions of Lie Bland-Severa and Nie}\label{sec_QPoisson}

Recall from \cite{AlekseevKosmannMeinrenken} that a \textit{quasi Poisson manifold} $(X,P,\rho)$  is a manifold $X$ equipped with a Lie group action $\mathcal{G} \curvearrowright X$ where the Lie algebra $\mathfrak{g}=Lie(\mathcal{G})$ is equipped with an invariant pairing,  and a bivector field $P$ satisfying the quasi-Poisson condition $[P, P ]= \nu(\Phi)$ where $\nu$ is the infinitesimal action of $\mathfrak{g}$ and $\Phi\in \Lambda^3\mathfrak{g}$ is $\Phi= \frac{1}{12}\sum_{ijk} (e_i, [e_j, e_k]) e_i \wedge e_j \wedge e_k$ for a basis $(e_i)_i$ of $\mathfrak{g}$. Let $r\in \mathfrak{g}^{\otimes 2}$ be a classical $r$-matrix whose symmetric part is the dual of the invariant pairing of $\mathfrak{g}$ and denote by $\overline{r}$ its skew-symmetric part. It is proved in \cite[Theorem $7.1$]{AlekseevKosmannMeinrenken} that the bivector field $\pi = P + \nu(\overline{r})$ satisfies $[\pi, \pi]=0$ so defines a structure of $G$-Poisson variety on $X$ which we call a \textit{twist} of the quasi Poisson structure $(X,P,\rho)$. 
\vspace{2mm}
\par Let $\mathbf{\Sigma}=(\Sigma, \mathcal{A})$ be a connected marked surface with $\mathcal{A}\neq \emptyset$. Li Bland-Severa \cite{LiBlandSevera_ModuliSpacesQuiltedSurfaces} and Nie  \cite{Nie_QPoissonGoldmanFormula} independently equipped the $\mathcal{G}= G^{\mathcal{A}}$ manifold $\mathcal{X}_G(\mathbf{\Sigma})$ with a quasi-Poisson structure $P^{LBSN}$. Their construction generalizes  the Massuyeau-Turaev construction in \cite{MassuyeauTuraev_QuasiPoisson} when $\mathbf{\Sigma}$ has a single boundary arc.
It is related to our construction as follows.

\begin{theorem}\label{theorem_LBSN} The Poisson variety $(\mathcal{X}_G(\mathbf{\Sigma}), \{\cdot, \cdot\}^{\mathfrak{o}})$ is a twist of the Li Bland-Severa-Nie quasi Poisson structure in the sense that if $\pi^{\mathfrak{o}}$ denotes the bivector field associated to $\{\cdot, \cdot\}^{\mathfrak{o}}$, then 
$$ \pi^{\mathfrak{o}} = P^{LBSN} + \nu\left( \sum_{a \in \mathcal{A}} \overline{r}^{\mathfrak{o}(a)}\right).$$
\end{theorem}

In particular, this re-proves the well-known fact (see e.g. \cite{Mouquin})  that the Fock-Rosly Poisson structure is a twist of the Li Bland-Severa-Nie quasi Poisson structure. In the particular case where $\mathbf{\Sigma}=\mathbb{B}$ and $\mathfrak{o}=(-,+)$, then $\mathcal{X}_G(\mathbb{B})$ is $G$ equipped with its Poisson-Lie structure and $G$ action given by conjugacy. The bivector field $P^{LBSN}$  in this case vanishes so we recover the observation, made by Drinfeld, that the Lie-Poisson structure of $G$ is a twist of the $0$ quasi Poisson structure, i.e. that at the infinitesimal level, that the skew-symmetric part $\overline{r}$ can be thought as a twist between the Lie bialgebra defining the Lie group $G$ and a quasi Lie bialgebra with vanishing cobracket (see \cite[Section $2.2$]{ChariPressley} for details). This observation is at the very origin of the notion of quasi-Poisson manifolds.

\begin{proof} Let $\{\cdot, \cdot\}^{LBSN}$ denote the bracket defined by $\{f, g\}^{LBSN} ([\rho]) := \left< D_{[\rho]}f \otimes D_{[\rho]}g, P^{LBSN}_{[\rho]}\right>$ (it does not satisfies Jacobi). An explicit formula for the bracket $\{f_{\mathcal{C}_1}, h_{\mathcal{C}_2}\}^{LBSN}$ was computed in \cite[Theorem $3$]{LiBlandSevera_ModuliSpacesQuiltedSurfaces} and \cite[Theorem $2.5$]{Nie_QPoissonGoldmanFormula}. This formula is very similar to the generalized Goldman formula we found in Theorem \ref{goldman_formula} with one difference: the $r$ matrices which appear in the sums are replaced by their symmetric part. So 

\begin{multline*} \{f_{\mathcal{C}_1}, h_{\mathcal{C}_2}\}^{\mathfrak{o}} ([\rho]) - \{f_{\mathcal{C}_1}, h_{\mathcal{C}_2}\}^{LBSN} ([\rho]) = \sum_{a\in \mathcal{A}} \sum_{(v_1, v_2) \in S(a)} \left<X_{f, \mathcal{C}_1}(v_1)\otimes X_{h, \mathcal{C}_2}(v_2), \overline{r}^{\mathfrak{o}(v_1, v_2)}\right> \\
= \left< D_{[\rho]}f_{\mathcal{C}_1} \otimes D_{[\rho]}h_{\mathcal{C}_2}, \nu\left(\sum_a \overline{r}^{\mathfrak{o}(a)}\right)\right>.
\end{multline*}
This proves the equality $\pi^{\mathfrak{o}} - P^{LBSN} = \nu\left( \sum_{a \in \mathcal{A}} \overline{r}^{\mathfrak{o}(a)}\right)$ as required.

\end{proof}

\bibliographystyle{amsalpha}
\bibliography{biblio}

\end{document}